\documentclass[10pt, a4paper, reqno]{amsart}      
\usepackage{amsfonts,amssymb,amsmath,epsfig,graphics,graphicx,rotating}
\usepackage{latexsym}
\usepackage{amsbsy}
\usepackage{a4wide}
\usepackage[utf8]{inputenc}
\usepackage[pdfpagelabels=true, pdfstartview=FitH]{hyperref}
\def\N{\mathbb N}
\def\Z{\mathbb Z}
\def\Zd{{\mathbb Z}^d}
\def\R{\mathbb R}

\def\sigb{\mathbf{\Sigma}}
\def\X{\mathrm X}
\def\1{\mathrm{I}}
\def\d{\mathrm d}

\def\L{\mathcal L}

\def\XX{\mathfrak X}
\def\a{\alpha}
\def\g{\gamma}
\def\b{\beta}
\def\gg{G_{\alpha,\beta,\gamma,\delta}^{k,l}}

\def\vfi{\varphi}
\def\etb{{\mathcal{S}}^k_{x,y}\eta}  
\def\zetb{\mathcal{S}^l_{x,y}\zeta}

\def\gen{g}
\def\gkab{\gen_{\alpha,\beta}^{k}}

\def\gka{\gen_{\alpha}^{k}}

\def\gkb{\gen_{\ast,\beta}^{k}}

\def\GMx{mass migration\ }
\def\Lip{\mathrm{\bf L}}

\def\mp{\bar{\mu}_{\varphi}}
\newtheorem{theorem}{Theorem}[section]
\newtheorem{lemma}[theorem]{Lemma}
\newtheorem{proposition}[theorem]{Proposition}
\newtheorem{corollary}[theorem]{Corollary}
\newtheorem{remark}[theorem]{Remark}
\newtheorem{definition}[theorem]{Definition}
\newtheorem{example}{Example}

\allowdisplaybreaks

\title[Mass Migration Processes]{Invariant measures of Mass Migration Processes}   
\author{L.Fajfrov\'a, T.Gobron, E.Saada}
 \address{Lucie Fajfrov\'a, 
 Institute of Information Theory and Automation -- Academy of Sciences of the Czech Republic, Pod
 Vod\'arenskou v\v{e}\v{z}\'{\i}~4, 182\,08 Praha 8. Czech Republic.} \email{fajfrova@utia.cas.cz}
 \address{Thierry Gobron, CNRS, UMR 8089, LPTM\\ Universit\'e de Cergy-Pontoise; 
 Site de Saint-Martin 2, 2 avenue Adolphe Chauvin, Pontoise
 \\ 95031 Cergy-Pontoise cedex, France} \email{Thierry.Gobron@u-cergy.fr}
 \address{Ellen Saada, CNRS, UMR  8145, Laboratoire MAP5, Universit\'e Paris Descartes, 
 Sorbonne Paris Cit\'e, 45 rue des Saints-P\`eres, 75270 Paris cedex 06. France.}
 \email{Ellen.Saada@mi.parisdescartes.fr}
\date{\today}

\begin{document}
\pagestyle{myheadings}
\begin{abstract}
 We introduce the Mass Migration Process (MMP), 
a conservative particle system on $\N^{\Zd}$. It consists in jumps of $k$ particles ($k\ge 1$)
between sites, with a jump rate depending only on the state of the system at the departure 
and arrival sites of the jump. It generalizes misanthropes processes, hence 
in particular zero range and target processes. After the construction of MMP,
our main focus is on its invariant measures. We obtain necessary and sufficient
conditions for the existence of translation-invariant and invariant product probability measures. 
In the particular cases of asymmetric mass migration zero range and mass migration 
target dynamics, these conditions yield explicit solutions.  
If these processes are moreover attractive, we obtain a full characterization 
of all translation-invariant, invariant probability  measures. 
We also consider attractiveness properties (through couplings) and 
condensation phenomena for MMP. We illustrate our results on many  examples; 
 in particular, we prove the  coexistence of both condensation 
 and attractiveness in one of them. 
\end{abstract}
\maketitle

\section{Introduction}         \label{s-intro}

In the study of an interacting particle system, an essential tool is the explicit knowledge
of an invariant measure. Conservative systems such as exclusion or zero-range processes
possess a one-parameter family of translation invariant and invariant product
probability measures, where the parameter represents the average particle density per site 
\cite{spitzer,Lig,andjel}. Under conditions on the rates, this is also the case
for misanthropes processes, which include the more recently studied target
processes \cite{Cocozza,godreche07,LG07}. 
All these  dynamics 
(we call them \emph{single-jump models})  consist in  individual jumps of particles between sites, 
with rates which are the product of two terms: 
 a transition probability giving the direction of the jump, 
 and a function depending on the occupation numbers at the departure and/or arrival
sites of the jump.  Each dynamics has its particular features, which dictate
the precise form of the rates. 

\medskip

However, finding invariant measures is challenging
as soon as one departs from these classical models. 
In this paper we address this question for a class of models
generalizing them, that we call \emph{mass migration processes} 
(abbreviated as MMP). Those dynamics, of state space $\N^{\Z^d}$,
 allow multiple (simultaneous) jumps of particles between sites, 
according to rates  written as above, but
where the function depends also on the number
$k\ge 1$ of particles which jump.  We distinguish among them
the processes for which the rates do not depend  on the 
occupation number of the arrival site and call them 
\emph{mass migration zero range processes}  (abbreviated as MM-ZRP),
and the processes for which the rates do not depend  on the 
occupation number of the departure site (as long as it is non-empty) 
and call them \emph{mass migration target processes}  (abbreviated as MM-TP).

\medskip
 
Examples of such dynamics have appeared in the literature in various contexts
and under various names. Without being exhaustive, we quote
some of them that will illustrate our results:
The totally asymmetric \emph{stick process} has a constant jump rate (independent of 
particles' numbers); it was studied in \cite{Sepa1} in the context of Ulam's problem, 
and a generalization with nearest neighbour jumps
was considered in  \cite{GS}; it possesses a one parameter family of
product geometric invariant measures.  
The MMP is a particular case of
the multiple particle jump model studied in \cite{GS} in the context of attractiveness
properties; under some conditions on the rates, the latter possesses a one parameter family
of translation invariant and invariant measures,
that are not always explicit.
 Dynamics with zero range interaction
and multiple jumps 
were studied in a finite volume setup
in the context of condensation properties: they were called 
\textit{mass transport models} in \cite{EH,EMZ1,EMZ2}, and 
 \textit{generalized zero range processes}
in \cite{GrL} (where the number $k$ of particles that may
jump together was bounded). In those references, the set of sites is 
an interval, a general finite graph or a cube; 
in \cite{EMZ1,EMZ2}, the authors consider a continuous mass instead of particles 
moving among sites. A generic form of stationary product measures is exhibited
for all those dynamics; such an explicit form is required to analyse condensation phenomena, this is why these models
are often chosen as examples (see the reviews \cite{evans0,EH,CG}). 
In the context of exactly solvable models and duality, the  $q$-Hahn asymmetric 
zero range processes are  other  dynamics with zero range interaction
and multiple jumps (see \cite{BC} and references therein) with explicit product invariant measures;
this knowledge is crucial for exact-solvability.  Another possible extension of these models 
are processes in which multiple births and deaths are superimposed to multiple jumps; their attractiveness properties 
have been established in \cite{bor1}, and  then they have been applied to the study of survival and extinction
of species  in \cite{bor2}. 
  
\medskip

In all those cases, a specific derivation is performed for each model to find
stationary product measures, under conditions on the rates if necessary. 
In the present paper, our central result  unifies
and generalizes those derivations, which in particular opens the door to the study
of condensation phenomena and exact solvability for new models, or models not tractable
up to now: In Theorem \ref{th:main2},
we obtain necessary and sufficient conditions on the rates
of a MMP (which is a dynamics in infinite volume)  to guarantee the existence of 
translation-invariant, product invariant probability measures for the process. 
We can state explicitly how these measures look like for single-jump models,
that is misanthropes, zero-range and target processes (abbreviated as MP, ZRP and TP), 
as well as for MM-ZRP and MM-TP. 
For MMP, we then analyse attractiveness properties, as well as condensation issues; 
 our exact formulas for invariant measures enable us to study the relations between 
those two properties. 
Finally, we illustrate our results on various examples, including the ones
previously mentioned. In particular, we prove for the first time that
attractiveness and condensation can coexist on an example of MM-ZRP. 

\medskip

The paper is organized as follows.
In Section \ref{s-model}, we give a formal description and state
 an existence theorem for MMP; proofs are done in Section \ref{sec:construction}.
Section \ref{s-invariant} is devoted to invariant measures of MMP: Theorem \ref{th:main2},
 single-jump models, MM-ZRP and MM-TP. Proofs are done in Section
\ref{s-proofs-sec3}.
Section \ref{s-coupling}  is devoted to attractiveness 
of  MMP, and contains the determination of the extremal translation invariant and invariant measures
for MM-ZRP and MM-TP.
We explain in Section \ref{s-discussion} how known results on condensation for ZRP  can be applied to MMP, and check whether MMP could also be attractive. 
In section \ref{s-examples} we bring examples of MMPs and, for each one, we check whether the conditions established in Sections \ref{s-model}, \ref{s-invariant}, \ref{s-coupling}, and \ref{s-discussion} are satisfied. \\

\section{The model}\label{s-model}
\subsection{Description, existence results}\label{subsec-existence}
\noindent
Let us introduce the \textit{mass migration process}, abbreviated as MMP. 
Particles are located on a countable set $\X$ of sites, typically $\Zd$. 
At a given time $t\ge 0$, for a configuration $\eta_t=\left(\eta_t(x):x\in\X\right)$,
$\eta_t(x)\in\N$ is the number of particles on site $x\in \X$. 
Particles move between sites with respect to the \emph{\GMx dynamics}, that is, $k\ge 1$ particles from the total amount $\a$ 
of particles at a departure site $x$ jump simultaneously to a target site $y$ occupied by $\beta$ particles with a rate  
\begin{equation}\label{rate0} 
p(x,y)\gen^k_{\a,\b}
\end{equation}
where $(p(x,y),x,y\in \X)$ is a transition probability on $\X$, 
and  where, for all $\a,\,\b\in\N$, 
\begin{equation}\label{rateg}
\gen^0_{\a,\beta}=0,\quad\gen^k_{\a,\beta} \hbox{ are nonnegative for } 0<k\leq \alpha,\quad
\hbox{  and }\quad \gen^k_{\a,\beta}=0 \hbox{ for  }  k>\a\hbox{ or  }\a=0. 
\end{equation} 
 This dynamics is {\sl conservative}: the total number of particles involved
in a transition is preserved. 
When  $\X$ is finite, rates (\ref{rate0}) define straightforwardly a 
Markov process $(\eta_t)_{t \geq 0}$ on the state space $\N^{\X}$. 
When  $\X$ is countably infinite, some care is required to define 
$(\eta_t)_{t\ge 0}$ as a Markov process.  For this, we rely on methods by
Andjel \cite{andjel} and Liggett \& Spitzer \cite{liggett-spitzer}
and proceed as follows. Details and proofs are given
in Section \ref{sec:construction}.  \\ \\
To avoid initial configurations that could cause explosions,
we first restrict the state space to
\begin{equation}\label{setX}
\XX=\{ \eta\in \N^{\X}: \|\eta \|<\infty \}          
\end{equation}
where 
\begin{equation}\label{norm-eta}\|\eta \|=\sum_{x\in\X} \eta(x)a_x
\end{equation}
for a positive function $a$ on $\X$ such that, for some  positive constant $M$,
\begin{equation}\label{Mconst}
\sum_{y\in \X} p(x,y) a_y + \sum_{y\in \X} a_y p(y,x) \leq M a_x,\text{ for every $x\in \X$} . 
\end{equation}
and 
\begin{equation}\label{sum-a_x-finite}
\sum_{x\in \X} a_x<\infty.
\end{equation} 
The following proposition and theorem  define the infinitesimal generator $\L$ 
and the corresponding semigroup 
$(S(t), t\geq 0)$ for the Markov process $(\eta_t)_{t \geq 0}$. 
We also introduce the associated Markovian coupled process,
and finally characterize invariant measures for 
 $(\eta_t)_{t \geq 0}$. \\

\noindent
We assume from now on that $\left(p(x,y):\,x,y\in\X\right)$ is an irreducible
transition probability satisfying 
\begin{equation}  \label{mpconst}
\sup\limits_{y\in \X} \sum\limits_{x\in \X} p(x,y)= m_p<\infty.  
\end{equation}
and that the rates satisfy \eqref{rateg}.  Let the set of Lipschitz functions on $\XX$ be
\begin{equation} \label{def-lip}
    \Lip=\{ f:\XX \rightarrow\R\ \text{ such that for some } L_f>0, 
    |f(\eta)-f(\zeta)|\leq L_f \|\eta-\zeta\|,\,\forall\eta,\zeta\in\XX  \}.
\end{equation}
(where $\|\eta-\zeta\|=\sum_{x\in\X} |\eta(x)-\zeta(x)|a_x$). 
\medskip
\begin{proposition}\label{prop:generator}
Assume that there exists a constant $C>0$ such that
\begin{equation}\label{un-assum}
\sum_{k=1}^{\a} k\, \gen^k_{\a,\b}\leq C(\a+\b) \qquad \text{ for all } \a>0,\, \b\geq 0,
\end{equation} 
then the infinitesimal generator defined for $f\in\Lip,\eta\in\XX$ by
\begin{equation}\label{generator}
    \L f(\eta)= \sum_{x,y\in \X}\sum_{k>0} \ p(x,y)  \
                \gen^k_{\eta(x),\eta(y)} 
                \left(f\left(\etb\right)-f\left(\eta\right)\right)
\end{equation}
where
\begin{equation}\label{sxyk}
    (\etb) (z) = \left\{
                      \begin{array}[c]{lll}
			\eta(x)-k  & \quad {\rm if}\ z=x  \;{\rm  and }\; \eta(x)\ge k\\
                        \eta(y)+k  & \quad {\rm if}\  z=y \;{\rm  and }\; \eta(x)\ge k\\\
                        \eta(z)    & \quad {\rm otherwise}
                      \end{array}
          \right.
\end{equation}
satisfies 
\begin{equation}\label{bound-to-generator}
    |\L f(\eta)|\le L_f C(1+M+m_p)\|\eta\|\,. 
\end{equation}
\end{proposition}
\medskip
We stress that the natural condition (\ref{un-assum}) on rates is not strong enough to prove Theorem 
\ref{ex-theorem} below (see Section \ref{sec:construction}), 
hence we introduce for this the following sufficient condition, which implies (\ref{un-assum}): 
There exists a constant $C>0$ such that 
\begin{equation}\label{assum}
\sum_{k=1}^{\a\vee\g} k\, |\gen^k_{\a,\b} - \gen^k_{\g,\delta}| \, \leq \,  C(|\a-\g| + |\b-\delta|) 
\qquad \text{ for all } \a, \b,\g, \delta\geq 0. 
\end{equation}

\smallskip
\begin{theorem}  \label{ex-theorem}
Consider the infinitesimal generator $\L$ given by \eqref{generator},
and assume that condition \eqref{assum}   is satisfied.
Then there exists a Markov semigroup of operators $(S(t), t\geq 0)$, defined
on Lipschitz functions $\Lip $ on $\XX$, and a constant $c>0$ such that   
for all $f\in\Lip,\ t\geq 0$, $\eta\in\XX$, the following items hold:         
\smallskip
\begin{itemize}
\item[\textit{(1)}] $S(t)f\in\Lip $ and $L_{S(t)f}\leq L_fe^{ct}$;  \\
\item[\textit{(2)}] 
$S(t) f(\eta)=f(\eta) + \int\limits_0^{t} \L S(s) f(\eta) \d s $;             \\ 
\item[\textit{(3)}]  if $\sum_{x\in\X}\eta(x)<+\infty$, $S(t) f(\eta)= \mathbb{E}^{\eta} f(\eta_t)$,   \\  
where $(\eta_t)_{t\geq 0}$ on the right-hand side is a (countable state space) 
Markov process with rates~(\ref{rate0}), initial configuration $\eta_0=\eta$, and $\mathbb{E}^{\eta}$
denotes its expectation.
\end{itemize}
\end{theorem}
\medskip
Having a Markov semigroup $S(t)$ on $\Lip$, Daniell-Kolmogorov extension theorem (e.g. \cite[Theorem\,31.1]{RW})
yields the corresponding Markov process, defined by probabilities $P^{\eta}$ on 
trajectories where projections $P^{\eta}(\pi_t\in \cdot)$ concentrate on $\XX$ and satisfy 
$\int f(\xi)P^{\eta}(\pi_t\in d\xi)=S(t)f(\eta)$. \\

We construct similarly a (Markovian) \emph{coupled process} 
$\left( \eta_t,\zeta_t  \right)_{t\ge 0}$ on $\XX\times \XX$ (where we set
$\|(\eta,\zeta)\|=\|\eta\|+\|\zeta\|$) whose
 marginals $(\eta_t)_{t\ge 0}$, $(\zeta_t)_{t\ge 0}$ are
copies of the original MMP. As detailed in Section \ref{sec:construction},
such a process in finite volume  enables to prove Theorem \ref{ex-theorem}, and it
can then be extended analogously to infinite volume, with a Markov semigroup $(\overline{S}(t):t\geq 0)$ 
derived from an infinitesimal generator $\overline{\L}$. 
 Whereas basic coupling was the natural coupling to use in \cite{andjel} and \cite{liggett-spitzer}, it
is not valid anymore here, hence we use the one introduced in \cite{GS}, 
which will moreover be helpful dealing later on with attractiveness
(the purpose for which it was built), see Section \ref{s-coupling} for details. \\

A probability measure $\bar{\mu}$ on $\N^{\X}$  is called  \emph{invariant}  
for the process with generator $\L$ and Markov semigroup $(S(t):t\geq 0)$ if 
\begin{subequations}
\begin{equation}
\bar{\mu} \text{ is supported on } \XX,   \text{ and }  \label{supp-on-X}
\end{equation}
\begin{equation}\label{def-mu-invariant}\int S(t) f \ d\bar{\mu} 
= \int f \ d\bar{\mu}  \qquad \text{ for every bounded } f\in
                                           \Lip \, \text{ and every } t\geq 0 .   
\end{equation}
\end{subequations} 
\begin{proposition}\label{intL=0}
Let $\bar{\mu}$ be a probability measure on $\N^{\X}$ 
satisfying 
\begin{equation}\label{int-norm-eta-fini}
\int \|\eta \| \ d\bar{\mu}(\eta) <\infty 
\end{equation} 
 Let us consider a process with 
generator $\L$ given by \eqref{generator} where rates satisfy \eqref{un-assum}.
Assume that the Markov semigroup $(S(t):t\geq 0)$ is such that statements \textit{(i)--(vii)} 
of Lemma \ref{lemma-spitzer} hold. 
Then \eqref{def-mu-invariant} is equivalent to 
\begin{equation}\label{suff-intL=0}
\int \L f \ d\bar{\mu} =0    \ \text{ for every bounded cylinder function}\  f \text{ on } \N^{\X}.
\end{equation}
\end{proposition} 
%
\subsection{Alternatives}\label{subsec-alter-cons}
%
For some models, Conditions \eqref{assum} and/or \eqref{un-assum} 
will not be valid, and/or the involved probability measures will not satisfy 
Condition \eqref{int-norm-eta-fini}.  
All these assumptions were sufficient for our construction, 
so for some examples they could be not necessary, and an alternative 
construction could work (this includes dynamics in finite
volume). Moreover some examples require a state space smaller than $ \XX$
and a different construction (see Section  \ref{s-examples}).\par\smallskip  

To deal with some of these cases, we will indicate how to adapt the proofs for invariant
measures results in Section \ref{s-invariant} for models
satisfying the two following assumptions: 
\begin{subequations}
\begin{itemize}
\item[$\blacklozenge$]
A probability measure $\bar{\mu}$ on $\N^{\X}$  is invariant  
for the process with generator $\L$ \\ given by \eqref{generator} and Markov semigroup $(S(t):t\geq 0)$ 
when \eqref{def-mu-invariant} is satisfied, \\ and \eqref{def-mu-invariant} is equivalent to  
\eqref{suff-intL=0}.
\\
\vspace{-7mm}
\begin{equation}
      \label{A1}
\end{equation}
\item[$\blacklozenge$]  
There exists a constant $C>0$ such that
\begin{equation}\label{assum-alternative}
\sum_{k\le\a} g_{\a,\b}^k \leq C < \infty.
\end{equation} 
\end{itemize} 
\end{subequations}
\smallskip
Note that assumption \eqref{assum-alternative}  implies  \eqref{un-assum}.
\par
\medskip
\subsection{Examples}\label{subsec-examples}
%
 All along this paper, we will illustrate our results on the following models,
which have only one conservation law.
They were all initially defined as single-jump models, that is with rates 
  $p(x,y)g^k_{\a,\b}$ such that
 $\gen^k_{\a,\b}=0$ for $k>1$. We
 analyse their generalizations to $k\ge 2$, which were either already defined or
 that we introduce in this work. \\

\noindent$\bullet$ 
In the \emph{misanthropes process} (MP), introduced by Cocozza in \cite{Cocozza}, 
\begin{equation}\label{misanth}
\gen^k_{\a,\b} \  = \ \1_{[k=1]} \gen^1_{\a,\b}
\end{equation}
for a nonnegative function $\gen^1_{\cdot,\cdot}$ on $\N\times\N$, non-decreasing 
(non-increasing) in its first (second) coordinate.   
In the works of Godr\`eche et al. \cite{godreche07,LG,LG07}, 
arbitrary (that is, without monotonicity properties) 
nonnegative $\gen^1_{\cdot,\cdot}$ on $\N\times\N$ are considered
in (\ref{misanth}) and the process is then called the \emph{dynamic urn model} 
or the \emph{migration process}. However, as per usual, we keep the denomination MP
also for those cases;  we denote by mass migration processes (MMP) 
the dynamics extended to multiple jumps.\\

\noindent$\bullet$
\emph{Zero range processes} (ZRP) are single-jump dynamics 
(that is, with rate (\ref{misanth})) introduced by Spitzer in \cite{spitzer}, for which
 the dependence on $\beta$ is dropped in $\gen^1_{\a,\b}$. 
 We denote by MM-ZRP their extension to multiple jumps, 
 which was introduced in \cite{GrL}. To simplify the notation for the rates,
 $\gka$ denotes a function of $k$
and $\a$ (the occupation number 
on the departure site of the jump) only: 
\begin{eqnarray} 
\gkab &=&  \1_{[k=1]} \gen^1_{\a} \quad  \text{ (ZRP) }\label{rate-zrp} \\
\gkab &=&  \gka               \qquad\quad  \text{ (MM-ZRP). }\label{rate-mm-zrp}
\end{eqnarray}
\noindent$\bullet$
 \emph{Target processes}  (TP) are single-jump dynamics 
(that is, with rate (\ref{misanth})) introduced in \cite{LG07},
for which the dependence on $\alpha$ is (almost) dropped in $\gen^1_{\a,\b}$: 
only $\alpha>0$ is required.
We define in this paper their generalization  to multiple jumps (MM-TP). 
To simplify the notation for the rates,  $\gkb$ denotes a function of $k$
and $\b$ (the occupation number on the arrival site of the jump) only: 
\begin{eqnarray} 
\gkab &=&  \1_{[k=1\leq\a]}\, g_{\ast,\b}^1 \quad\  \text{ (TP) }\label{rate-tp}\\
\gkab &=&  \1_{[k\leq \a]}\,\gkb    \qquad  \text{ (MM-TP). } \label{rate-mm-tp}
\end{eqnarray}
We will study in detail various examples of MM-ZRP and MM-TP in Section \ref{s-examples}.\\
%


\section{Product invariant measures}\label{s-invariant}
Throughout this section we consider a mass migration process (MMP) $(\eta_t)_{t\ge 0}$
on the set of sites  $\X=\Zd$, with generator \eqref{generator}, rates
satisfying assumption \eqref{un-assum},  which is translation invariant, that is with
$(p(x,y),x,y\in\X)$ a translation invariant transition probability 
(hence bistochastic, so that $m_p=1$ in \eqref{mpconst}).
Let us denote by $\mathcal{S}$ the set of translation-invariant probability measures on $\N^{\Z^d}$,
and by $\mathcal{I}$ the set of invariant probability measures for the MMP $(\eta_t)_{t\ge 0}$. 
We are interested in product, translation-invariant
and invariant probability measures for  $(\eta_t)_{t\ge 0}$.
Most proofs are done in Section \ref{s-proofs-sec3}.

\subsection{Necessary and sufficient conditions for product invariant measures}\label{subsec:cns-product}
In this subsection, we exhibit necessary and sufficient relations between the rates 
of the MMP $(\eta_t)_{t\ge 0}$ and the single site marginal $\mu$ of a \emph{product, translation-invariant} probability
measure $\bar{\mu}$ on $\N^{\Zd}$  which make the latter invariant for $(\eta_t)_{t\ge 0}$. 

\begin{theorem}\label{th:main2}
Consider a mass migration process with generator $\L$  given by \eqref{generator},
whose rates satisfy  \eqref{un-assum}.
Let $\bar{\mu}\in{\mathcal S}$ be a  product measure 
whose single site marginal $\mu$ has a finite first moment $\|\mu\|_1 = \sum_{n\in\N} n\mu(n)$.
Let us denote for all $\a,\b\geq 0$  
\begin{equation}\label{A-def2} 
\mathbf{A}(\a,\b) = \left\{
\begin{array}{ll}
 \sum\limits_{k\leq \b}   g^k_{\a+k,\b-k} \dfrac{\mu(\a+k)\mu(\b-k)}{\mu(\a)\mu(\b)}
  - \sum\limits_{k\leq \a} g^k_{\a,\b} & \ \text{if } \mu(\a)\mu(\b)\neq 0
\\[1mm]
 0 & \ \text{otherwise.}
 \end{array}
\right.
\end{equation}
A necessary condition for $\bar{\mu}$ to be invariant for the process is 
\begin{equation}\label{inv-sym21}
\sum_{k\le \b} g^k_{\a+k,\b-k}\mu(\a+k)\mu(\b-k)=0 \ 
\text{for all }\a,\b\geq 0\ \text{such that }\mu(\a)\mu(\b)= 0
\end{equation} 
combined with
\begin{equation}\label{inv-sym22}
\mathbf{A}(\a,\b) = -   \mathbf{A}(\b,\a)\ \text{for all }\a,\b\geq 0.
\end{equation} 
These two conditions are also sufficient when $p(\cdot,\cdot)$ is symmetric. When $p(\cdot,\cdot)$ is asymmetric,   
\eqref{inv-sym21} has to be combined with the following stronger condition 
to be necessary and sufficient: there exists a function $\psi$ on $\N$ such that 
\begin{equation}\label{inv-psi2}
\mathbf{A}(\a,\b) = \psi(\b) - \psi(\a) \quad
\text{for all } \a,\b\geq 0 \quad
\text{such that }\mu(\a)\mu(\b)\neq 0. 
\end{equation} 
\end{theorem}
\medskip

 Theorem \ref{th:main2}  includes dynamics 
with state space $\{0,\cdots,\g\}^{\Zd}$, for some $\g>0$, as well as
 measures $\bar{\mu}$ such that $\mu(\a)=0$ 
for various value(s) of $\a$. We will deal with such cases in
a forthcoming paper, and we will  restrict ourselves in the present paper
to  measures such that  $\mu(\a)>0$ for all $\a\in\N$. In this case definition 
\eqref{A-def2} reduces to its first line, and condition \eqref{inv-sym21}
is absent.
 \begin{remark}\label{Aaa}
(a) Conditions \eqref{inv-sym22} or \eqref{inv-psi2} imply that $\mathbf{A}(\a,\a) =0$ for all $\a\ge 0$.\par
\noindent (b)  Condition \eqref{inv-psi2} combined with $\psi(0)=0$ implies 
\begin{equation}\label{inv-psi-bis}
\psi(\a)=- \mathbf{A}(\a,0)=\mathbf{A}(0,\a)\quad
\text{for all } \a\geq 0. 
\end{equation}
\end{remark} 
\medskip

It is also possible to deal with a product measure $\bar{\mu}$
whose single site
marginal $\mu$ has an infinite first moment, either by taking an alternative
norm to the one defined in \eqref{norm-eta} in the construction of the process, 
or if assumptions (\ref{A1})--(\ref{assum-alternative}) are valid.
\begin{remark}\label{rk:alternative-ax}
If $\sum_{n\ge 0} n \mu(n) =+\infty$ but there exists a non-decreasing function 
$f:[0,+\infty)\to[0,+\infty)$, with $f(0)=0$ and increasing to infinity when $n\rightarrow\infty$ in
such a way that $\sum_{n\ge 0} f(n) \mu(n) <\infty$, then it is possible to replace the coefficients $a_x$ 
in the norm $\|\cdot\|$ defined in \eqref{norm-eta} by new coefficients
$a^*_x$ for which a new norm $\|\cdot\|^*$ defined by 
$\|\eta\|^*=\sum_{x\in\X} \eta(x)a^*_x$ satisfies $\bar{\mu}(\|\eta\|^*<\infty)=1$.
\end{remark}
\smallskip
\begin{corollary}\label{cor:main-alternative}
Under assumptions (\ref{A1})--(\ref{assum-alternative}), Theorem \ref{th:main2} is valid
without assuming that
$\mu$ has a finite first moment.
\end{corollary}
{}From Theorem \ref{th:main2}  and Corollary \ref{cor:main-alternative}, 
 we deduce that if the product measure $\bar{\mu}$ is invariant for the MMP, 
the latter possesses  a one-parameter family of invariant product 
probability measures:

\smallskip
\begin{corollary}\label{family} 
Consider a MMP with generator $\L$ given by \eqref{generator}.
Let $\bar{\mu}\in{\mathcal S}$ be a  product measure invariant for the process. Assume that either 
the rates of the MMP  satisfy assumption \eqref{un-assum} and the
single site marginal $\mu$ of $\bar{\mu}$ has a finite first moment, 
or that the rates of the MMP  satisfy assumptions (\ref{A1})--(\ref{assum-alternative}). 
Let $\varphi_c\ge 1$ be the radius of convergence of the series
\begin{equation}\label{Z}
Z_\vfi= \sum_{n=0}^\infty \vfi^n \mu(n).
\end{equation}
Then for all $\varphi < \varphi_c$,  the translation invariant
product measure $\bar{\mu}_\varphi$ defined by its single site marginal 
\begin{equation}\label{muf}
\mu_\vfi (n) =  \bar{\mu}_\vfi (\eta(x)=n) =\frac{1}{Z_\vfi} \vfi^n \mu(n),\qquad n\in\N, 
\end{equation}
is invariant for the process. 
Moreover, 
\begin{equation}
      \label{eitherif}
\text{either if }\sum_{n\ge 0} n \vfi_c^n {\mu}(n) <\infty,\ 
\text{or if }Z_{\varphi_c}<\infty  \text{ and assumptions (\ref{A1})--(\ref{assum-alternative}) are satisfied,}  
\end{equation} 
then the 
measure $\bar{\mu}_{\varphi_c}$ is also invariant for the process. 

Therefore the latter possesses a one-parameter family of invariant product 
probability measures, either $\{\mp:\vfi\in {\rm Rad}(Z)\} $ or
 $\{\mp:\vfi\in {\rm Rad}(Z')\} $, where ${\rm Rad}(Z')  
\subseteq {\rm Rad}(Z) $ are defined by
\begin{equation}\label{Rad}
{\rm Rad}(Z) = \left\{ 
\begin{array}{ll} 
{} (0,{\varphi_c}] & \text{if } Z_{\varphi_c}<\infty\\[2mm]
{} (0,{\varphi_c}) & \text{if } Z_{\varphi_c}=+\infty
\end{array}
\right.
\end{equation} 
and 
\begin{equation}\label{Rad'}
{\rm Rad}(Z') =  \left\{ 
\begin{array}{ll} 
{} (0,{\varphi_c}] & \text{if } \sum_{n\ge 0} n \vfi_c^n {\mu}(n) <\infty \\[2mm]
{} (0,{\varphi_c}) & \text{otherwise. }
\end{array}
\right.
\end{equation}
\end{corollary}
For further use, when $Z_{\varphi_c}<\infty$ we define  the \textit{critical density} by
\begin{equation}\label{eq:rho_c}
\rho_c= Z_{\varphi_c}^{-1} \sum_{n\ge 0} n \vfi_c^n {\mu}(n).
\end{equation}
%

\subsection{Applications}\label{subsec:Applications}
\makebox{} \\
In this subsection, we use Theorem  
\ref{th:main2},  Corollaries \ref{cor:main-alternative} and \ref{family}
 to characterize invariant product  probability measures 
for the various examples introduced in Section \ref{subsec-examples}. 
The necessary and sufficient conditions obtained in Theorem  \ref{th:main2} will be exploited in two ways: 
1) the rates being given, check that a measure $\bar{\mu}$ is invariant;
2)~a measure $\bar{\mu}$ being given, define rates for which $\bar{\mu}$ is invariant.

This second point of view will enable us in Section \ref{s-discussion}  
to associate many different dynamics to a measure $\bar{\mu}$
satisfying condensation properties. 
\medskip
\subsubsection{Single-jump models}\label{subsubsec:single-jump} 
\mbox{}\\

\noindent
 The single-jump models we consider are misanthropes processes (MP), for which
the generic rate is  $p(x,y)g^k_{\a,\b}$ with $g^k_{\a,\b}$ given by (\ref{misanth}). 
In this case  we derive  the following proposition from Theorem \ref{th:main2}.
\begin{proposition}\label{prop:compcondMP}
Consider a MP
 with rates  $p(x,y)g^1_{\a,\b}$. Assume that
 the  rates
$(g^1_{1,\a})_{\a\ge 0}$ and $(g^1_{\a,0})_{\a\ge 1}$ are positive. 
Let $\bar{\mu}\in\mathcal{S}$ be a product probability measure 
whose single site marginal $\mu$  satisfies $\mu(\a)>0$ for all $\a\in\N$,  
and such that either $\mu$ has a finite first moment
or assumptions \eqref{A1}--\eqref{assum-alternative} are verified.  
\\
Then $\bar{\mu}$ is invariant for the MP  if and only if 
its single site marginal $\mu$ is given by 
\begin{equation}\label{k=1:mug2}
\mu(\a)= \mu(0) \prod_{k=1}^\a  \Big[\; 
\frac{\mu(1)}{\mu(0)} \frac{g^1_{1,k-1} }{g^1_{k,0}} \Big]\quad  \text{ for all } \a\ge 1;
\end{equation}
and provided that the following compatibility conditions are fulfilled, respectively
\begin{equation}\label{k=1-compat}
g^1_{\a+1,\b}=  \frac{g^1_{\a+1,0}}{g^1_{1,\a} }  \frac{g^1_{1,\b}}{g^1_{\b+1,0}} 
g^1_{\b+1,\a}\text{ for all } \a, \b\ge 0,
\end{equation}
in the symmetric case, 
and  \eqref{k=1-compat} together with
\begin{equation}\label{k=1:inv-psi}
g^1_{\b,\a}  = g^1_{\a,\b} + g^1_{\b,0} -g^1_{\a,0} \quad \text{ for all } \a, \b \ge 0
\end{equation}
in the asymmetric case.
\end{proposition}  
\begin{proof}
For MP, definition \eqref{A-def2} becomes:
\begin{eqnarray}\label{k=1:A-def}
\mathbf{A}(\a,\b) &=& \frac{\mu(\a+1)}{\mu(\a)} \frac{\mu(\b-1)}{\mu(\b)}g_{\a+1,\b-1}^1 
 - g_{\a,\b}^1 \qquad \text{for all } \a\geq 0,\b\geq 1\\
\mathbf{A}(\a,0) &=& -g_{\a,0}^1 \makebox[5.9cm]{} \text{for all } \a\geq 0. \label{k=1:A-def0}
\end{eqnarray}
Here, condition \eqref{inv-sym22} of Theorem \ref{th:main2} reduces to  well known
\textit{detailed balance conditions}, also called 
\textit{pairwise balance conditions} in the asymmetric case \cite{LG, SRB}.
Indeed, using \eqref{k=1:A-def}, condition \eqref{inv-sym22} for $\a\ge 1$, $\b\ge 1$ can be written
as:
\begin{equation}\label{k=1:DB1}
g^1_{\a+1,\b-1} \mu(\a+1)\mu(\b-1) - g^1_{\b,\a} \mu(\a)\mu(\b) = 
g^1_{\a,\b} \mu(\a)\mu(\b) - g^1_{\b+1,\a-1} \mu(\a-1)\mu(\b+1)
\end{equation}
where both sides of the equality have the same form and differ only by a shift $(\a,\b)\rightarrow (\a-1,\b+1)$.
For $\a=0$, $\b\ge 1$, using  \eqref{k=1:A-def},  \eqref{k=1:A-def0}, condition  \eqref{inv-sym22} reads
\begin{equation}\label{k=1:DB2}
g^1_{1,\b-1} \mu(1)\mu(\b-1) - g^1_{\b,0} \mu(0)\mu(\b) = 0.
\end{equation}
Iterating \eqref{k=1:DB1} then using \eqref{k=1:DB2}, one gets for $\a\ge 0$, $\b\ge 1$
\begin{eqnarray}\label{k=1:DB3}
g^1_{\a+1,\b-1} \mu(\a+1)\mu(\b-1) - g^1_{\b,\a} \mu(\a)\mu(\b) &= &
g^1_{1,\a+\b-1} \mu(1)\mu(\a+\b-1) - g^1_{\a+\b,0} \mu(\a+\b)\mu(0)\nonumber\\
&=& 0.
\end{eqnarray}
Equation \eqref{k=1:DB3} expresses detailed balance condition, which can be written 
in a more symmetric form as (cf. \cite{LG}):
\begin{equation}\label{k=1:DB}
g^1_{\a+1,\b} \mu(\a+1)\mu(\b)= g^1_{\b+1,\a} \mu(\b+1)\mu(\a)\quad \text{ for all } \a, \b \ge 0.
\end{equation}
On one hand,  \eqref{k=1:DB} for $\b=0$ gives a relation between the invariant 
product measure and the subsets of rates
$(g^1_{1,\a})_{\a\ge 0}$ and $(g^1_{\a,0})_{\a\ge 1}$. One has 
\begin{equation}\label{k=1:mug}
 \frac{g^1_{\a+1,0} }{g^1_{1,\a} }= \frac{\mu(1)}{\mu(0)} \; 
 \frac{\mu(\a)}{\mu(\a+1)}\text{ for all } \a \ge 0.
\end{equation}
\begin{remark}\label{consistently}

Equations \eqref{k=1:mug} allows to derive the measure $\mu$, given the rates, up to the ratio  $ {\mu(1)}/{\mu(0)}$.
This is consistent with  Corollary \ref{family} which defines a family of invariant product measures from a given one.
\end{remark} 
In other words, the existence of a product invariant measure for the 
dynamics implies the existence of a 
family of such measures. Indeed, from \eqref{k=1:mug}, one gets  \eqref{k=1:mug2}. 

On the other hand, inserting expression \eqref{k=1:mug} in \eqref{k=1:DB} 
leads to an expression of  detailed balance condition 
in terms of the jump rates only, that is  \eqref{k=1-compat}. 
Condition  \eqref{k=1-compat} has to be supplemented by 
condition \eqref{inv-psi2} in the asymmetric case: 
the latter, after using Remark~\ref{Aaa}(b) and \eqref{k=1:A-def0}, \eqref{k=1:DB3},
writes \eqref{k=1:inv-psi}.  
\end{proof}
\medskip
\noindent 
In \cite{Cocozza}, conditions \eqref{k=1-compat}, \eqref{k=1:inv-psi} 
(denoted there by (2.3), (2.4)) were proved to be 
sufficient to get product invariant measures, satisfying \eqref{k=1:mug}
(denoted there by (2.6)).  \\  

\noindent
We now apply  Proposition \ref{prop:compcondMP} to ZRP and TP, that is, we check what
become the compatibility conditions \eqref{k=1-compat}, \eqref{k=1:inv-psi},
and the form of the marginal \eqref{k=1:mug2} of a product invariant measure,
when it exists. We also explicit Corollary \ref{family} in these cases. 
\\ \\
%
\noindent $\bullet$ \textit{Zero range process (ZRP).} 
 
\smallskip
\noindent
Here, 
both conditions \eqref{k=1-compat} and  \eqref{k=1:inv-psi} are always  satisfied, so that there are no compatibility conditions
on the rates. \par 
\smallskip

\noindent
Equation \eqref{k=1:mug} expresses that for a product invariant  measure 
with  marginal $\mu$, 
 there is a constant $\varphi_\mu >0$
such that for all $\a\ge 1$,
\begin{equation}\label{zrp-mu-invt}
g_\a^1 = \varphi_\mu  \frac{\mu(\a-1)}{\mu(\a)}
\end{equation}
Consistently with  Remark \ref{consistently}, a convenient expression for $\varphi_\mu$ is thus
\begin{equation}\label{fimu} 
\varphi_\mu =g_1^1 \frac{\mu(1)}{\mu(0)}>0
\end{equation}

Let $\varphi_c\ge 1$ be the radius of convergence of the series  \eqref{Z} which writes here 
\begin{equation}\label{inv-zrp-Z}
Z_\varphi= \mu(0)  \sum_{k=0}^\infty\frac{{(\varphi\varphi_\mu)}^k}{(g_k^1)!}
\end{equation}
where
\begin{equation}\label{eq:factoriel-g}
(g_k^1)! = 
\begin{cases}
1 & \text{ if } k =0\cr
g_k^1 g_{k-1}^1\cdots g_1^1 & \text{ if } k > 0.
\end{cases}
\end{equation}
By Corollary \ref{family},
there exists a family of product, translation-invariant, 
invariant measures, $\{\bar{\mu}_{\varphi}, 0<\varphi<\varphi_c\}$  or
$\{\bar{\mu}_{\varphi}, 0<\varphi\le\varphi_c\}$ if \eqref{eitherif} is satisfied, 
with single site marginal (cf. \eqref{k=1:mug2}):
\begin{equation}\label{inv-zrp}
\mu_{\varphi}(\a)=\frac{\mu(0)}{Z_\varphi}\frac{(\varphi\varphi_\mu)^{\a} }
{(g_\a^1)!}\, \text{ for } \a\geq 0.
\end{equation}
\begin{remark}\label{zrp-k=1-s=as}
Instead of considering equation \eqref{k=1:inv-psi}, one can 
insert Equation \eqref{zrp-mu-invt} in \eqref{k=1:A-def}, \eqref{k=1:A-def0} 
and get the following expression for  $A(\a,\b)$,
\begin{equation}
A(\a,\b) =  g_\b^1 - g_\a^1 \quad\text{ for all } \a, \b \ge 0
\end{equation}
which is directly the form of condition \eqref{inv-psi2} with $\psi(\a)= g_\a^1$. 
This implies that, for  the zero range process, 
conditions for the existence of  product invariant measures are the same  
for both  $p(\cdot,\cdot)$ symmetric and  $p(\cdot,\cdot)$ asymmetric.
\end{remark}
\noindent
The invariant measures derived here are those studied by \cite{spitzer,andjel}.\\ 

%
\noindent $\bullet$ \textit{Target process (TP).} \par
%
\noindent
For a target process, conditions \eqref{k=1-compat} are always  
satisfied, hence there are no compatibility conditions on the rates in the symmetric case.
Condition  \eqref{k=1:inv-psi} becomes here
\begin{equation}\label{compat-tp-asym}
 g_{*,\b}^1= g_{*,1}^1\qquad\mbox{ for  }\,\b>0.  
\end{equation}
Thus, only under condition \eqref{compat-tp-asym} does an asymmetric target process
 admit  invariant product probability
measures.\\
\noindent
Equation \eqref{k=1:mug} expresses that for a given product invariant  measure 
with  marginal $\mu$, 
 there is a constant 
\begin{equation}\label{tildefimu} 
\tilde\varphi_\mu = \frac{1}{g_{*,0}^1} \frac{\mu(1)}{\mu(0)}>0
\end{equation}
such that for all $\b\ge 0$,
\begin{equation}\label{tp-mu-invt}
g_{*,\b}^1  \tilde \varphi_\mu= \frac{\mu(\b+1)}{\mu(\b)}.
\end{equation}
Let $\varphi_c\ge 1$ be the radius of convergence of the series
\begin{equation}\label{inv-tp-Z}
 Z_\varphi= \mu(0) \Bigl( 1 +  \sum_{k=1}^\infty(\varphi\tilde \varphi_\mu)^k\; 
  g_{*,k-1}^1\,  g_{*,k-2}^1  \cdots  g_{*,0}^1\Bigr).
\end{equation}
Then, by Corollary \ref{family},
there exists a one-parameter family of product, translation-invariant, 
 invariant probability measures, $\{\bar{\mu}_{\varphi}, 0<\varphi<\varphi_c\}$
  or $\{\bar{\mu}_{\varphi}, 0<\varphi\le\varphi_c\}$ if \eqref{eitherif} is satisfied,  
 with single site marginal given by (cf. \eqref{k=1:mug2}):
\begin{equation}\label{inv-tp}
\mu_{\varphi}(0)=\frac{\mu(0)}{Z_\varphi}, \qquad \mu_{\varphi}(\a)=\frac{\mu(0)}{Z_\varphi}
(\varphi\tilde \varphi_\mu)^\a\;  
g_{*,\a-1}^1\,  g_{*,\a-2}^1  \cdots g_{*,0}^1 \, \text{ for } \a\geq 1.
\end{equation}
Then, only under condition \eqref{compat-tp-asym}, an asymmetric target process
 admits a family of invariant product probability
measures,  
for which the expression \eqref{inv-tp}  for the single site marginal
depends on the two distinct jump rates, and becomes: 
\begin{equation}\label{asym-tp-inv}  
\mu_\varphi(\a) 
=\mu_\varphi(0) (\varphi\tilde \varphi_\mu)^{\a}
  g_{*,0}^1 (  g_{*,1}^1)^{\a-1}
\quad\text{for all } \a\ge 1,\quad
\mu_\varphi(0) = \frac{(1 - \varphi \tilde \varphi_\mu)  g_{*,1}^1}
{g_{*,1}^1 + \varphi\tilde \varphi_\mu  ( g_{*,0}^1 -   g_{*,1}^1)}.
\end{equation} 
\begin{remark}\label{rk:duality-zrp-tp}
Comparing \eqref{inv-tp}
with the analogous formula \eqref{inv-zrp} for zero range process, we see that 
there is a family of invariant product measures,  common to both a 
zero range process  and a symmetric target process  
if and only if there exists a constant $c_0 >0$ such that 
\begin{equation}\label{duality-zrp-tp} 
  g_{*,\a}^1 = \frac{c_0^2}{g_{\a+1}^1} 
 \end{equation}
holds for all $ \a\geq 0$.  In \cite{LG07}, 
\eqref{duality-zrp-tp} is called a \textit{duality relation} between ZRP and TP.
\end{remark}

The invariant measures derived here were studied in \cite{LG07}. \\ 

\subsubsection{Mass migration zero range process  (MM-ZRP). } 
\makebox{} \\
 In this subsection and the following one, we will see that the possibility
for $k$ to be larger than 1 yields a very different behavior than for single-jump models. 
\begin{proposition}\label{gzrp}
Let $\bar{\mu}\in\mathcal{S}$ be a product measure 
whose single site marginal $\mu$  satisfies $\mu(0)>0$. 
Consider a MM-ZRP
 with rates $p(x,y)g^k_{\a}$ for an asymmetric $p(\cdot,\cdot)$ and $g^k_{\a}$.
  Assume that either the rates satisfy  \eqref{un-assum} and 
 $\mu$ has a finite first moment, or that assumptions \eqref{A1}--\eqref{assum-alternative}
 are satisfied. Then $\bar{\mu}$ is invariant for the MM-ZRP  if and only if 
\begin{equation}\label{inv-s}
g^k_{\a+k}\mu(\a+k) = \mu(\a)\frac{\mu(k)}{\mu(0)}g^k_k        
\qquad \text{for all } k\geq 1,\a\geq 1 .
\end{equation} 
If $p(\cdot,\cdot)$ is symmetric then condition \eqref{inv-s} is only 
sufficient for $\bar{\mu}$ to be invariant.
\end{proposition}
The form \eqref{inv-s} for the rate is the generic one exhibited in 
\cite{EMZ1,GrL}. We will study such models in detail along the paper, 
in Subsection \ref{subsec-attra-cond}, examples \ref{trojka} and \ref{gen-zrp}
in Section \ref{s-examples}.\\ 

 An additional assumption on the rates yields the following corollary.
\begin{corollary}\label{cor:gzrp}
Let $\bar{\mu}\in\mathcal{S}$ be a product measure 
whose single site marginal $\mu$ satisfies $\mu(0)>0$. 
Consider a MM-ZRP with rates $p(x,y)g^k_{\a}$ for an asymmetric $p(\cdot,\cdot)$.
Assume that either the rates satisfy  \eqref{un-assum} and 
$\mu$ has a finite first moment, or that assumptions \eqref{A1}--\eqref{assum-alternative}
are satisfied. Assume that 
\begin{equation}\label{eq:g1_zr0} 
\forall\, \a\geq 1,\quad g^1_\a>0.
\end{equation}
Then $\bar{\mu}$ is invariant for the MM-ZRP  if and only if 
\begin{equation}\label{inv-rate-zrp-1}
\frac{\mu(\a)}{\mu(0)} = \frac{1}{(g^1_\a)!}
\Big(\frac{\mu(1)}{\mu(0)} g_1^1 \Big)^{\a} \text{ for all } \a\geq 1,\quad\&
\end{equation}
\begin{equation}\label{inv-rate-zrp}
\quad\&\quad g^k_{\a} = \frac{(g^1_\a)!}{(g^1_{\a-k})!(g^1_k)!}g^k_k 
\text{ for all } 1\leq k\leq \a
\end{equation}
where $(g^1_k)!$ was defined in \eqref{eq:factoriel-g}.
\end{corollary}
Indeed putting $k=1$ in \eqref{inv-s} gives \eqref{inv-rate-zrp-1},
and then inserting \eqref{inv-rate-zrp-1} back in \eqref{inv-s}
for $k\geq 2$ gives the compatibility condition \eqref{inv-rate-zrp} 
on the rates. The converse is straightforward.
\begin{remark}\label{rk:compat-zrp-k}
Unlike for ZRP, the rates have here to satisfy a compatibility condition, 
namely~\eqref{inv-rate-zrp}.
\end{remark} 
The result of Proposition \ref{gzrp} can be used on one hand 
to find invariant measures when the rates of a process are given, and on the other
hand to set rates of a process such that it has a prescribed invariant measure. 
Hence we rephrase Proposition \ref{gzrp} and Corollary \ref{cor:gzrp} as follows.
\begin{proposition}\label{gzrp-exact} \makebox{} \\
(a) If $\mu$ is a probability measure   on $\N$ with $\mu(\a)>0$ for all $\a\in\N$, 
then formula \eqref{inv-s} gives transition rates for a MM-ZRP
for which the product  measure $\bar{\mu}$ with single site marginal 
$\mu$  is invariant.  Here 
$g^k_k>0,\, k\geq 1$ can be arbitrarily chosen  such that the rates 
$g^k_{\a+k}\text{ for all }\a,k\geq 1$ satisfy   \eqref{un-assum}.  
\\[2mm]
(b)
If a MM-ZRP has rates $g^k_\a$ which satisfy  \eqref{un-assum},   \eqref{eq:g1_zr0}
and \eqref{inv-rate-zrp},
then this MM-ZRP possesses a one-parameter family of product invariant probability
measures $\bar{\mu}_\varphi$ with single site marginal given by  \eqref{fimu}--\eqref{inv-zrp}, 
for $\varphi\in {\rm Rad}(Z')$  or $\varphi\in {\rm Rad}(Z)$. \par
\smallskip\noindent
Note that both (a) and (b) hold independently of $p(\cdot,\cdot)$. 
If we assume $p(\cdot,\cdot)$ asymmetric  then condition \eqref{inv-rate-zrp}  is   also necessary for 
the existence of product invariant measures.
\end{proposition}
\begin{remark}\label{rk:inv-zrp-k>1}
In (b) above, the measure $\bar{\mu}_\varphi$, which 
depends only on the rates $g^k_\a$ for $k=1$, is the same as in the single-jump case.
\end{remark} 
We conclude that an asymmetric MM-ZRP has the set of product, 
translation-invariant, invariant measures either empty (if \eqref{inv-s}
is not satisfied for any $\mu$) or given explicitly by either
$\{\bar{\mu}_{\varphi}: \varphi\in {\rm Rad}(Z') \}$
or $\{\bar{\mu}_{\varphi}: \varphi\in {\rm Rad}(Z) \}$ 
where $\bar{\mu}_{\varphi}$ are product measures with marginals 
given by  \eqref{fimu}--\eqref{inv-zrp}. 
 
The situation is rather different for a symmetric MM-ZRP. 
In general, the only equivalence result follows from Theorem \ref{th:main2} and says that
a product, translation-invariant probability measure $\bar{\mu}$ (such that 
$\bar{\mu}(\XX)=1$) is invariant for this process with rates $g^k_{\a}$ if and only if 
\begin{eqnarray}\label{inv-s-sym}
&& \sum_{k\leq \b}g^k_{\a+k}{\mu(\a+k)\mu(\b-k)}  - \sum_{k\leq \a}g^k_{\a} {\mu(\a)\mu(\b)} \nonumber \\
&& =  - \sum_{k\leq \a}g^k_{\b+k}{\mu(\b+k)\mu(\a-k)}  + \sum_{k\leq \b}g^k_{\b}  {\mu(\b)\mu(\a)}  
 \qquad  \forall \a\geq 0,\b\geq 1.   
\end{eqnarray}
If \eqref{inv-s} is satisfied, then all product, 
translation-invariant probability measures $\bar{\mu}_{\varphi}$ with marginals 
given by \eqref{inv-zrp} are  invariant for the symmetric 
MM-ZRP with rates $g^k_{\a}$. But there could exist other product translation-invariant probability measures satisfying  
\eqref{inv-s-sym} but not~\eqref{inv-s}.\par
%

\subsubsection{Mass migration target process  (MM-TP). } 
%
\begin{theorem}\label{target}
Let $\bar{\mu}\in\mathcal{S}$ be a product measure  
whose single site marginal $\mu$ satisfies $\mu(\a)>0$ for all $\a\in\N$.
Consider a MM-TP with rates $p(x,y)g^k_{*,\a}$ for an asymmetric $p(\cdot,\cdot)$.
Assume that either the rates satisfy  \eqref{un-assum} and 
$\mu$ has a finite first moment, or that assumptions \eqref{A1}--\eqref{assum-alternative}
are satisfied. Then $\bar{\mu}$ is invariant for the MM-TP if and only if 
\begin{equation}\label{inv-tt}
g^{\a}_{*,\b} = g^{\a}_{*,0} +\frac{1}{\mu(\b)} \sum_{k=1}^{\b} H_{\a}(\b,k) g^{k}_{*,0}  
      \qquad \text{for all } \a\geq 1,\b\geq 1 
\end{equation} 
where the $H_\alpha(\beta,k)$ depend only on $\mu(\cdot)$ and are solution of the following recurrence relations:
\begin{equation}\label{tp:hrec}
\begin{cases} H_\alpha(\beta,\beta) =   \Delta_\alpha(\beta) \mu(0) 
 \ \ \text{ for } \beta\ge 1,  \\[1mm]
H_\alpha(\beta,k) =  \Delta_\alpha( k) \mu(\beta-k) +\sum_{l=1}^{\beta-k} 
 \Delta_\alpha(l) H_l(\beta-l,k) \ \ \text{ for }  \beta\ge 2,\,1\le k\le \beta-1,
\end{cases}
\end{equation} 
with for all $r>0$ and all $s\ge 0$
\begin{equation}\label{tp:Delta}
\Delta_r(s)= \frac{\mu(r+s)}{\mu(r)}- \frac{\mu(r+s-1)}{\mu(r-1)}.
\end{equation}
\end{theorem}
\bigskip\medskip
\noindent
By iterating \eqref{tp:hrec}, we obtain the simplest terms $H_\alpha(\beta,k)$:
\begin{eqnarray}
  H_\alpha(\beta,\beta-1) &=&  \Delta_\alpha(\beta -1) \mu(1)
  +   \Delta_\alpha(1)  \Delta_1(\beta-1) \mu(0)
 \nonumber\\
   H_\alpha(\beta,\beta- 2) &=&  \Delta_\alpha(\beta-2) \mu(2)
  +   \Delta_\alpha(1) \Delta_1(\beta-2) \bigl( \mu(1) + \Delta_1(1)  \mu(0) \bigr)
   +  \Delta_\alpha(2)   \Delta_2(\beta-2) \mu(0).
    \nonumber
\end{eqnarray}
The general term can be  guessed: 
\begin{eqnarray}\nonumber
&&   H_\alpha(\beta,k) =  \Delta_\alpha(k) \mu(\beta -k) \\ \label{tp:hexp}
&&+\sum_{r=1}^{\b-k}\, 
 \sum_{\substack{  k_1,\dots,k_r\ge 1
k_1+\dots+k_r\le \b-k}}\!\!\!\!\!\!
   \Delta_\alpha(k_1)\Delta_{k_1}(k_{2})\dots  
    \Delta_{k_{r-1}}(k_{r}) \Delta_{k_r}(k) \mu(\beta-k-\sum_{j=1}^r k_j). 
\end{eqnarray}
Like for MM-ZRP, the result of Theorem \ref{target} can be used on one hand to find
 invariant measures  of a MM-TP with given rates, and on the other
hand to set rates of a MM-TP  such that it has a prescribed invariant measure. 
 As in Corollary \ref{cor:gzrp} and Proposition \ref{gzrp-exact}, 
an additional assumption on the rates enables to write a compatibility
condition on the rates. 
\begin{proposition}\label{target-exact} 
\makebox{} \\
(a) For a probability measure $\mu$ on $\N$ with $\mu(\a)>0$ for all $\a\in\N$, 
for arbitrarily chosen $g^{\a}_{*,0} >0$ for all $\a\geq 1$  such that 
condition \eqref{un-assum}   is satisfied, 
equations \eqref {inv-tt}--\eqref {tp:hrec}  give  transition rates for a MM-TP 
for which the product measure with single site marginal $\mu$ 
is invariant, provided that all such rates are non-negative. 
A sufficient condition for  the $g^\a_{*,\b}$ to be non-negative  
is that $\Delta_r(s) \geq 0$ for all $r>0$, $s\ge 0$,
or equivalently that the ratio $\mu(n+1)/\mu(n)$ is a nondecreasing function of $n$. 
\\[2mm]
(b) 
Consider a MM-TP with the following rates.
For $k\ge 1$, $g^k_{*,0}$ and $g^k_{*,1}$ are given such that
\begin{equation}\label{eq:weak-phil}
g^1_{*,0}> 0,\quad  g^1_{*,1}> 0;\qquad  \text{ and }\ \forall   \a\geq 2,\quad 
g^1_{*,1}+\sum\limits_{k=2}^\a (g^k_{*,1} -  g^k_{*,0}) >0.
\end{equation} 
For the weights $w(\a)$, $\a\ge 0$, defined as 
\begin{equation}\label{ex5-w-def}
\left\{ 
\begin{array}{ll}
w(0)= 1\\
w(1)= g^1_{*,0}\\
w(2)=g^1_{*,1} g^1_{*,0}\\
w(\a)= w(\a -1) \left( g^1_{*,1}
+\sum\limits_{i=2}^{\a -1} (g^i_{*,1} - g^i_{*,0} )\right) \qquad \forall\,  \a \ge 3
\end{array}
\right.
\end{equation}
then setting  for all $r>0$ and all $s\ge 0$,
\begin{equation}\label{tp:Delta*}
\Delta^*_r(s)= \frac{w(r+s)}{w(r)}- \frac{w(r+s-1)}{w(r-1)}
\end{equation}
and
\begin{equation}\label{target-rates-c}
\qquad\left\{
\begin{array}{l}
\displaystyle{ H^*_\alpha(\beta,\beta) =  \frac{w(0)}{w(\b)} 
\Delta^*_\alpha(\beta) } \hfill   \forall\, \a\geq 1,\b\geq 1,\\[4mm]
\displaystyle{ H^*_\alpha(\beta,k) =  \frac{w(\b-k)}{w(\b)} 
\Delta^*_\alpha(k) }\\ [4mm]
\phantom{\displaystyle{ H^*_\alpha(\beta,k)}}
+  \displaystyle{ \sum\limits_{l=1}^{\b-k} \frac{w(\b-l)}{w(\b)} 
\Delta^*_\alpha(l) H^*_l(\b-l,k)} \qquad \forall\, \a\geq 1,\b\geq 2,\, 1\le k\le \beta-1,
\end{array}
\right.
\end{equation}
the rates $g^{\a}_{*,\b}$, for all $\a\geq 1$, $\b > 1$, are given by
\begin{equation}\label{target-rates-cc}
\displaystyle{g^{\a}_{*,\b} = g^{\a}_{*,0} + \sum\limits_{k=1}^{\b} H^*_{\a}(\b,k) g^{k}_{*,0}}\,.
\end{equation}
Then, provided these rates satisfy \eqref{un-assum}, 
there is a one-parameter family of product invariant probability measures 
$\{\bar{\mu}_\varphi,\vfi\in {\rm Rad}(Z') \}$ or $\{\bar{\mu}_{\varphi}: \varphi\in {\rm Rad}(Z) \}$  
with single site marginal given by, for $\tilde\varphi_\mu$ defined in \eqref{tildefimu},
\begin{eqnarray}\label{inv-tar}
 \mu_\varphi(\a) & =& \dfrac{\mu(0)}{Z_\varphi} (\varphi\tilde\varphi_\mu)^\a \, w(\a) \ 
  \text{ for }\a\geq 0,\\\label{inv-tar2}
 Z_{\varphi}&=&\mu(0)\sum_{n\ge 0} (\varphi\tilde\varphi_\mu)^{n}w(n).
  \end{eqnarray}
Note that both (a) and (b) hold independently of  $p(\cdot,\cdot)$. 
If we assume $p(\cdot,\cdot)$ asymmetric  then 
conditions  \eqref{eq:weak-phil}--\eqref{target-rates-cc} on rates are   also necessary for 
the existence of product invariant probability measures.
\end{proposition} 
\begin{remark}\label{rk:tp-with-attr}
(a) In Proposition \ref{target-exact}(b), if  $g^k_{*,1}\ge  g^k_{*,0}$ for $k\ge 2$, 
then \eqref{eq:weak-phil}
is satisfied, but the process will not be attractive unless $g^k_{*,1}=g^k_{*,0}$, see 
\eqref{atr:GTP-1} in Lemma \ref{lemma:attract} below. Similarly, 
in Proposition \ref{target-exact}(a), if 
$\mu(n+1)/\mu(n)$ is strictly nondecreasing, the process will not be attractive, see 
Proposition \ref{no-attract-cond-single-jump-and-mm-tp} below. \par
(b) Taking $g^k_{*,1}= g^k_{*,0}=0$ for $k\ge 2$ in Proposition \ref{target-exact}(b)
yields $g^{\a}_{*,\b}=0$, for all $\a>1,\b\ge 0$ by
\eqref{target-rates-cc}, and
\begin{eqnarray*}
\Delta^*_1(1)&=& g^1_{*,1}- g^1_{*,0}\quad;\quad \Delta^*_1(l)=0\quad\forall\,l\ge 2\\
H^*_1(\b,1) &=& \frac{g^1_{*,1}- g^1_{*,0}}{g^1_{*,0}}\quad\forall\,\b\ge 1\quad\text{\rm thus}\\
g^1_{*,\b} &=& g^1_{*,0} +  H^*_1(\b,1) g^1_{*,0}= g^1_{*,1}\quad\forall\,\b\ge 1,
\end{eqnarray*}
that is, we recover \eqref{compat-tp-asym}.
\end{remark}
Therefore an asymmetric MM-TP has the set of all product, 
translation-invariant, invariant measures either empty 
(when rates do not satisfy \eqref{inv-tt}--\eqref{tp:hrec}
 for any $\mu$) or given explicitly as the set 
$\{\bar{\mu}_{\varphi}: \varphi\in {\rm Rad}(Z') \}$ or $\{\bar{\mu}_{\varphi}: \varphi\in {\rm Rad}(Z) \}$ 
where $\bar{\mu}_{\varphi}$ are product measures with marginals given by \eqref{inv-tar}.
 
Note that for a symmetric MM-TP, the set of all product, translation-invariant,
invariant measures is in general bigger. In addition to the measures $\bar{\mu}_{\varphi}$ 
with single site marginal 
\eqref{inv-tar} when \eqref{inv-tt}--\eqref{tp:hrec} are satisfied, there may exist 
other product translation-invariant measures satisfying  \eqref{inv-sym22} 
but not~\eqref{inv-tt}--\eqref{tp:hrec}.\\

\section{Proofs for Section \ref{s-invariant}}\label{s-proofs-sec3}

\subsection{Proofs for Subsection \ref{subsec:cns-product}}
Hereafter, we prove Theorem \ref{th:main2}, Remark \ref{rk:alternative-ax},  
and Corollaries \ref{cor:main-alternative} and \ref{family}. 
We first prove the following  lemma.

\begin{lemma}\label{lem:properties-Amu-muA}
Under the assumptions of Theorem \ref{th:main2}, for all $\a,\b\in\N $,  
the three quantities   $\sum_{\g\ge 0}  |\mathbf{A}(\a,\g)| \mu(\g)$, 
$\sum_{\g\ge 0}  |\mathbf{A}(\g,\b)| \mu(\g)$,
and
$\sum_{\a\geq 0}\sum_{\g\geq 0}  |\mathbf{A}(\a,\g) | \mu(\g)\mu(\a)$ are finite.
Moreover
if \eqref{inv-psi2} is satisfied, we have
\begin{equation}\label{31bis}
\sum\limits_{\b\geq 0} |\psi(\b)| \mu(\b)   < \infty .
\end{equation}
\end{lemma}
\begin{proof}
We define the set 
\begin{equation}\label{def:Smu}
S_\mu = \{\a\in \N, \mu(\a)\not=0\}.
\end{equation}
For $N$ fixed, we deal simultaneously with the term 
$\sum_{\g\le N} | \mathbf{A}(\a,\g)| \mu(\g)$ for a given value of $\a\in\N$  with $\a\in S_\mu$ 
and for its  sum over $\a\le N$, that is $\sum_{\a\le N}\sum_{\g\le N} | \mathbf{A}(\a,\g)| \mu(\g)\mu(\a)$.
 We omit a similar proof for $\sum_{\g\ge 0}  |\mathbf{A}(\g,\b)| \mu(\g)$.
 \begin{eqnarray}\label{les-2-sommes-a}
&&\sum_{\g\le N}  |\mathbf{A}(\a,\g)| \mu(\g) \le 
\frac{1}{\mu(\a)} \sum_{\g\le N}  \sum_{k\leq \g}   g^k_{\a+k,\g-k} \mu(\a+k)\mu(\g-k)
 + \sum_{\g\le N}\sum_{k\leq \a} g^k_{\a,\g} \mu(\g) \label{les-2-sommes-a-abs}\\
&&\sum_{\a\le N}\sum_{\g\le N}  |\mathbf{A}(\a,\g) |\mu(\g)\mu(\a)\le
\sum_{\a\le N}\sum_{\g\le N}  \sum_{k\leq \g} g^k_{\a+k,\g-k}\mu(\a+k)\mu(\g-k) \nonumber\\
&&\qquad\qquad\qquad\qquad\qquad\qquad
 + \sum_{\a\le N}\sum_{\g\le N}\sum_{k\leq \a} g^k_{\a,\g}\mu(\a)\mu(\g).\label{les-2-sommes-b-abs}
\end{eqnarray}  
 We  have by \eqref{un-assum}
\begin{eqnarray}\label{2eme-partie-a}
\sum_{\g\le N}\mu(\g)\sum_{k\leq \a} g^k_{\a,\g}&\le& \sum_{\g\le N}\mu(\g) C(\a+ \g)\le C(\a+ \|\mu\|_1)\\
\sum\limits_{\a\le N}\mu(\a)\sum\limits_{\g\le N}\mu(\g)\sum\limits_{k\leq \a} g^k_{\a,\g}&\le&
\sum\limits_{\a\le N}\mu(\a)C(\a+ \|\mu\|_1)\le 2C\|\mu\|_1.
\label{2eme-partie-b}
\end{eqnarray} 
We have also
 \begin{eqnarray}
\sum_{\g\le N}  \sum_{k\leq \g}   g^k_{\a+k,\g-k} \mu(\a+k)\mu(\g-k) &= &
\sum_{k\leq N}\sum_{k\leq \g\leq N}  g^k_{\a+k,\g-k} \mu(\a+k)\mu(\g-k)\nonumber\\ &\leq&
\sum_{k\leq N} \sum_{\a\le N}\sum_{k\leq \g\le N}  g^k_{\a+k,\g-k}\mu(\a+k)\mu(\g-k)
\label{eq:estimee-compliquee-2}
\end{eqnarray} 
 and finally 
 \begin{eqnarray}
\sum_{k\leq N} \sum_{\a\le N}\sum_{k\leq \g\le N}  g^k_{\a+k,\g-k}\mu(\a+k)\mu(\g-k)
 & = &
\sum_{k\leq N} \ \sum_{k\leq \a^* \le k+N}\ \sum_{\g^*\le N-k}  g^k_{\a^*,\g^*}\mu(\a^*)\mu(\g^*) \nonumber\\
& \leq &
\sum_{k\leq N}\ \sum_{k\le\a^* \le 2N}\ \sum_{0\le \g^*\le N} g^k_{\a^*,\g^*}\mu(\a^*)\mu(\g^*)\nonumber
\\  & \leq & 
\sum_{0\le\a^* \le 2N}\ \sum_{0\le \g^*\le N}\ \sum_{k\leq \a^*} g^k_{\a^*,\g^*}\mu(\a^*)\mu(\g^*)\nonumber\\
&\leq& \sum_{0\le\a^* \le 2N}\ \sum_{0\le \g^*\le N} \mu(\a^*)\mu(\g^*)C(\a^*+ \g^*)\nonumber\\ &\leq&
2C\|\mu\|_1.
\label{eq:estimee-compliquee-1}
\end{eqnarray} 
Combining \eqref{les-2-sommes-a}--\eqref{eq:estimee-compliquee-1} we get
 \begin{eqnarray}\label{A-integrable}
\sum_{\g\le N}  |\mathbf{A}(\a,\g)| \mu(\g)&\leq& 2\frac{C}{\mu(\a)} \|\mu\|_1  + C(\a+ \|\mu\|_1)\\
\sum_{\a\le N}\sum_{\g\le N}  |\mathbf{A}(\a,\g) |\mu(\g)\mu(\a)&\le &4C\|\mu\|_1\label{A-integrable-bis}
\end{eqnarray} 
and we can let  $N\to+\infty$ to conclude that for any $\a\geq 0$  
with $\a\in S_\mu$,  
 $\sum_{\g\ge 0}  |\mathbf{A}(\a,\g)| \mu(\g)$ is finite, and so is
$\sum_{\a\geq 0}\sum_{\g\geq 0}  |\mathbf{A}(\a,\g) | \mu(\g)\mu(\a)$.

Finally, if  $\psi$ is a function on $\N$ such that 
$\psi(\beta)-\psi(\a) =  \mathbf{A}(\a,\b)$ for all $\a,\b\in S_\mu$, then for every $\a\in S_\mu$,
\begin{equation*}
\sum\limits_{\b\geq 0} |\psi(\b)| \mu(\b)  
\leq  \sum\limits_{\b\geq 0} |\mathbf{A}(\a,\b)| \mu(\b) + |\psi(\a)| < \infty 
\end{equation*}
and \eqref{31bis} is also satisfied.
\end{proof}

\begin{proof} \textit{of Theorem \ref{th:main2}.}

\noindent
{\bf Step 1.} 
We  first note that, using \eqref{sum-a_x-finite}, 
\begin{equation}\label{int-eta-when-product}
\int_\XX \|\eta \| \ d\bar{\mu}(\eta) \leq \int \|\eta \| \ d\bar{\mu}(\eta)=\sum_{x\in\X}a_x\int \eta(x) d\mu(\eta(x))=\left(\sum_{x\in \X} a_x\right)\left(\sum_{n\ge 0} 
n \mu(n)\right).
\end{equation}
Hence,  if $\mu$ has a finite first moment, then  $\bar{\mu}(\XX)=\bar{\mu}(\|\eta\|<\infty)=1$ is satisfied. \par
\medskip
Now, as a consequence of Proposition \ref{intL=0}, Lemmas \ref{L-n-conv} 
and \ref{bound1}, the probability measure $\bar{\mu}$ is invariant 
 if and only if  for every bounded cylinder function $f$ on $\N^{\Zd}$,
\begin{equation}\label{inv}
0= \int \L f (\eta)\,d\bar{\mu}(\eta) =\lim_{n\rightarrow\infty}\int \L_n f (\eta)\,d\bar{\mu}(\eta),
\end{equation}
where $ \L_n $ is the finite volume approximation of  $\L$ defined by \eqref{ngenerator}.\par
\medskip
Let $f$ be a bounded cylinder function, and denote by $V_f\in \Zd$ its finite support,
 and by $m_f$ an upper bound for $|f|$.
We want to derive an expression for $\int \L_n f (\eta)\,d\bar{\mu}(\eta)$, 
involving $ \mathbf{A}(\cdot,\cdot)$ defined in \eqref{A-def2}. We have 
\begin{eqnarray}
\int \L_n f (\eta)\,d\bar{\mu}(\eta) 
&=& \!\!\!\!\!\!\!\!\sum_{x,y\in \X_n,x\neq y,\{x,y\}\cap V_f\not=\emptyset} p(x,y)
 \int \sum_{k>0} g^k_{\eta(x),\eta(y) } (f(\etb) -f(\eta))\,d\bar{\mu}(\eta).
 \label{aux6}
\end{eqnarray}
For further use, note that 
since the probability measure $\mu$ has a finite first moment $\|\mu\|_1$, we have  by \eqref{un-assum},
for $x,y\in \X_n,x\neq y$ 
\begin{equation}\label{eq:2sommes-cv}\left.
  \begin{array}{l}
  \int\sum\limits_{k>0 } g^k_{\eta(x),\eta(y) }| f(\etb)| \,d\bar{\mu}(\eta)  \\[2mm]
  \int\sum\limits_{k>0} g^k_{\eta(x),\eta(y) }  |f(\eta)| \,d\bar{\mu}(\eta) 
  \end{array}
  \right\}
\leq m_f C\int (\eta(x)+\eta(y)) \,d\bar{\mu}(\eta) 
\leq 2m_f C \|\mu\|_1 .
\end{equation}\\
\noindent
{\bf Step 2.}
We now prove that condition \eqref{inv-sym21} is necessary.
Let $\a_0$, $\b_0$ such that $\mu(\a_0)\mu(\b_0)= 0$ and consider the function 
$f_0(\eta) = 1_{\{\eta(x_0)=\a_0,\eta(y_0)=\b_0\}}$ for some $x_0$, $y_0$ with $p(x_0,y_0)\not=0$.
If $\bar{\mu}$ is an invariant product measure, we have, since $\mu(\a_0)\mu(\b_0)= 0$,
\begin{eqnarray*}
0 &=& \int \L_n f_0 (\eta)\,d\bar{\mu}(\eta)\\
&=& \int \sum_{x,y\in \X_n,x\neq y} p(x,y) \sum_{k>0} g^k_{\eta(x),\eta(y) } (f_0(S^k_{x,y}\eta) -f_0(\eta))\,d\bar{\mu}(\eta)\\
&=& \int \sum_{x,y\in \X_n,x\neq y} p(x,y) \sum_{k>0} g^k_{\eta(x),\eta(y) } f_0(S^k_{x,y}\eta) \,d\bar{\mu}(\eta).
\end{eqnarray*}
The last line is a sum of non negative terms, thus each term is zero and in particular the one for $x=x_0$, $y=y_0$, which reads
\begin{equation*}
\sum_{k\le \b_0} g^k_{\a_0+k,\b_0-k } \mu(\a_0+k)\mu(\b_0-k) = 0
\end{equation*}
that is Condition  \eqref{inv-sym21}.\\ \\
{\bf Step 3.}
For $x,y\in \X_n,x\neq y,\{x,y\}\cap V_f\not=\emptyset$, thanks to \eqref{eq:2sommes-cv}, 
the integral in the right hand side of \eqref{aux6} reads
\begin{eqnarray} 
&&\int \sum_{k>0} g^k_{\eta(x),\eta(y) } (f(\etb) -f(\eta))\,d\bar{\mu}(\eta) \nonumber\\
&&\qquad= \sum_{k>0} \int_{\{\eta(x)\ge k\}}  g^k_{\eta(x),\eta(y) } f(\etb) \,d\bar{\mu}( \eta)
 -  \sum_{k>0} \int_{\{\eta(x)\ge k\}}   g^k_{\eta(x),\eta(y) } f(\eta) \,d\bar{\mu}(\eta)
\nonumber\\
&&\qquad= \sum_{k>0} \int_{\{\eta(y)\ge k\}}  g^k_{\eta(x)+k,\eta(y) -k} f(\eta) \,d\bar{\mu}(S_{y,x}^k \eta)
 -  \sum_{k>0} \int_{\{\eta(x)\ge k\}}   g^k_{\eta(x),\eta(y) } f(\eta) \,d\bar{\mu}(\eta) \nonumber\\
 &&\qquad= \sum_{k>0} \int_{\{\eta(y)\ge k\}}  1_{\{\eta(x)\in S_\mu\}} 1_{\{\eta(y)\in S_\mu\}}  
  g^k_{\eta(x)+k,\eta(y)
  -k} f(\eta) \,d\bar{\mu}(S_{y,x}^k \eta)\nonumber\\
&&\qquad\qquad
 -  \sum_{k>0} \int_{\{\eta(x)\ge k\}}   g^k_{\eta(x),\eta(y) } f(\eta) \,d\bar{\mu}(\eta) \nonumber\\
  &&\qquad= \sum_{k>0} \int_{\{\eta(y)\ge k\}}  1_{\{\eta(x)\in S_\mu\}} 1_{\{\eta(y)\in S_\mu\}} 
   g^k_{\eta(x)+k,\eta(y) -k} f(\eta)D_\mu^k(\eta(x),\eta(y)) \,d\bar{\mu}(\eta)\nonumber\\
&&\qquad\qquad
 -  \sum_{k>0} \int_{\{\eta(x)\ge k\}}   g^k_{\eta(x),\eta(y) } f(\eta) \,d\bar{\mu}(\eta) \nonumber\\
 &&\qquad= \int 1_{\{\eta(x)\in S_\mu\}} 1_{\{\eta(y)\in S_\mu\}}  
 \bigl(  \sum_{k=1}^{\eta(y)}   g^k_{\eta(x)+k,\eta(y) -k} 
 D_\mu^k(\eta(x),\eta(y))  -  \sum_{k=1}^{\eta(x)} g^k_{\eta(x),\eta(y) } \bigr)  f(\eta) \,d\bar{\mu}(\eta) 
\nonumber\\
&&\qquad=  \int 1_{\{\eta(x)\in S_\mu\}} 1_{\{\eta(y)\in S_\mu\}}  
\mathbf{A}(\eta(x),\eta(y)) f(\eta) \,d\bar{\mu}(\eta)\nonumber\\
&&\qquad=  \int  \mathbf{A}(\eta(x),\eta(y)) f(\eta) \,d\bar{\mu}(\eta)\label{arranged-intLf}
\end{eqnarray}
 where we used Condition \eqref{inv-sym21} in the third equality and introduced in the next one the quantity
 \begin{eqnarray}
 D_{\mu}^k (\alpha,\beta)=
 \begin{cases}
 \displaystyle{\frac{\mu(\alpha + k)\mu(\b -k)}{\mu(\a) \mu(\b)}} &\hbox{ if } k\le \b \hbox{ and }\mu(\a) \mu(\b)\not=0, \cr
 1& \hbox{ if }\mu(\a) \mu(\b)=0.
 \end{cases}
\end{eqnarray}
 \noindent
{\bf Step 4.}
We prove  \textit{necessity} of condition (\ref{inv-psi2}) and  
condition (\ref{inv-sym22}), respectively for the asymmetric  and  symmetric cases.
 Let us denote 
\begin{equation}\label{Amu-muA}
  (\mathbf{A}\mu)(\a) = \sum_{\g\ge 0}  \mathbf{A}(\a,\g) \mu(\g),\qquad    (\mu\mathbf{A})(\b) = \sum_{\g\ge 0}  \mathbf{A}(\g,\b) \mu(\g).
\end{equation} 
Using \eqref{arranged-intLf}, we have for any $n$ large enough so that
$V_f \subset \X_n $
\begin{eqnarray}
\int \L_n f (\eta)\,d\bar{\mu}(\eta)  
         &=&  \sum_{x,y\in V_f,x\neq y} p(x,y) 
         \int \mathbf{A}(\eta(x),\eta(y)) f(\eta)\ d\bar{\mu}(\eta)\nonumber\\
&\qquad&+\sum_{x\in V_f,y\in \X_n\setminus V_f} p(x,y) 
\int \Bigl(\sum_{\b\ge 0}\mathbf{A}(\eta(x),\b) \mu(\b) \Bigr) f(\eta)\ d\bar{\mu}(\eta)\nonumber\\
&\qquad&+\sum_{x\in \X_n\setminus V_f,y\in V_f} p(x,y) 
\int \Bigl(\sum_{\a\ge 0}\mathbf{A}(\a,\eta(y))\mu(\a) \Bigr) f(\eta)\ d\bar{\mu}(\eta)\nonumber\\
 &=&  \sum_{x,y\in  V_f,x\neq y} p(x,y) \int \mathbf{A}(\eta(x),\eta(y)) f(\eta)\ d\bar{\mu}(\eta)\nonumber\\
&\qquad&+\sum_{x\in V_f,y\in \X_n\setminus V_f} p(x,y) 
\int (\mathbf{A}\mu)(\eta(x))  f(\eta)\ d\bar{\mu}(\eta)\nonumber\\
&\qquad&+\sum_{x\in \X_n\setminus V_f,y\in V_f} p(x,y) 
\int (\mu\mathbf{A})(\eta(y)) f(\eta)\ d\bar{\mu}(\eta).
\label{ineq:pour1et2}
\end{eqnarray}
In the first equality, we have split the sum over $x$ and $y$ in three terms, 
depending on whether they belong to the support of $f$, and used the fact that $\bar \mu$ 
is a product measure. In the second equality, we introduced notations \eqref{Amu-muA}.

Consider $f(\eta)=\1_{[\eta(x_0)=\a]}$ for some fixed $x_0\in \Zd$ and $\a \in S_\mu$. 
We get in this case from \eqref{ineq:pour1et2},
\begin{equation*}
\int \L_n f (\eta)\,d\bar{\mu}(\eta)  
         = \sum_{y\in \X_n\setminus \{x_0\}}  
         \Bigl( p(x_0,y) (\mathbf{A}\mu)(\a)  + p(y,x_0) (\mu\mathbf{A})(\a) \Bigr) \mu(\a) .
\end{equation*}
Taking the limit $n\to+ \infty$ gives by \eqref{inv}
\begin{equation}\label{r2}
 (\mathbf{A}\mu)(\a) + (\mu\mathbf{A})(\a)=0.
\end{equation}
Now consider $f(\eta)=\1_{[\eta(x_0)=\a,\eta(y_0)=\b]}$ for some fixed $x_0$, $y_0$ in $\Zd$, $x_0\not= y_0$ 
and $\a, \b \in S_\mu$. We get from \eqref{ineq:pour1et2},
\begin{eqnarray}
\int \L_n f (\eta)\,d\bar{\mu}(\eta)  
 =\mu(\a) \mu(\b)&\Bigl(&p(x_0,y_0) \mathbf{A}(\a,\b) + p(y_0,x _0) \mathbf{A}(\b,\a) \nonumber\\
&&+\sum_{z\in \X_n\setminus \{x_0,y_0\}}\Bigl( p(x_0,z) (\mathbf{A}\mu)(\a)  + p(z, x_0) (\mu\mathbf{A})(\a) \Bigr)\nonumber\\
&&+\sum_{z\in \X_n\setminus \{x_0,y_0\}}\Bigl( p(y_0,z) (\mathbf{A}\mu)(\b)  + p(z, y_0) (\mu\mathbf{A})(\b) \Bigr) \; \Bigr).\nonumber
\end{eqnarray}
Taking the limit $n\to+ \infty$ gives by \eqref{inv}
\begin{eqnarray}
\mu(\a) \mu(\b)&\Bigl(& p(x_0,y_0) \mathbf{A}(\a,\b) + p(y_0,x _0) \mathbf{A}(\b,\a)  \nonumber\\
&&+( 1 - p(x_0,y_0) ) \bigl(  (\mathbf{A}\mu)(\a) + (\mu\mathbf{A})(\b) \bigr) \nonumber\\
&& +( 1 - p(y_0,x_0) ) \bigl( (\mu\mathbf{A})(\a)   + (\mathbf{A}\mu)(\b)  \bigr) \Bigr) = 0.
\label{r1}
\end{eqnarray}
Using \eqref{r2} and the fact that $\mu(\a) \mu(\b)\not=0$, \eqref{r1} becomes
\begin{eqnarray}
&& p(x_0,y_0) \Bigl(\mathbf{A}(\a,\b) + (\mu\mathbf{A})(\a)  - (\mu\mathbf{A})(\b) )\Bigr) \nonumber\\
&&+p(y_0,x_0) \Bigl(  \mathbf{A}(\b,\a) + (\mu\mathbf{A})(\b)   - (\mu\mathbf{A})(\a)\Bigr) = 0.
\label{r3}
\end{eqnarray}
The same relation being valid under exchange of $x_0$ and $y_0$, both the symmetric and antisymmetric 
parts of equation \eqref{r3} are separately zero.
The symmetric part reads
\begin{equation}\label{f1}   
\Bigl(p(x_0,y_0) +p(y_0,x_0)  \Bigr) \Bigl(  \mathbf{A}(\a,\b) +  \mathbf{A}(\b,\a) \Bigr) =0
\end{equation}
which implies  \eqref{inv-sym22} since there is some $(x_0,y_0)$ such that $p(x_0,y_0) +p(y_0,x_0) \not=0$.

\noindent Assuming  \eqref{inv-sym22}, the antisymmetric part gives
\begin{equation*}\label{r4}
 \left( p(x_0,y_0)-p(y_0,x_0)  \right) \Bigl( \mathbf{A}(\a,\b)  -(\mu\mathbf{A})(\b)  + (\mu\mathbf{A})(\a)\Bigr)  =0. 
\end{equation*}
If $p(\cdot,\cdot)$ is symmetric, this equation is identically zero for any choice of $\mathbf{A}$,
and we are left with \eqref{inv-sym22} alone.  
If $p(\cdot,\cdot)$ is asymmetric, there is a choice of $(x_0,y_0)$ such that $p(x_0,y_0)\neq p(y_0,x_0)$  
and we obtain 
\begin{equation}\label{f3}
 \mathbf{A}(\a,\b) =  \psi(\b) - \psi(\a) .
\end{equation} 
Since \eqref{f3} defines $\psi$ up to a constant, 
we may take $\psi(0)=0$ and write  $\psi(\cdot)$ in terms of $\mathbf{A}$ as
 \begin{equation}\label{f4}
 \psi(\g) = (\mu\mathbf{A})(\g) - (\mu\mathbf{A})(0) .
 \end{equation} 
\noindent
{\bf Step 5.}
Let us show now that conditions \eqref{inv-sym21}--\eqref{inv-sym22}  
and conditions \eqref{inv-sym21}--\eqref{inv-psi2}  \textit{are sufficient} for $\bar{\mu}$ to be invariant, 
respectively in the symmetric  and in the asymmetric  case. \\[-3mm]

Assume first $p(x,y)=p(y,x)$ for all $x,y$ and conditions \eqref{inv-sym21}--\eqref{inv-sym22} satisfied.
Equation \eqref{arranged-intLf} gives
\begin{eqnarray} 
 \int \L_n f (\eta)\ d\bar{\mu} (\eta)
 &=&  \frac{1}{2} \int  \sum_{x,y\in\X_n,x\neq y,\{x,y\}\cap V_f\not=\emptyset} 
p(x,y) \,\bigl(\mathbf{A}(\eta(x),\eta(y))+ 
\mathbf{A}(\eta(y),\eta(x))\bigr) f(\eta)\ d\bar{\mu} (\eta) \nonumber\\
\label{eq:cs-main-sym} &=& 0 .
\end{eqnarray}

Assume now $p(.,.)$ asymmetric and conditions \eqref{inv-sym21}--\eqref{inv-psi2} satisfied.
Thus, recalling that $f$ is a cylinder function, and using that by \eqref{31bis}, $\psi$ is integrable, we have 
\begin{equation}\label{psif2}\
  \sum_{x,y\in \X_n,\{x,y\}\cap V_f\not=\emptyset} p(x,y)
\int\Bigl|  \psi(\eta(y)) f(\eta)\Bigr|d\bar{\mu}(\eta)
  \le  m_f |V_f| \sum_{\a\in\N} |\psi(\a)| \mu(\a) < \infty
\end{equation}
then
\begin{eqnarray*}
 \int \L_n f (\eta)\ d\bar{\mu}(\eta) &=&   \int\sum_{x,y\in \X_n,\{x,y\}\cap V_f\not=\emptyset} p(x,y) \Bigl( \psi(\eta(y))-\psi(\eta(x)) \Bigr) f(\eta)\ d\bar{\mu}(\eta) \\
&=&   \int\sum_{x,y\in \X_n,\{x,y\}\cap V_f\not=\emptyset} 
(p(x,y) - p(y,x)) \psi(\eta(y)) f(\eta)\ d\bar{\mu}(\eta) .
\end{eqnarray*}
For every $n$ such that $V_f \subset \X_n $,
\begin{eqnarray*}
 \int \L_n f (\eta)\ d\bar{\mu}(\eta) &=&
 \sum_{y\in  V_f} \sum_{x\in \X_n} \bigl(p(x,y)-p(y,x)\bigr)  \int \psi(\eta(y)) f(\eta)\ d\bar{\mu} (\eta)
 \\ &\makebox{} & + 
  \sum_{y\in\X_n\setminus V_f} \sum_{x\in V_f}  \bigl(p(x,y)-p(y,x)\bigr)  
 \Bigl( \sum_{\a\in\N} \psi(\a) \mu(\a) \Bigr)\int f(\eta)\ d\bar{\mu} (\eta)\\
  &=&
   \sum_{y\in  V_f} \sum_{x\in \X_n} \bigl(p(x,y)-p(y,x)\bigr)  \int \psi(\eta(y)) f(\eta)\ d\bar{\mu} (\eta)
 \\ &\makebox{} & -
  \sum_{y\in V_f} \sum_{x\in \X_n\setminus V_f}  \bigl(p(x,y)-p(y,x)\bigr)  
 \Bigl( \sum_{\a\in\N} \psi(\a) \mu(\a) \Bigr)\int f(\eta)\ d\bar{\mu} (\eta)\\
  &=&
 \sum_{y\in  V_f}   \sum_{x\in \X_n} \bigl(p(x,y)-p(y,x)\bigr)  \int \Bigl(\psi(\eta(y)) - \sum_{\a\in\N} \psi(\a) \mu(\a) \Bigr)  f(\eta)\ d\bar{\mu} (\eta)\\
  &=&
 \sum_{y\in  V_f}   \sum_{x\not\in \X_n} \bigl(p(y,x)-p(x,y)\bigr)  \int \Bigl(\psi(\eta(y)) - \sum_{\a\in\N} \psi(\a) \mu(\a) \Bigr)  f(\eta)\ d\bar{\mu} (\eta).
\end{eqnarray*}
 In the first equality, we have split the sum on $y$ in two terms and used the fact that $\bar \mu$ is a product measure. In the second equality,
 we have exchanged the roles of $x$ and $y$ in the second term. The third equality uses that
\begin{equation*}
  \sum\limits_{y\in V_f}\sum\limits_{x\in V_f}  (p(x,y) -  p(y,x)) = 0.
\end{equation*}
In the fourth equality, we used the fact that $p(\cdot,\cdot)$ is bistochastic, 
\begin{equation*}
  \sum\limits_{x\in \X_n} (p(x,y) -  p(y,x)) =  \sum\limits_{x\not\in \X_n} (p(y,x) -  p(x,y) ).
\end{equation*}
Thus, using \eqref{psif2}, we have 
\begin{eqnarray*}
 \Bigl| \int \L_n f (\eta)\ d\bar{\mu}(\eta)\Bigr| \le  2 \sum_{y\in  V_f} \Bigl|\sum_{x\not\in \X_n}  \bigl(p(x,y) -  p(y,x)\bigr) \Bigr| m_f   \sum_{\a\in\N} |\psi(\a)| \mu(\a)
\end{eqnarray*} 
which converges to 0 when $n\to+\infty$.\\
Thus equation \eqref{inv} is satisfied.
\end{proof}
\smallskip 

\begin{proof}\textit{ of Remark \ref{rk:alternative-ax}.}
Define, for $x\in\X$,   $a^*_x=a_x/f^{-1}(\|x\|_1)$ where $f^{-1}$ is 
the inverse function of $f$ and 
 $\|x\|_1=\sum_{i=1}^d |x_i|$. Define, for all $x\in\X,k\in\N$,
the events 
$$A_x=\{\eta(x)\geq f^{-1}(\|x\|_1)\} = \{ \eta(x)a^*_x\geq a_x\},\quad 
B_k=\{\eta(0)\geq f^{-1}(k)\}.$$
 Since $\bar{\mu}$ is translation-invariant, 
$\bar{\mu}(A_x)  = \bar{\mu}\left(\eta(0)\geq f^{-1}(\|x\|_1)\right) = \bar{\mu}(B_k)$ 
whenever  $\|x\|_1 =k$. Then we have
$$ \sum_{k\in\N} \bar{\mu}(B_k)  = \sum_{k\in\N} \bar{\mu}(\eta(0)\geq f^{-1}(k) )
= \sum_{k\in\N} \bar{\mu}(f(\eta(0))\geq k)
=\sum_{n\in\N} f(n) \mu(n) <\infty.$$
It follows from  Borel-Cantelli Lemma  
that $\bar{\mu}(B_k \ \text{holds for infinitely many } k\in\N)=0$. 
Therefore also $\bar{\mu}(A_x \ \text{holds for at most finitely many } x\in\Zd) =1$. 
But it means that \\
 $\bar{\mu}(\sum_{x\in\Zd} \eta(x)a^*_x <\infty)=1$. 
\end{proof}
\begin{proof} \textit{of Corollary \ref{cor:main-alternative}}.
Note first that in the proof of Lemma \ref{lem:properties-Amu-muA},
 replacing each bound involving the first moment 
of $\mu$ by the bound \eqref{assum-alternative} yields the results of the lemma.\par 
We now explain where and how to modify the proof of Theorem \ref{th:main2}
to obtain Corollary \ref{cor:main-alternative} under assumptions (\ref{A1})--(\ref{assum-alternative}).
By assumption (\ref{A1}), the first remark in Step 1 is useless. The first equality in
\eqref{inv} comes from assumption (\ref{A1}); to derive the second one
thanks to assumption (\ref{assum-alternative}),
we need to go into the proofs of Subsections \ref{ex1}, \ref{ex4}.

Let $f$ be a bounded cylinder function, where $V_f$ denotes the support of $f$ and $m_f$ 
an upper bound for $|f|$. Let $\eta\in\N^X$.
Using \eqref{assum-alternative}, we replace the computation \eqref{eq:bound1} done in the proof of
Lemma \ref{bound1} by
\begin{eqnarray}\nonumber
|{\L}_{n} f(\eta) | \!&\!\leq \!&\!  \sum_{x,y\in\X_n,x\not=y,\{x,y\}\cap V_f\not=\emptyset} 
 p(x,y) \sum_{k=1}^{\eta(x)} g^k_{\eta(x),\eta(y)} |f(\etb)-f(\eta) |  \\ \nonumber
 &\leq& 2m_f \sum_{x,y\in\X_n,x\not=y,\{x,y\}\cap V_f\not=\emptyset} 
 p(x,y) \sum_{k=1}^{\eta(x)} g^k_{\eta(x),\eta(y)}  \\
 &\leq& 2m_f\,C\,|V_f|.
 \label{eq:bound1-bis}
\end{eqnarray}
This computation \eqref{eq:bound1-bis} enables to adapt 
Lemmas \ref{bound1}, \ref{an2.1}, \ref{an2.4}, \ref{L-n-conv},
and \eqref{bound-to-generator} in Proposition \ref{prop:generator}, thus 
we obtain \eqref{inv}.

Then by assumption (\ref{assum-alternative}), the two terms considered in
\eqref{eq:2sommes-cv} are bounded directly by $m_f$, which enables to do Step 3.
No other modification is needed.
\end{proof} 
\begin{proof}\textit{of Corollary \ref{family}}.
For every $\varphi <\varphi_c$,  and for $\varphi =\varphi_c$ 
if $\sum_{n\ge 0} n \vfi_c^n {\mu}(n) <\infty $, if 
the rates of the MMP
  satisfy assumption \eqref{un-assum} and the
single site marginal $\mu$ of $\bar{\mu}$ has a finite first moment,
then  the measure  $\bar{\mu}_\varphi$ has a finite first moment,
and Theorem \ref{th:main2} 
may be applied to $\bar{\mu}_\varphi$. 

For every $\varphi <\varphi_c$,  and for $\varphi =\varphi_c$
if  $\sum_{n\ge 0} n \vfi_c^n {\mu}(n)=+\infty $ with  $Z_{\varphi_c}<\infty$, 
if the rates of the MMP
  satisfy assumptions (\ref{A1})--(\ref{assum-alternative}), then
   Corollary \ref{cor:main-alternative}
may be applied to $\bar{\mu}_\varphi$.

It follows from the definition \eqref{muf} of $\bar{\mu}_\varphi$ that  
the quantities $\mathbf{A}(.,.)$ defined in  
\eqref{A-def2}
coincide for both $\bar{\mu}$ and  $\bar{\mu}_\varphi$  for all $a,\b\in\N$.
Since the necessary and sufficient conditions given in 
Theorem \ref{th:main2} are expressed through these quantities, 
they are thus satisfied for $\bar{\mu}_\varphi$ as soon as they are satisfied for $\bar{\mu}$.
If $\bar{\mu}$ is invariant for the MMP, then these necessary and sufficient conditions 
are satisfied.

Thus, by Theorem \ref{th:main2} or by Corollary \ref{cor:main-alternative} for $\varphi<\varphi_c$,
and, for $\varphi =\varphi_c$, either by Theorem \ref{th:main2} 
if  $\sum_{n\ge 0} n \vfi_c^n {\mu}(n) <\infty $ under assumption \eqref{un-assum},
or by Corollary \ref{cor:main-alternative} if
 $\sum_{n\ge 0} n \vfi_c^n {\mu}(n)=+\infty $ with  $Z_{\varphi_c}<\infty$
 under assumptions (\ref{A1})--(\ref{assum-alternative}),
 we have that $\bar{\mu}_{\varphi}$ is also invariant for the process.
\end{proof}
%
\subsection{Proofs for Subsection \ref{subsec:Applications}}
\begin{proof} \textit{of Proposition \ref{gzrp}}.
For MM-ZRP, we have from \eqref{A-def2}, 
$$
\mathbf{A}(\a,\b) = \sum_{k\leq \b}\frac{\mu(\a+k)\mu(\b-k)}{\mu(\a)\mu(\b)} g^k_{\a+k} - \sum_{k\leq \a}g^k_{\a} 
\quad \text{ and }\ \mathbf{A}(\a,0) = - \sum_{k\leq \a}g^k_{\a}. $$
Then the invariance condition \eqref{inv-psi2} writes,  using \eqref{inv-psi-bis}, 
\begin{equation}\label{inf-psi-d}   
\sum_{k\leq \b}\frac{\mu(\a+k)\mu(\b-k)}{\mu(\a)\mu(\b)} g^k_{\a+k} = \sum_{k\leq \b}g^k_{\b}    \qquad \forall \a,\b\geq 0.
\end{equation}
To prove that condition \eqref{inv-s} implies \eqref{inf-psi-d} is straightforward. 
For the  converse, let us denote $\phi_{\a}^k=g_{\a}^k\mu(\a)$, so \eqref{inf-psi-d} writes  
\begin{equation}\label{inv-s-fi} 
\sum_{k\leq \b} \mu(\b-k)\phi^k_{\a+k} =\mu(\a)\sum_{k\leq \b}\phi^k_{\b} \qquad \forall \a,\b\geq 0.
\end{equation}
Taking $\b=1$  in \eqref{inv-s-fi} gives
\begin{equation}\label{inv-s-fi-1} 
\phi_{\a+1}^1= \frac{\mu(\a)}{\mu(0)}\phi_1^1\quad\text{ for }\a\geq 1.
\end{equation}
  We prove by induction on $(n,i)$, with $0< i <n$, that we have
\begin{equation}\label{indukce} 
\phi^{n-i}_{n} = \frac{\mu(i)}{\mu(0)}\phi_{n-i}^{n-i} \qquad \text{for all }\ 0< i <n,\,n\geq 2,
\end{equation} 
that is, \eqref{inv-s} after the change of variables $n=\a+k,\,i=\a$.
 To initiate the induction for $n=2,i=1$,
we choose  $\b=\a=1$ in \eqref{inv-s-fi}. 
To complete the induction with $n>2,i=n-1$, we choose $\a=n-1$ in \eqref{inv-s-fi-1}.

For the induction step, let us fix $n>2,\,0< i <n$ and assume that \eqref{indukce}
holds for all $(n',i')$ such that $2\leq n'< n $ and $0< i'< n'$. 
Use \eqref{inv-s-fi}  for $\a = i$, $\b=n-i$ (that is, $\a+\b=n$). 
We obtain
$$ \sum_{k=1}^{n-i} \mu(n-i-k) \phi_{k+i}^k = \mu(i) \sum_{k=1}^{n-i} \phi_{n-i}^k\,. $$
After a change of variable $l=n-i-k$, we have 
\begin{equation}\label{eq:almost-ind} 
\mu(0) \phi_{n}^{n-i} + \sum_{l=1}^{n-i-1} \mu(l) \phi_{n-l}^{n-i-l} 
= \mu(i) \sum_{l=0}^{n-i-1} \phi_{n-i}^{n-i-l}\,.   
\end{equation}
Using the induction assumption for each term in the sum of the l.h.s. of \eqref{eq:almost-ind} yields 
%
$$ \mu(0) \phi_{n}^{n-i} = \mu(i) \phi_{n-i}^{n-i}\,. $$  
The induction is proved. 
\end{proof}
\smallskip
\begin{proof}\textit{of Theorem \ref{target}.} \\
We apply Theorem \ref{th:main2}  with Corollary \ref{cor:main-alternative}. 
  The product measure $\bar{\mu}$ 
is invariant for the MM-TP if and only if Conditions \eqref{inv-sym22} and \eqref{inv-psi2}
are satisfied. Formula \eqref{A-def2}  writes
\begin{equation} \label{A-target}
\mathbf{A}(\a,\b) = \sum_{k\le\b}\frac{\mu(\a+k)\mu(\b-k)}{\mu(\a)\mu(\b)} g^k_{*,\b-k}
 - \sum_{k\le\a}g^k_{*,\b},\qquad \text{for all } \a\geq 0,\b\geq 0.
\end{equation}
 Using Remark \ref{Aaa}, we set $\psi(0)=0$ and  equation \eqref{inv-psi-bis} reads
\begin{equation} \label{def-psi}
\psi(\a)=-\mathbf{A}(\a,0)= \sum_{k\le \a} g_{*,0}^k \ \text{for } \a>0.
\end{equation}
Thus the set of equations \eqref{inv-psi2}  is equivalent to
\begin{eqnarray} \label{inf-psi-t}  
 \mathbf{A}(0,\b) 
  &=&  \psi(\b)  \qquad \forall \b\geq 0 \ \ \& \\
  \mathbf{A}(\a,\b) - \mathbf{A}(\a-1,\b)  &=& \psi(\a-1) 
  - \psi(\a)  \qquad \forall \a>0, \b\geq 0.  \label{tp:grec0}
\end{eqnarray} 
We first consider equations (\ref{tp:grec0}).  Using \eqref{A-target}, \eqref{def-psi} 
 and notation \eqref{tp:Delta},  they can be written as,
for all $\a>0,\b>0$,
\begin{equation}\label{tp:rec}
g_{*,\beta}^\alpha = g_{*,0}^\alpha +\sum_{k=1}^{\beta} \frac{\mu(\beta-k)}{\mu(\beta)}
\Delta_\a(k) g_{*,\beta-k}^k \, .
\end{equation}
For $\b=1$, it reads 
\begin{eqnarray}\label{tp:rec1}
g_{*,1}^\alpha &=& g_{*,0}^\alpha + \frac{\mu(0)}{\mu(1)}\Delta_\a(1) g_{*,0}^1 
=g_{*,0}^\alpha + \frac{\mu(0)}{\mu(1)}
\Big(\frac{\mu(\a+1)}{\mu(\a)}-\frac{\mu(\a)}{\mu(\a-1)}\Big) g_{*,0}^1  \\
&=& g_{*,0}^\alpha + \frac{1}{\mu(1)} H_\alpha(1,1)g_{*,0}^1 \nonumber
\end{eqnarray}
which is equation \eqref{inv-tt} for $\b=1$,  using Definition \eqref{tp:hrec}.
 We now prove by induction on $\b\ge 1$ that equation \eqref{inv-tt} is satisfied.
 Suppose that the rates fulfill equations \eqref{tp:hrec} for all values 
of $\b$ up to some value $\tilde \b -1\ge 1$.  We derive 
\eqref{inv-tt} for $\b = \tilde \b$ as follows.  We have 
\begin{eqnarray}
g_{*, \tilde\b}^\a 
&=& g_{*,0}^\a
+\frac{\mu(0)}{\mu( \tilde\b)}\Delta_\a( \tilde\b)g_{*,0}^{ \tilde\b}  
+\sum_{k=1}^{ \tilde\b-1} \frac{\mu( \tilde\b-k)}{\mu( \tilde\b)}
\Delta_\a(k) g_{*, \tilde\b-k}^k \nonumber\\
&=& g_{*,0}^\alpha 
+\frac{\mu(0)}{\mu( \tilde\b)}\Delta_\a( \tilde\b)g_{*,0}^{ \tilde\b} 
 +\sum_{k=1}^{ \tilde\b-1} \frac{\mu( \tilde\b-k)}{\mu( \tilde\b)}
\Delta_\a(k) \Bigl( g_{*,0}^k 
+ \frac{1}{\mu( \tilde\b - k)} \sum_{l=1}^{\tilde\b -k} H_k(\tilde \b - k,l) g_{*,0}^l\Bigr) \nonumber\\
&=& g_{*,0}^\a +\frac{\mu(0)}{\mu( \tilde\b)}\Delta_\a( \tilde\beta)g_{*,0}^{ \tilde\beta} 
+\sum_{k=1}^{ \tilde\b-1} \frac{\mu( \tilde\b-k)}{\mu( \tilde\beta)} \Delta_\a(k) g_{*,0}^k
+\frac{1}{\mu( \tilde\b)}
\sum_{l=1}^{\tilde\beta -1} \sum_{k=1}^{ \tilde\b-l} 
\Delta_\a(k) H_k(\tilde \b - k,l) g_{*,0}^l \nonumber\\
&=& g_{*,0}^\a +\frac{\mu(0)}{\mu( \tilde\beta)}\Delta_\a( \tilde\beta)g_{*,0}^{ \tilde\beta} 
+\frac{1}{\mu( \tilde\b)}\sum_{k=1}^{ \tilde\b-1}\Bigl( \Delta_\a(k)  \mu( \tilde\b-k)
+\sum_{l=1}^{ \tilde\b-k} \Delta_\a(l) H_k(\tilde \b - l,k) \Bigr)g_{*,0}^k\nonumber\\
&=& g_{*,0}^\a +\frac{1}{\mu( \tilde\b)}
\sum_{k=1}^{ \tilde\b} H_\a( \tilde \b,k ) \;g_{*,0}^k\nonumber
\label{tp:rec2}\end{eqnarray}
 where we used \eqref{tp:rec} for the first equality and the induction hypothesis
for the second one, we exchanged the order of the two sums 
then the names of indexes in the third and fourth ones, and we used 
 Definition \eqref{tp:hrec} to conclude.  \par
Conversely, suppose that equations \eqref{inv-tt} are valid. We now prove that
equations\eqref{tp:grec0} or equivalently  equations \eqref{tp:rec}  are verified.
We have
\begin{eqnarray}
g_{*, \b}^\a 
&=& g_{*, 0}^\a + \frac{1}{\mu(\b)}\sum_{k=1}^\b H_\a(\b,l) g_{*, 0}^k\nonumber\\
&=& g_{*, 0}^\a + \frac{1}{\mu(\b)} H_\a(\b,\b) g_{*, 0}^\b
+ \frac{1}{\mu(\b)}\sum_{k=1}^{\b-1} \Bigl( \Delta_\a(k)\mu(\b-k)
+\sum_{l=1}^{\b-k}  \Delta_\a(l) H_l(\b-l,k)\Bigr) g_{*, 0}^k
\nonumber\\
&=& g_{*, 0}^\a + \frac{1}{\mu(\b)} \Bigl(H_\a(\b,\b) g_{*, 0}^\b
+ \sum_{k=1}^{\b-1} \Delta_\a(k)\mu(\b-k) g_{*, 0}^k
+\sum_{k=1}^{\b-1} \Delta_\a(k) \sum_{l=1}^{\b-k} H_k(\b-k,l) g_{*, 0}^l\Bigr)
\nonumber\\
&=& g_{*, 0}^\a + \frac{1}{\mu(\b)}\Bigl(\Delta_\a(\b)\mu(0) g_{*, 0}^\b
+ \sum_{k=1}^{\b-1} \Delta_\a(k)\mu(\b-k) g_{*, \b-k}^\a\Bigr)
\nonumber\\
&=& g_{*, 0}^\a 
+ \sum_{k=1}^{\b} \Delta_\a(k)\mu(\b-k) g_{*, \b-k}^\a\nonumber
\label{tp:rec3}
\end{eqnarray}
 \par  
 where we used  Definitions \eqref{tp:hrec} for the second and fourth equalities,
and we exchanged the order of the two sums 
then the names of indexes in the third  one. \par
\smallskip
It remains to show that solutions of  \eqref{inv-tt} also verify \eqref{inf-psi-t}.    
We do it in Lemma \ref{lemma:gtp4} below which finishes this proof.
\end{proof}
 We need two preparatory lemmas to derive Lemma \ref{lemma:gtp4}. 
\begin{lemma}\label{lem:4.2}     
 The coefficients $H_\alpha(\beta,k)$ solution of \eqref{tp:hrec} 
 are also solution of the following (dual)
recurrence relation 
\begin{equation}\label{tp:hrec2}
\begin{cases} H_\alpha(\beta,\beta) =   \Delta_\alpha(\beta) \mu(0) 
 \ \ \text{ for } \beta\ge 1,  \\[0.1cm]
H_\alpha(\beta,k) =  \Delta_\alpha( k) \mu(\beta-k)
+\sum_{l=1}^{\beta-k}  H_\alpha(\beta-k,l)  \Delta_l(k)  \  \  
\text{ for }  \beta\ge 2,\,1\le k\le \beta-1. \end{cases}
\end{equation} 
\end{lemma}
\begin{proof}
We do an induction on $(\beta,k)$, with $1\le k\le \beta-1$. 
Notice first that for all integers $\beta,m$, we have, by \eqref{tp:hrec},
\begin{equation}\label{tp:dual0}
\Delta_\alpha(m) H_m(\beta,\beta) = \Delta_\alpha(m) \Delta_m(\beta) \mu(0) 
=  \Delta_m(\beta) H_\alpha(m,m) .
\end{equation}
To initiate the induction with $\beta=2,k=1$, 
and to complete it for $k=\beta-1,\beta>2$, we compute, 
for $\beta\ge 2$, using \eqref{tp:hrec}, \eqref{tp:dual0}:
\[
H_\alpha(\beta,\beta-1) = \Delta_\alpha(\beta-1) \mu(1) + \Delta_\alpha(1) H_1(\beta-1,\beta-1)
=\Delta_\alpha(\beta-1) \mu(1) +\Delta_1(\beta-1) H_\alpha(1,1) 
\]
which is the expression of $H_\alpha(\beta,\beta-1)$ given by  \eqref{tp:hrec2}. 
Then we fix $1\le k<\beta-1$, 
and we assume  \eqref{tp:hrec2} satisfied
for all  $(\beta', k')$ such 
that $\beta'< \beta$ and $1\le k'< \beta'-1$.  We have
\begin{eqnarray*}
&& \Delta_\alpha( k) \mu(\beta-k)+\sum_{l=1}^{\beta-k}  H_\alpha(\beta-k,l)  \Delta_l(k)\\
&&  = \Delta_\alpha( k) \mu(\beta-k)+\sum_{l=1}^{\beta-k-1} \Bigl( \Delta_\alpha( l) \mu(\beta-l-k)
        +  \sum_{m=1}^{\beta-l-k}  \Delta_\alpha( m) H_m(\beta-k-m,l)\Bigr)  \Delta_l(k)\\ 
        &&\quad +H_\alpha(\beta-k,\beta-k)\Delta_{\beta-k}(k) \\
&&  = \Delta_\alpha( k) \mu(\beta-k)  
+ \sum_{l=1}^{\beta-k-1}  \Delta_\alpha( l)\Bigl(  \Delta_l( k) \mu(\beta-l-k)
        + \sum_{m=1}^{\beta-l-k}  H_l(\beta-l-k,m) \Delta_m(k)\Bigr)\\ 
        &&\quad +\mu(0)\Delta_\alpha(\beta-k)\Delta_{\beta-k}(k)\\
&&  = \Delta_\alpha( k) \mu(\beta-k)+ \sum_{l=1}^{\beta-k}   \Delta_\alpha(l) H_l(\beta-l,k) \\
&&  =  H_\alpha(\beta,k)
\end{eqnarray*} 
where we applied first \eqref{tp:hrec} to $H_\alpha(\beta-k,l)$ 
to get the first equality, exchanged the order of the two sums 
then the names of indexes to get the second one, applied the induction hypothesis with 
$\beta'= \beta-l < \beta$ to get the third one and finally applied \eqref{tp:hrec} in the other direction
to conclude.
\end{proof}
\begin{lemma}\label{lem:tp:sum} 
The function defined by $\psi(\beta,k) = \sum_{l=1}^{\beta - k} \mu(l) H_l(\beta -l , k)$ 
for $\beta>1$ and  $k< \beta$ satisfies
\begin{equation}\label{tp:sum}
\psi(\beta,k) = \sum_{l=1}^{\beta - k} \mu(l) H_l(\beta -l , k) = \mu(0)\mu(\beta) -\mu(k)\mu(\beta -k).
\end{equation}
\end{lemma}
\begin{proof}
We do an induction on $(\beta,k)$, with $1\le k\le \beta-1$.
To initiate the induction with $\beta=2,k=1$, and to complete it 
for $k=\beta-1,\beta>2$, we compute, for $\beta\ge 2$, using \eqref{tp:hrec}, \eqref{tp:Delta}:
$$
\psi(\beta,\beta-1) = \mu(1) H_1(\beta-1 , \beta-1) = \mu(1) \mu(0) \Delta_1(\beta-1)
= \mu(1) \mu(0) \left(\frac{\mu(\beta)}{\mu(1)}-\frac{\mu(\beta-1)}{\mu(0)}\right).
$$
Then we fix $1\le k<\beta-1$, 
and we assume  \eqref{tp:sum} satisfied
for all  $(\beta', k')$ such 
that $\beta'< \beta$ and $1\le k'< \beta'-1$.  We have
\begin{eqnarray*}
\psi(\beta,k) 
              &=& \sum_{l=1}^{\beta - k-1} \mu(l) \Bigl(\Delta_l( k) \mu(\beta-l-k)
                   +\sum_{m=1}^{\beta-l-k}  H_l(\beta-l-k,m)  \Delta_m(k)\Bigr)\\&&\quad 
                   +\mu(\beta - k) H_{\beta - k}(k , k)\\
              &=& \sum_{l=1}^{\beta - k-1} \Bigl( \mu(l)  \mu(\beta-l-k)
                   +\sum_{m=1}^{\beta-l-k} \mu(m)  H_m(\beta-m-k,l) \Bigr)  \Delta_l( k)\\&&\quad 
                   +\mu(\beta - k) \Delta_{\beta - k}(k)\mu(0)\\
              &=& \sum_{l=1}^{\beta - k-1} \Bigl( \mu(l)  \mu(\beta-l-k)
               + \psi(\beta-k,l ) \Bigr)  \Delta_l( k) 
              +\mu(\beta - k) \Delta_{\beta - k}(k)\mu(0)\\
              &=& \sum_{l=1}^{\beta - k} \mu(0)  \mu(\beta-k)  \Delta_l( k) \\
&=&  \mu(0)  \mu(\beta-k) \Bigl(\frac{\mu(\beta)}{\mu(\beta-k)}-\frac{\mu(k)}{\mu(0)}\Bigr)\\
&=& \mu(0)\mu(\beta) -\mu(k)\mu(\beta -k)
\end{eqnarray*}
where we applied \eqref{tp:hrec2} to get the first equality, 
we exchanged the order of  sums then the names of indexes to get 
the second one, we used the definition of $\psi(.,.)$ to get the third one,  
we applied the induction hypothesis for the fourth one
and \eqref{tp:Delta} for the fifth one.
\end{proof}

Now we are ready to show that
\begin{lemma} \label{lemma:gtp4}
Equation \eqref{inf-psi-t} holds with rates 
$ g_{*,\beta}^\alpha$ as in \eqref {inv-tt}--\eqref {tp:hrec}.
\end{lemma}
\begin{proof} 
We have,  for $\b\ge 1$,
\begin{eqnarray}
\mathbf{A}(0,\b)  &=& \1_{[\b\ge 2]}\sum_{k=1}^{\beta-1} 
\frac{\mu(k) \mu(\beta-k)}{\mu(0)\mu(\beta)} g_{*,\beta-k}^k 
+\frac{\mu(\beta) \mu(0)}{\mu(0)\mu(\beta)} g_{*,0}^{\beta}\nonumber\\
&=&  \1_{[\b\ge 2]}\sum_{k=1}^{\beta-1} \frac{\mu(k) \mu(\beta-k)}{\mu(0)\mu(\beta)} \Bigl( g_{*,0}^k 
+ \frac{1}{\mu(\beta-k) }\sum_{l=1}^{\beta-k} H_k(\beta-k,l) g_{*,0}^l\Bigr)+g_{*,0}^{\beta}  \nonumber\\
&=& \1_{[\b\ge 2]} \frac{1}{\mu(0)\mu(\beta) }  \sum_{k=1}^{\beta-1} 
\Bigl( \mu(k) \mu(\beta-k) +\sum_{l=1}^{\beta-k} \mu(l) 
H_l(\beta-l,k)\Bigr) g_{*,0}^k +g_{*,0}^{\beta} \nonumber\\
&=&  \frac{1}{\mu(0)\mu(\beta) }  \sum_{k\le \beta} 
\Bigl( \mu(0)\mu(\beta)\Bigr) g_{*,0}^k  =   \sum_{k\le \beta}   g_{*,0}^k
=  - \mathbf{A}(\b,0)  \nonumber
\end{eqnarray} 
where we used \eqref{A-target} for the first equality, \eqref{inv-tt} for the second one, 
we exchanged the order of  sums then the names of indexes to get 
the third one, we applied Lemma \ref{lem:tp:sum}
for the fourth one and \eqref{def-psi} for the last one. 
\end{proof}
\begin{proof} \textit{of Proposition \ref{target-exact}.}
 Point (a) follows directly from Theorem \ref{target}. 
As a preliminary for point (b), take a MM-TP whose rates satisfy \eqref{assum} and \eqref{eq:weak-phil},
for which a product probability measure $\bar\mu\in{\mathcal{S}}$ with single-site marginal $\mu$
(with $\mu(\a)>0$ for all $\a\ge 0$)
is invariant. Then by  Theorem \ref{target},
 equations \eqref{inv-tt}--\eqref{tp:hrec}
are satisfied, so that for $\b=1,\a\ge 1$ we obtain again Equation \eqref{tp:rec1} that we may write as
\begin{equation}\label{eq:mu-it0-target}
\frac{\mu(\a+1)}{\mu(\a)} = \frac{\mu(\a)}{\mu(\a-1)}
 + \frac{(g^\a_{*,1} - g^\a_{*,0}) }{g^1_{*,0}} \frac{\mu(1)}{\mu(0)}
\end{equation}
which implies that
\begin{eqnarray}\label{eq:mu-it-target2}
\frac{\mu(2)}{\mu(1)} &=&  \frac{\mu(1)}{\mu(0)} \frac{g^1_{*,1}}{g^1_{*,0}}\qquad\&\\
\label{eq:mu-it-target-ge2}
\mu(\a+1) &=& \mu(\a) \frac{\mu(1)}{\mu(0)} \frac{1}{g^1_{*,0}}
 \left(g^1_{*,1}+\sum\limits_{i=2}^\a (g^i_{*,1} -  g^i_{*,0}) \right)\qquad\mbox{ for }\a\ge 2
\end{eqnarray} 
that is, using notations \eqref{tildefimu}, \eqref{ex5-w-def},
\begin{equation}\label{eq:mu-it0-target-mid}
\mu(\a)=\mu(0)(\tilde\varphi_\mu)^\a w(\a)\qquad\mbox{ for }\a\ge 0.
\end{equation} 
Defining, for $\a\ge 1,1\le k\le\b$, $\Delta^*_\a(\b)$ and $ H^*_\a(\b,k)$ by
\begin{equation}\label{eq:defH*D*}
\Delta_\a(k)=(\tilde\varphi_\mu)^k\Delta^*_\a(k),\qquad
H_\a(\b,k)=(\tilde\varphi_\mu)^\b w(\b)\mu(0) H^*_\a(\b,k)
\end{equation}
we have that $H^*_\a(\b,k)$ and $\Delta^*_\a(\b)$ satisfy \eqref{target-rates-c}
and \eqref{tp:Delta*}.\\

\noindent
We now come to point (b). We consider a MM-TP where 
the $g^\a_{*,\b}$  satisfy \eqref{un-assum}  and conditions 
\eqref{eq:weak-phil}, \eqref{target-rates-cc} 
for  
$H^*_\alpha(\beta,k),\, \a\geq 1, 1\le k\le \beta$
 given in  \eqref{target-rates-c},   $\Delta^*_r(s),\,r>0,s\ge 0$ given in 
\eqref{tp:Delta*}
and $w(\a),\,\a\ge 0$ given in 
\eqref{ex5-w-def}. Then, consistently with \eqref{eq:mu-it0-target-mid}--\eqref{eq:defH*D*},
 formulas \eqref{inv-tar}--\eqref{inv-tar2} 
define a one-parameter family of product invariant measures. 
\end{proof}

\section{Attractiveness and coupling rates}\label{s-coupling}
 In this section, we take advantage of the work done in \cite{GS} on 
multiple-jump conservative dynamics of misanthrope type, a class of models
including MMP,  to analyze attractiveness properties of MMP. 
In Subsection \ref{subsec:IcapSe}, we  
derive the extremal translation invariant and invariant probability measures
for MM-ZRP and MM-TP. In Lemma \ref{lemma:attract} and in Subsection \ref{subsec-attra-cond}
we give necessary and sufficient conditions for attractiveness of MMP, MM-ZRP and MM-TP
based on the conditions established in \cite{GS}.  \par
\medskip

Let us 
 first recall the set-up for attractiveness, for which we refer to \cite{Lig}. 
 We consider the following partial order on 
the state space $\XX\subset \N^{\X}$. For $\eta,\zeta\in\XX$,
$$\eta\leq \zeta \ \ \text{if and only if } \eta(x)\leq \zeta(x)\ \text{for every }x\in\X.$$ 
A  bounded, continuous  function $f$ on $\XX$ 
is called \emph{monotone} if $f(\eta)\leq f(\zeta)$ whenever $\eta\leq \zeta$.
We call a particle system \emph{attractive} if for every bounded, monotone continuous function $f$
 on $\XX$ and every time $t>0$, the function $S(t) f$ is again a bounded, monotone, continuous function
  on $\XX$.  To check attractiveness of a particle system, the first step is to derive necessary conditions
on the rates. The second step is to check that they are sufficient by constructing an \emph{increasing},
markovian coupling $\left( \eta_t,\zeta_t  \right)_{t\ge 0}$ of two copies of the process, 
that is a coupling which preserves the ordering of its marginals: 
\\[.1cm]
$\makebox[2cm]{} \eta_0\leq \zeta_0 $ implies 
$\eta_t\leq \zeta_t,\ P^{(\eta_0,\zeta_0)}$-almost surely, for every $t>0$. 
\\[.1cm]
Usually, for dynamics with transitions consisting in the jump, birth or death of a single particle, 
the so-called \textit{basic coupling} is used: for jumps, it means that coupled particles
try to jump together as much as possible from the same departure site to the same arrival site. 
But basic coupling is useless when $k\ge 1$ particles jump simultaneously, 
as in the class of conservative dynamics analyzed 
in \cite{GS}, which includes the MMP. 
Let us rewrite the results from
\cite{GS} we need in the context of models with generator
\eqref{generator}.
\smallskip
 
 For all $\a,\b,k\geq 0$, we denote
\begin{equation}\label{not:sigma}
{\bf\sigb}^{k}_{\a,\b}=\sum_{k'>k} \gen^{k'}_{\a,\b}  =\sum_{k'=k+1}^\a \gen^{k'}_{\a,\b}.
\end{equation} 
The following necessary and sufficient conditions for attractiveness were established in  
\cite[Theorem 2.9]{GS}:
For all $\,\a,\b,\g,\delta\geq 0$, with $\a\leq \g, \  \b\leq \delta$, 
\begin{subequations}
\begin{equation}\label{gs-atr-1}  
\forall\, l\geq 0,  \quad 
{\bf\sigb}^{\delta-\b+l}_{\a,\b} \ \leq \ {\bf\sigb}^{l}_{\g,\delta}
\end{equation} 
\begin{equation}\label{gs-atr-2} 
\qquad\&\quad 
\forall\, k\geq 0,  \quad 
{\bf\sigb}^{k}_{\a,\b} \ \geq \ {\bf\sigb}^{\g-\a+k}_{\g,\delta}.
\end{equation} 
\end{subequations}
Sufficiency was proved by constructing the following (Markovian) increasing coupling process
$\left( \eta_t,\zeta_t  \right)_{t\ge 0}$ on $\XX\times \XX$,
with semigroup $(\overline{S}(t):t\geq 0)$ 
and infinitesimal generator $\overline{\L}$. 
\begin{equation}\label{couplingL}
\overline{\L}h(\eta,\zeta) = \sum_{x\in \X} \sum_{y\in \X} p(x,y)\, \sum_{k\geq 0} \sum_{l\geq 0}  
                                 G^{k,l}_{\eta(x),\eta(y),\zeta(x),\zeta(y)} 
\left(h\left(\etb,\zetb\right)-h\left(\eta,\zeta\right)\right)
\end{equation} 
is defined for $(\eta,\zeta)\in\XX \times \XX$ on the set of functions
\begin{eqnarray} \nonumber
  \overline{\Lip}&=&\{ h:\XX \times \XX\rightarrow\R\ \text{ such that for some } \overline{L_h}>0,\\ 
  &&|h(\eta_1,\zeta_1)-h(\eta_2,\zeta_2)|\leq \overline{L_h} (\|\eta_1-\eta_2\|+\|\zeta_1-\zeta_2\|),\,
  \forall\eta_1,\eta_2,\zeta_1,\zeta_2\in\XX  \}.\label{def-lip-coupl}
\end{eqnarray}
The rates for coupled jumps are defined in \cite[Proposition 2.11]{GS}:
  for all $\a,\b,\gamma,\delta\geq 0$, for all $1\le k\le\a$ when $1\le\a$, 
 for all $1\le l\le\gamma$, 
\begin{equation}\label{coupling} 
\begin{array}{lll}
\gg  &=&  \biggl( \gen^k_{\a,\b}-\gen^k_{\a,\b}\wedge \Bigl( \sigb^{l}_{\g,\delta}  
 -  \sigb^{l}_{\g,\delta}  
\wedge   \sigb^{k}_{\a,\b}      \Bigr)  \biggr) \wedge   
 \biggl( \gen^l_{\g,\delta}-\gen^l_{\g,\delta}\wedge \Bigl( \sigb^{k}_{\a,\b}   -  
\sigb^{l}_{\g,\delta}  \wedge   \sigb^{k}_{\a,\b}      \Bigr)  \biggr),\\
G_{\alpha,\b,\gamma,\delta}^{k,0} &=&  \gen^k_{\a,\b}
 - \gen^k_{\a,\b}\wedge \Bigl( \sigb^{0}_{\g,\delta}   -  \sigb^{0}_{\g,\delta}  
\wedge   \sigb^{k}_{\a,\b}      \Bigr),   
\\
G_{\alpha,\b,\gamma,\delta}^{0,l} &=&  \gen^l_{\g,\delta} - \, 
\gen^l_{\g,\delta}\, \wedge \Bigl( \sigb^{0}_{\a,\b}  
 -  \sigb^{0}_{\a,\b}  \wedge   \sigb^{l}_{\g,\delta}      \Bigr).      
\end{array}
\end{equation}
In other words, for a given time $t$ and a departure site $x\in\X$, 
a target site $y$ is chosen with probability $p(x,y)$,
and the simultaneous jump of $k$ particles in the first coordinate 
together with $l$ particles in the second coordinate from $x$ to $y$ occurs with a rate
$
p(x,y)G^{k,l}_{\a,\b,\gamma,\delta} 
$
where $\a=\eta_{t_-}(x)$, $\b=\eta_{t_-}(y)$, $\gamma=\zeta_{t_-}(x)$, 
$\delta=\zeta_{t_-}(y)$, $0\leq k\leq \a$ and $0\leq l\leq \g$. 
Other jumps are not allowed. \par
Even if this coupling was originally
built in connection with attractiveness, it
is of course always valid, thus we used it in Section
\ref{s-model}. \par
Attractiveness conditions \eqref{gs-atr-1}--\eqref{gs-atr-2} 
can be reset in the following explicit form
for MMP. 
\begin{lemma} \label{lemma:attract} 
Consider (a) the MMP given by generator 
(\ref{generator}) with rates $\gkab$ and, respectively in (b) and (c), 
the MM-ZRP with rates $\gka$ and the MM-TP with rates $\gkb$.
\begin{itemize}
\item[(a)]    The MMP is attractive if and only if for all $\, \a\geq 1,\,\,  \b\geq 0,\, k\geq 0,$ 
\begin{subequations}
\begin{equation} \label{atr:mmp-1}  
{\bf\sigb}^{k+1}_{\a+1,\b} \quad \leq \quad {\bf\sigb}^{k}_{\a,\b}
\quad \leq\quad {\bf\sigb}^{k}_{\a+1,\b}
\end{equation} 
\begin{equation} \label{atr:mmp-2} 
\& \qquad    
{\bf\sigb}^{k+1}_{\a,\b}\quad \leq\quad {\bf\sigb}^{k}_{\a,\b+1}
\quad \leq\quad  {\bf\sigb}^{k}_{\a,\b}.
\end{equation} 
\end{subequations}
\item[(b)]   The MM-ZRP is attractive if and only if for all $\, \a \geq 1, \, k\geq 0$, 
\begin{equation} \label{atr:gzrp} 
\sum\limits_{k'>k+1}^{\a+1} \gen_{\a+1}^{k'} \quad\leq 
\quad  \sum\limits_{k'>k}^{\a} \gen_{\a}^{k'} \quad \leq 
\quad  \sum\limits_{k'>k}^{\a+1} \gen_{\a+1}^{k'} .
\end{equation}
\item[(b')]   The MM-ZRP is attractive if and only if for all $\a \geq 1, \,1\leq m\leq  \a$ 
\begin{equation} \label{atr:gzrp:2} 
0 \ \leq \ 
\sum_{j=0}^{m-1} \left(\gen_{\a}^{\a-j} - \gen_{\a+1}^{\a+1-j}\right) \quad\leq 
\gen_{\a+1}^{\a+1-m} .
\end{equation}
 \item[(c)] The MM-TP is attractive if and only if for all $\, \a\geq 1,\,\, \b\geq 0,\, k\geq 0,$ 
\begin{subequations}
\begin{equation} \label{atr:GTP-1} 
g^{\a+1}_{*,\b} \quad\leq \quad  g^{\a}_{*,\b}  
\quad\ \& \ \quad 
g^{\a}_{*,\b+1} \quad\leq\quad   g^{\a}_{*,\b}  
\end{equation} 
\begin{equation}  \label{atr:GTP-2}
\& \quad\quad \sum\limits_{k'>k+1}^{\a} g_{*,\b}^{k'} \quad 
\leq\quad  \sum\limits_{k'>k}^{\a} g_{*,\b+1}^{k'}. 
\end{equation}
\end{subequations}
\end{itemize}
\end{lemma}
\begin{proof}
\mbox{}\\ 
\noindent
(a) Taking $(\g,\delta)=(\a+1,\b)$ in \eqref{gs-atr-1} (resp. in \eqref{gs-atr-2}) gives the 
second (resp. first) inequality in \eqref{atr:mmp-1}. 
Taking $(\g,\delta)=(\a,\b+1)$ in \eqref{gs-atr-1} (resp. in \eqref{gs-atr-2}) gives the 
first (resp. second) inequality in \eqref{atr:mmp-2}.\par
For the converse, iterate $\g-\a$ times the second (resp. first) inequality in \eqref{atr:mmp-1},
then $\delta-\b$ times the first (resp. second) inequality in \eqref{atr:mmp-2}, to get \eqref{gs-atr-1}
(resp. \eqref{gs-atr-2}).\par
 \noindent
(b) Writing \eqref{atr:mmp-1} for MM-ZRP gives \eqref{atr:gzrp}, and \eqref{atr:mmp-2}
becomes trivial  for MM-ZRP.\par
 \noindent
(b')
To rewrite \eqref{atr:gzrp} as \eqref{atr:gzrp:2}, make the change of variables
$k'=\a-j$ (resp. $k'=\a+1-j$) in the middle (resp. right and left) sum, then subtract the left sum
from every term of the inequalities.\par\smallskip
\noindent
(c) Taking $k=\a-1$ in \eqref{atr:mmp-1}--\eqref{atr:mmp-2} gives \eqref{atr:GTP-1}.
Taking $k<\a-1$ in  the first inequality of \eqref{atr:mmp-2} gives \eqref{atr:GTP-2}.
For the converse, the second inequality in \eqref{atr:mmp-2} (resp. the first 
inequality in \eqref{atr:mmp-1}) corresponds to the
replacement of $\a$ by $k'$ followed by a sum over $k'$ of the second (resp. the first)
inequality in \eqref{atr:GTP-1}.  
\end{proof}
 \begin{remark}\label{rk:single-attra}
For single-jump models, all the inequalities in Lemma
\ref{lemma:attract} are restricted to $k=0$. 
First, we recover
well-known results: for ZRP,  attractiveness means that 
 $g_\a^1$ is a nondecreasing function of $\a$ (and basic coupling
is an increasing coupling process, see \cite{andjel}
for a detailed construction). The misanthrope process was introduced in \cite{Cocozza} 
as an attractive process, that is with the function $g_{\a,\b}^1$ 
nondecreasing in its first coordinate $\a$ and  nonincreasing in its second one $\b$ 
(thus explaining the name \emph{misanthrope}). \par
For TP, attractiveness means that the function $g^1_{*,\b}$ is nonincreasing in  $\b$.
When a ZRP and a TP are dual single jump models (see Remark \ref{rk:duality-zrp-tp}),
attractiveness of the first is equivalent to attractiveness of the second.
\end{remark} 
\subsection{Characterization of $(\mathcal{I}\cap\mathcal{S})_{e}$ for MM-ZRP and MM-TP}\label{subsec:IcapSe}
The fact that attractiveness is a very strong property of conservative 
particle systems is  illustrated by the
 characterization of the set of all translation-invariant and invariant measures
of MMP.  The latter follows from \cite[Theorem 5.13]{GS}, which requires
attractiveness and irreducibility conditions for the coupling transition rates,
denoted by \textit{(IC)}. Under Conditions \textit{(IC)}, if an initial coupled configuration $(\xi_0,\zeta_0)$
contains a pair of discrepancies of opposite signs (that is, on two sites $x,y$, $(\xi_0(x)-\zeta_0(x))(\xi_0(y)-\zeta_0(y))<0$), 
there is a positive probability (for the coupled evolution) that it evolves into a locally ordered coupled
configuration (that is, for some $t\ge 0$, $(\xi_t(x)-\zeta_t(x))(\xi_t(y)-\zeta_t(y))\ge 0$;
we refer to \cite[Section 2.3]{GS} for a detailed analysis of the evolution of discrepancies).

We derive here $(\mathcal{I}\cap\mathcal{S})_{e}$ for the particular cases of MM-ZRP in the context of
Proposition \ref{gzrp-exact} and MM-TP in the context of Proposition \ref{target-exact}.  
To check conditions 
\textit{(IC)}, we explicit the values of $G^{k,l}_{\a,\b,\gamma,\delta}$
for $\a,\b,\gamma,\delta\geq 0$ and $(k,l)\in\{(0,1),(1,0),(1,1)\}$, according to 
\eqref{not:sigma} expressions \eqref{coupling}. 
\begin{equation}\label{eq:G01}  
G^{0,1}_{\a,\b,\gamma,\delta} 
= 
\left\{
\begin{array}{ll}
\gen^1_{\g,\delta}   & \ \text{ if  } \  \sigb^{0}_{\a,\b} \leq \sigb^{1}_{\g,\delta}   
 \\[.1cm]
\sigb^{0}_{\g,\delta} - \sigb^{0}_{\a,\b}  & \ \text{ if } \ 
\sigb^{1}_{\g,\delta} <  \sigb^{0}_{\a,\b} \leq \sigb^{0}_{\g,\delta} 
 \\[.1cm]
0  & \ \text{ if } \ \sigb^{0}_{\g,\delta} <  \sigb^{0}_{\a,\b}
\end{array}
\right.
\end{equation}
\begin{equation}\label{eq:G10}   
G^{1,0}_{\a,\b,\gamma,\delta}
= 
\left\{
\begin{array}{ll}
0  & \ \text{ if } \ \sigb^{0}_{\a,\b} \leq \sigb^{0}_{\g,\delta} 
 \\[.1cm]
\sigb^{0}_{\a,\b}- \sigb^{0}_{\g,\delta}   & \ \text{ if  } \ 
\sigb^{1}_{\a,\b} \leq \sigb^{0}_{\g,\delta} < \sigb^{0}_{\a,\b}
 \\[.1cm]
\gen^1_{\a,\b}    & \ \text{ if } \
\sigb^{0}_{\g,\delta} < \sigb^{1}_{\a,\b}
\end{array}
\right.
\end{equation}
\begin{equation}\label{eq:G11}  
G^{1,1}_{\a,\b,\gamma,\delta}
= 
\left\{
\begin{array}{ll}
0   & \ \text{ if } \sigb^{0}_{\a,\b} \leq \sigb^{1}_{\g,\delta}    \\[.1cm]
  \sigb^{0}_{\a,\b} - \sigb^{1}_{\g,\delta}   & \ \text{ if } 
  \sigb^{1}_{\a,\b} \leq \sigb^{1}_{\g,\delta} < \sigb^{0}_{\a,\b}\leq \sigb^{0}_{\g,\delta} \\[.1cm]
\gen^1_{\g,\delta}   & \ \text{ if } 
\sigb^{1}_{\a,\b} \leq \sigb^{1}_{\g,\delta} \leq \sigb^{0}_{\g,\delta}< \sigb^{0}_{\a,\b}     \\[.1cm]
\gen^1_{\a,\b}  & \ \text{ if }   
\sigb^{1}_{\g,\delta} < \sigb^{1}_{\a,\b} \leq \sigb^{0}_{\a,\b}\leq \sigb^{0}_{\g,\delta} \\[.1cm]
  \sigb^{0}_{\g,\delta} - \sigb^{1}_{\a,\b}   & \ \text{ if } 
  \sigb^{1}_{\g,\delta} < \sigb^{1}_{\a,\b}  \leq \sigb^{0}_{\g,\delta}< \sigb^{0}_{\a,\b} \\[.1cm]
0  & \ \text{ if } \sigb^{0}_{\g,\delta}\leq \sigb^{1}_{\a,\b}
\end{array}
\right.
\end{equation}
\begin{proposition}\label{thm:as_zr}
Consider a MM-ZRP for which the $g^k_\a$'s 
satisfy    \eqref{un-assum},   \eqref{eq:g1_zr0} and \eqref{inv-rate-zrp}.
Assume that the attractiveness condition \eqref{atr:gzrp} is satisfied. 
Then 
\begin{equation}\label{eq:as_zr}
(\mathcal{I}\cap\mathcal{S})_{e} = \{\bar{\mu}_{\varphi}: \varphi\in {\rm Rad}(Z') \}
 \text{  or  }\{\bar{\mu}_{\varphi}: \varphi\in {\rm Rad}(Z) \} 
\end{equation}
where each $\bar{\mu}_{\varphi}$  is a product probability measure with marginals given by  
\eqref{inv-zrp}.
\end{proposition}
\begin{proof}   
By Proposition \ref{gzrp-exact}, the  MM-ZRP 
has a one-parameter family of product, translation-invariant, invariant probability measures. 
As explained above,  due to attractiveness, we have to check conditions 
\textit{(IC)} of \cite[Theorem 5.13]{GS}. For $x,y\in\Z^d$, let $\a,\b,\gamma,\delta\geq 0$ be such that
$\xi_0(x)=\a,\xi_0(y)=\b,\zeta_0(x)=\gamma,\zeta_0(y)=\delta$ with $\a<\gamma,\b>\delta$.
Condition \eqref{atr:gzrp} for $k=0$ implies 
$$\sigb^{0}_{\a,\b} \le \sigb^{0}_{\gamma,\delta}\qquad \& \qquad
\sigb^{1}_{\a,\b} \le \sigb^{1}_{\gamma,\delta}\le \sigb^{0}_{\gamma,\delta}.$$
This combined with \eqref{eq:g1_zr0} enables to simplify \eqref{eq:G01}--\eqref{eq:G11},  which become 
\begin{eqnarray*}  
G^{0,1}_{\a,\b,\gamma,\delta} 
&= &
\left\{
\begin{array}{ll}
\gen^1_{\g,\delta}   & \ \text{ if  } \  \sigb^{0}_{\a,\b} \leq \sigb^{1}_{\g,\delta}   
 \\[.1cm]
\sigb^{0}_{\g,\delta} - \sigb^{0}_{\a,\b}  & \ \text{ if } \ 
\sigb^{1}_{\g,\delta} <  \sigb^{0}_{\a,\b} \leq \sigb^{0}_{\g,\delta} 
\end{array}
\right. \\    
G^{1,0}_{\a,\b,\gamma,\delta}
&=& 0 \\   
G^{1,1}_{\a,\b,\gamma,\delta}
&=& 
\left\{
\begin{array}{ll}
0   & \ \text{ if } \sigb^{0}_{\a,\b} \leq \sigb^{1}_{\g,\delta}    \\[.1cm]
  \sigb^{0}_{\a,\b} - \sigb^{1}_{\g,\delta}   & \ \text{ if } 
  \sigb^{1}_{\a,\b} \leq \sigb^{1}_{\g,\delta} < \sigb^{0}_{\a,\b}\leq \sigb^{0}_{\g,\delta} \\[.1cm]
0  & \ \text{ if } \sigb^{0}_{\g,\delta}\leq \sigb^{1}_{\a,\b}
\end{array}
\right.
\end{eqnarray*}
Therefore
\begin{equation}\label{eq:IC-for-MM-ZRP-and-MM-TP}
\begin{array}{lrl}
 \text{if} &   \sigb^{0}_{\a,\b} \le \sigb^{1}_{\gamma,\delta}
 & \text{then }\quad G^{0,1}_{\a,\b,\gamma,\delta}=g^1_{\gamma,\delta}>0;
 \\[2mm]
 \text{while if} & \sigb^{0}_{\a,\b} > \sigb^{1}_{\gamma,\delta}
 &  \text{then }\quad  G^{1,1}_{\a,\b,\gamma,\delta}
 =\sigb^{0}_{\a,\b}-\sigb^{1}_{\gamma,\delta} >0.
\end{array}
\end{equation}
Then, proceeding as in the proof of \cite[Proposition 6.3]{GS}, we build a path of coupled transitions
which leads to  a locally ordered coupled configuration. In a similar way, if $\a<\gamma,\b=\delta$, 
we build a path of coupled transitions which creates a discrepancy on $y$. Hence conditions 
\textit{(IC)} are satisfied.
 \end{proof}
 \begin{remark}
The path of coupled transitions built previously is a priori not unique. For instance, 
for the stick process (see Example \ref{ex:GZRP-2} in Section \ref{s-examples}), other paths are 
exhibited in \cite[Subsection 6.4]{GS}. 
 \end{remark}
\begin{proposition}\label{thm:as_tp}
Consider a  MM-TP for which the rates $g^k_{*,\b}$  satisfy   \eqref{un-assum}, 
 relations \eqref{eq:weak-phil}--\eqref{ex5-w-def}, and are such that $g^1_{*,\a}>0$ for every $\a\geq 0$.
 Assume moreover the attractiveness conditions \eqref{atr:GTP-1}--\eqref{atr:GTP-2} satisfied. 
Then 
\begin{equation}\label{eq:as_tp}
(\mathcal{I}\cap\mathcal{S})_{e} = \{\bar{\mu}_{\varphi}: \varphi\in {\rm Rad}(Z') \}
 \text{  or  }\{\bar{\mu}_{\varphi}: \varphi\in {\rm Rad}(Z) \} 
\end{equation}
where each $\bar{\mu}_{\varphi}$  is the product measure with  marginals given by \eqref{inv-tar}.
\end{proposition}
\begin{proof}   
By  Proposition \ref{target-exact},
the MM-TP  
has a one-parameter family of product, translation-invariant, invariant probability measures.

 Since we assumed that $g^1_{*,\a}>0$ for every $\a\geq 0$, 
we can proceed exactly as in the proof of Proposition \ref{thm:as_zr} to check conditions 
\textit{(IC)} of \cite[Theorem 5.13]{GS}, restricting ourselves to coupled jumps with $k,l\in\{0,1\}$.
\end{proof}
\medskip
%
\subsection{Attractiveness of MM-ZRP with product invariant measures}\label{subsec-attra-cond}
 In this paragraph we give equivalent forms of condition \eqref{atr:gzrp:2} for attractiveness
of MM-ZRP for the following rate, which is generic for MM-ZRP with product invariant 
measures 
 (cf.  \eqref{inv-s}), as will be developed in Examples 3 and 4 of Section \ref{s-examples}. 
\begin{equation} \label{ex:ZRh}
\gka = \frac{\pi(\a-k)}{\pi(\a)}\,h(k)\quad \text{ for }\ 1\leq k\leq \a
\end{equation}
where $\pi$ is a positive function on $\N$ and $h$ is a non-negative
 function on $\N\setminus\{0\}$; we set $\pi(i)=0$ for $i\in\Z^-$ and $h(0)=0$.

In this setting, condition \eqref{atr:gzrp:2} for attractiveness writes:
For all $\a \geq 1, \,1\leq m\leq\a $,
\begin{equation} \label{attractGZR}
0 \ \ \stackrel{\text(a)}{\leq} \ \ 
\sum_{j=0}^{m-1} \pi(j) \Bigl( \frac{h(\a-j)}{\pi(\a)}  - \frac{h(\a+1-j)}{\pi(\a+1)} \Bigr)\ \stackrel{\text(b)}{\leq} \   
\frac{\pi(m)h(\a+1-m)}{\pi(\a+1)} .
\end{equation}
For each 
 $n\geq 0$ denote
\begin{equation} \label{def:r-for-ZRh}  
r(n)=\frac{\pi(n)}{\pi(n+1)}.
\end{equation}
\begin{lemma} \label{attr-al-gzrp}
(a) Assume that $r$ is a nonincreasing function. Then condition 
(\ref{attractGZR}) for attractiveness  is equivalent to 
\begin{subequations}
\begin{equation} \label{at-1}
\frac{h(\a+1)}{\pi(\a+1)}\leq \frac{h(\a)}{\pi(\a)} \ \quad \forall \a\geq 1\ \qquad \&   
\end{equation}
\begin{equation}\label{at-2}
\&\qquad \sum_{k=1}^{\a} h(k) \Bigl( \frac{\pi(\a-k)}{\pi(\a)} - \frac{\pi(\a-k+1)}{\pi(\a+1)} \Bigr)\ 
\leq \pi(0) \frac{h(\a+1)}{\pi(\a+1)}  \ \quad \forall \a\geq 1\ 
\end{equation}
\end{subequations}
and the following condition is necessary for the attractiveness 
of MM-ZRP with $\gka$ given by \eqref{ex:ZRh}. 
\begin{equation}\label{at-3}
S(\a) \leq S(\a+1) \quad \forall \a\geq 2, 
 \quad\text{ where }\quad S(\a)= \frac{1}{\pi(\a)} \sum_{i=1}^{\a-1} \pi(i)\pi(\a-i). 
\end{equation}
(a') If we moreover assume that $h(\a) = h_0\pi(\a) \ \forall \a\geq 1$ 
for some constant $h_0>0$ then condition \eqref{at-3}
is sufficient and necessary for the attractiveness of 
MM-ZRP with $\gka$ of the form
\begin{equation}\label{ex:ZRh0}
\gka = h_0\frac{\pi(\a-k)}{\pi(\a)}\,\pi(k)\quad \text{ for }\ 1\leq k\leq \a. 
\end{equation}
 (b) Assume that $r$ is a nondecreasing function. Then, to have attractiveness,
Condition \eqref{at-1}  is necessary, and Condition
\begin{equation}\label{at-bis-1}
  \frac{\pi(\a+1)}{\pi(\a)} 
\geq \max_{1\leq k\leq\a }\Bigl( \frac{h(k+1)}{h(k)}\Bigr)  \ \quad \forall \a\geq 1
\end{equation} 
is sufficient.  
\end{lemma}
\begin{proof}  
 \textit{Preliminaries:} Note that the function
\begin{equation}\label{at-fct-as-r}
m\mapsto\phi_\a(m)=\frac{\pi(m)h(\a+1-m)}{\pi(\a+1)} 
- \sum_{j=0}^{m-1} \pi(j) \Bigl( \frac{h(\a-j)}{\pi(\a)}  
- \frac{h(\a+1-j)}{\pi(\a+1)} \Bigr)
\end{equation}
defined for $1\leq m\leq\a$, which represents the r.h.s. of 
(\ref{attractGZR}b) minus its l.h.s.,
has the same monotonicity as $r$: if $r$ is nonincreasing 
(resp. nondecreasing), then $\phi_\a$ is also nonincreasing
(resp. nondecreasing).\\ 

\noindent
Part (a).
\mbox{}\\
$\bullet$
 Taking $m=1$ in (\ref{attractGZR}a) yields \eqref{at-1}.
 Then \eqref{at-1} combined with the fact that $r$ is nonincreasing implies that,
for $0\le j\le m-1$, $1\le m\le \a$,
$$\frac{h(\a-j)}{h(\a-j+1)}\ge \frac{\pi(\a)}{\pi(\a+1)} $$
hence (\ref{attractGZR}a) is valid. Therefore \eqref{at-1} is equivalent 
to (\ref{attractGZR}a).\par
\noindent 
$\bullet$
Because $r$ is nonincreasing, so is $\phi_\a(.)$ (by the preliminaries).
Hence condition (\ref{attractGZR}b) for every $1\leq m\leq\a $ 
is equivalent to (\ref{attractGZR}b) for  $m=\a$.
The latter is equivalent, by the change of variables $k=\a-j$, to 
condition \eqref{at-2}.\par
\noindent
$\bullet$
Using \eqref{at-1} twice in   \eqref{at-2}, first through inequality 
$\displaystyle{h(k)\geq \pi(k)\frac{h(\a)}{\pi(\a)}}$
for each summand on its left hand side, then on its r.h.s.,
 we obtain
$$ \frac{h(\a)}{\pi(\a)}\sum_{k=1}^{\a} \frac{\pi(k)}{\pi(0)}
 \Bigl( \frac{\pi(\a-k)}{\pi(\a)} - \frac{\pi(\a-k+1)}{\pi(\a+1)} \Bigr)\ 
\leq \frac{h(\a+1)}{\pi(\a+1)} \leq \frac{h(\a)}{\pi(\a)}
$$
which is trivial for $\a=1$, and implies 
\eqref{at-3} for $\a\ge 2$. \\ \\
 Part (a'). 
\mbox{}\\
If moreover the equality in \eqref{at-1} holds for all $\a$ 
(that is, $h(\a)=h_0\pi(\a)$), then 
\eqref{at-3} implies \eqref{at-2}. \\ \\
Part (b).
\mbox{}\\ 
As in Part (a), taking $m=1$ in (\ref{attractGZR}a) yields \eqref{at-1}. For the converse,
Condition \eqref{at-bis-1} implies (\ref{attractGZR}a). 
 To get (\ref{attractGZR}b), we use that because $r$ is nondecreasing, 
so is $\phi_\a(.)$ (by the preliminaries). Hence condition (\ref{attractGZR}b) 
for every $1\leq m\leq\a $ is equivalent to (\ref{attractGZR}b) for  $m=1$, 
which comes from the fact that $r$ is nondecreasing. 
 \end{proof}
\begin{remark}\label{rk:gkk}
Taking $k=\a$ in \eqref{ex:ZRh0} implies that for all $\a\ge 1$,
$g_{\a}^{\a} = h_0\pi(0)$.
\end{remark}
\medskip
Because \eqref{at-3} does not depend on $h$ anymore, Lemma \ref{attr-al-gzrp} implies: 
\begin{corollary}\label{cor:attr-al-gzrp}
 Assume $r$ nonincreasing. If the MM-ZRP with 
 $\gka = \frac{\pi(\a-k)}{\pi(\a)}\,\pi(k)$ is not attractive then the MM-ZRP with 
 $\gka = \frac{\pi(\a-k)}{\pi(\a)}\,h(k)$ is not attractive for any other choice of $h(\cdot)$. 
\end{corollary}
Therefore for attractiveness properties, it is enough to restrict ourselves to the MM-ZRP with $\gka$ given by \eqref{ex:ZRh0}
with $r$ nonincreasing, for which  we arrange condition \eqref{at-3}. By \eqref{def:r-for-ZRh},  
we can write for  $\a\geq 2$, $i\geq 2$  with $i\leq \a$,
\begin{equation}\label{arranged-at-3}
S(\a)= \frac{\pi(i) \pi(\a-i)}{\pi(\a)} 
=\pi(1) \frac{r(\a-1)r(\a-2)r(\a-3)\ldots r(\a-i)}{r(1)r(2)\ldots r(i-1)}. \end{equation}  
\bigskip


\section{Condensation} \label{s-discussion}

It follows from Section \ref{s-invariant} that we are able, 
for a probability measure $\bar{\mu}$ which is 
product, translation-invariant on $\Zd$, with single site marginal
 $\mu$ satisfying some conditions, to define rates of a MMP for which 
$\bar{\mu}$ is invariant. Moreover this yields a one-parameter family 
$\{\mp:\vfi\in {\rm Rad}(Z')\}$  of invariant measures for this process;  
 the single site partition function $Z_\vfi$, defined in  \eqref{Z},
has radius of convergence $\vfi_c$,  and the critical density $\rho_c$ was defined in \eqref{eq:rho_c}. 

The point is that, as we explain below in subsection \ref{subsec:cond-known},  
condensation phenomena amount to  convergence results   
concerning only stationary distributions, and not the details of the dynamics.
So it is enough to verify convergence properties of 
the partition function and the average density of particles 
with respect to the involved probability  measure to know whether 
condensation occurs in the stationary states.
Then having a stationary state which produces condensation we can 
study different MMPs which lead to this stationary state: 
We give some examples at the end of this section, and others in Section \ref{s-examples}.

We end this section checking whether attractiveness and condensation could be
compatible for MMPs. We answer `no' for particular cases. 
We then turn to this question for the second set-up for condensation, 
where we show that the answer `yes' is possible. 
\\ 
\subsection{Known results}\label{subsec:cond-known}
Consider a finite set $\Lambda\subset\Z$ with $L$ sites where $N$ particles evolve
according to the restriction on $\Lambda$ of the infinite volume dynamics,  with a periodic boundary condition. 
To stress dependence on finite volume we denote by $\bar{\mu}^{\Lambda}_\vfi  $ 
the product measure on $\N^{\Lambda}$ with
marginals $\mu_{\vfi}$. Such a  finite volume process has a unique stationary 
distribution $\mu_{N,L}$ (independent of $\vfi$) which can be expressed 
for $\eta\in\N^{\Lambda} $ as 
\begin{equation}\label{can-ensem}
\mu_{N,L} (\eta) = \bar{\mu}_\vfi^{\Lambda}\left( \eta | \sum_{x\in\Lambda} \eta(x) =N \right) .
\end{equation} 
The first set-up for condensation  \cite{AL,evans0,GSS} concerns, 
for a system with finite critical density $\rho_c<\infty$,
the thermodynamic limit of the above restricted dynamics 
when $N/L\rightarrow\rho>\rho_c $ when $N,L\rightarrow\infty$.
Then there is no appropriate invariant measure  and the system can be seen at two levels:
\emph{a~homogeneous background} (fluid phase) 
where occupation numbers  per site are distributed according to the product measure at critical density $\bar{\mu}_{\vfi_c}$, 
and \emph{a~condensate} (condensed phase) in which the excess mass is
accumulated on a single site.   We say that the process (or dynamics) \textit{allows condensation}. 

In the second set-up for condensation  \cite{BL,FeLaSi, rcg}, 
the number $L$ of sites is  fixed and only the total number $N$ of particles goes to infinity.
This set-up   requires only $\vfi_c$ to be finite, but $\rho_c$ can be infinite.
 We talk about  \textit{condensation on a fixed finite volume}.
 
We mention both approaches since we will present examples for each one.  In either case, we say that the process 
\textit{exhibits condensation}.  Note that  both cases require that $Z_{\vfi_c}<+\infty$. 
\\

 In the first set-up for condensation, assuming $\rho_c<\infty$, 
the following \emph{equivalence of ensembles} result was proved 
in \cite{GSS} in the context of zero range processes:
For all $\rho\geq 0$,  $n\in\N$, we have (for $\mu_\vfi$  given by \eqref{inv-zrp}) 
\begin{equation}\label{eoe}
\lim_{\substack{ L,N\rightarrow\infty\\ N/L\rightarrow\rho}}\mu_{N,L}(\eta(x_1)=k_1,...,\eta(x_n)=k_n)=
\prod_{i=1}^n \mu_\vfi(k_i) \qquad \text{ for all } k_1,\ldots,k_n\in\N.
\end{equation}
In the particular case of a ZRP with rate
 \begin{equation}\label{ex:g_b>2}
 g_n^1=1+\frac{b}{n}\qquad\mbox{ for some }\quad b>2,
 \end{equation}
 by  \eqref{inv-zrp},  the product invariant measures $\mp$ have single site marginals given by 
\begin{equation}\label{generic-cond}
\mu_{\vfi}(k) = Z_{\vfi}^{-1}\frac{\vfi^k}{\prod_{i=1}^k \left(1+\frac{b}{i}\right) }.
\end{equation}
The local convergence in \eqref{eoe} was improved in \cite{AL} in (strong) convergence 
in the total variation norm: for $\rho>\rho_c$
\begin{equation}\label{loul-r}
\lim_{\substack{ L,N\rightarrow\infty\\ N/L\rightarrow\rho}}\left\|\mu_{N,L} - \frac{1}{L} \sum_{x\in\Lambda} 
\left(\mu_{\varphi_c}^{\Lambda\setminus\{x_L\}}\ast\delta_{N-\sum_{\Lambda\setminus\{x_L\}} \eta }  \right)
\circ \sigma^{x,x_L}\right\|_{\rm{t.v.}} = 0, 
\end{equation}
where $x_L$ is the last site of linearly ordered $\Lambda$, 
$\mu_{\varphi_c}^{\Lambda\setminus\{x_L\}}\ast\delta_{N-\sum_{\Lambda\setminus\{x_L\}} \eta }$ 
is a probability measure on $\N^{\Lambda}$  with marginal on  $\N^{\Lambda\setminus\{x_L\} }$ given by
$\mu_{\varphi_c}^{\Lambda\setminus\{x_L\}}$ and one site marginal on $ \{x_L\} $ given by the Dirac measure 
at $N- \sum_{x\in\Lambda\setminus\{x_L\}}  \eta(x)$,
 and $\sigma^{x,x_L}$ is a permutation of configurations given by
$\left( \sigma^{x,x_L} \eta \right) (z) = 
\begin{cases} \eta(z)\ \ \text{ for } z\neq x,x_L \\ \eta(x_L)  \ \text{ for } z=x \\
\eta(x) \ \ \text{ for } z=x_L. \end{cases}$
\\
If moreover  $b>2$, \eqref{loul-r} implies a stable limit theorem 
for $M_L(\eta)=\max_{x\in\Lambda} \eta(x)$ (see \cite{AL}). \\ 

\noindent
 For the second set-up for condensation, 
consider the case $b\in (1,2]$ for the measure $\bar{\mu}_{\vfi}$ 
defined by \eqref{generic-cond}. As mentioned in \cite{GSS}, 
$Z_{\vfi}<\infty$ for every $\vfi\in[0,1]$, that is, $\vfi_c=1$, and 
there is no finite critical density  but \\[2mm]
\centerline{$\mu_1(0)=\frac{b-1}{b}\quad \& \quad 
\mu_1(k) \sim \Gamma(b) (b-1) k^{-b} $ for $k>0$.}
\\[2mm]
Quoting from \cite{GSS}, ``a typical configuration for this stationary 
distribution has a hierarchical structure which can be understood as a precursor 
for a condensation phenomenon''. 
It is proved in \cite{FeLaSi} that for a fixed number of sites   $L\geq 2$,  the probability measures 
\begin{equation}\label{L_fix_conver}
\hat{\mu}_{N,L-1}\ \text{ converge weakly to } \hat{\mu}_{1}^{L-1} \ \text{ as  } N\rightarrow\infty
\end{equation}
where $\hat{\mu}_{N,L-1}$, $\hat{\mu}_{1}^{L-1}$ are measures on 
$\N^{\{1,...,L-1\}}$ derived  respectively from  $\mu_{N,L}$ and  
${\bar\mu}_{1}^{\{1,...,L\}}$ on $\N^{\{1,...,L\}}$, applying an operator 
on them which orders the coordinates of a configuration and  removes the last one 
(the maximal one).  It means that all but finitely many particles accumulate at one site.

 \subsection{Back to MMP}\label{subsec:cond-mmp}
As indicated at the beginning of this section, let us give some examples of MMPs 
exhibiting condensation. 
For instance, consider again the generic model for condensation used firstly 
in \cite{evans0} and later in  \cite{GSS} and  \cite{AL}, that is, ZRP with rate \eqref{ex:g_b>2}
and invariant measures $\mp$  given by their single site marginals \eqref{generic-cond}.
Then  many possible MMPs have this stationary state and hence they  may exhibit condensation. \\

\noindent $\bullet$ If we choose a MM-ZRP, by \eqref{inv-s} in Proposition \ref{gzrp}, it has rates
\begin{equation}\label{gzrp-gen} 
g^k_{\a+k} = c(k) \frac{\mu(\a)}{\mu(\a+k)} =  c(k)\ \prod_{i=\a+1}^{\a+k} \left( 1+\frac{b}{i}\right)
\end{equation} 
for $k\geq 1,\, \a\geq 0$ where $c$ is an arbitrary positive function on $\N\setminus\{0\}$ and $\mu$ is the special case of \eqref{generic-cond} for $\vfi=1$. 
We study this case in details later on, as Example \ref{gen-zrp} 
in Section~\ref{s-examples}.\\

\noindent $\bullet$ If we choose a MM-TP, we obtain from  Proposition \ref{target-exact}(a), using that the ratio $\mu(n+1)/\mu(n)$
 is nondecreasing, a recursive formula for its rates
\begin{eqnarray}
g^{\a}_{*,0} &=& c^*(\a) \nonumber \\[-1mm]
g^\a_{*,\b}  &=& c^*(\a) + \sum_{k=1}^{\b} \frac{H_{\a}(\b,k)}{\mu(\b)}\, c^*(k),
\qquad  \b\geq 1,  \label{target-g.e.}
 \end{eqnarray}
for every $\a\geq 1$, where  $c^*$ is an arbitrary positive function on $\N\setminus\{0\}$ and 
the function $H$ is given by formulas \eqref{tp:Delta}--\eqref{tp:hexp}. \\ 

For both these  MM-ZRP and MM-TP, the results \eqref{eoe}, 
\eqref{loul-r}, \eqref{L_fix_conver} hold,  that is, they may exhibit condensation in both set-ups. \par

\begin{remark}\label{rk:duality-MM-ZRP-MM-TP}
We  observe a duality between MM-ZRP and MM-TP having 
rates $g^k_{\a}$ 
and  $g^k_{*,\b}$ 
which lead to the same family of product, 
translation-invariant, and invariant measures. But it is no longer a simple 
formula  as it was the case for the duality between ZRP and TP explicited
in Remark \ref{rk:duality-zrp-tp}.  Combining \eqref{gzrp-gen} 
and \eqref{tp:Delta} we  obtain that 
$$\Delta_{\a}(k) = c(k) \left( \frac{1}{g^k_{\a+k}} - \frac{1}{g^k_{\a-1+k}} \right). $$
By \eqref{target-g.e.}, we then obtain a formula for $ g^\a_{*,\b}$ 
\begin{eqnarray*} g^\a_{*,\b} = c^*(\a) + \sum_{k=1}^{\b} c^*(k) 
 \Bigl[ g^k_{\b} \left( \frac{1}{g^k_{\a+k}} - \frac{1}{g^k_{\a-1+k}} \right)
+ \sum_{r=1}^{\b-k} \!\!\sum_{\substack{ k_1,\dots,k_r\ge 1 \\ k_1+\dots+k_r\le \b-k}}\!\!\! 
  \frac{c(k_1)\cdots c(k_r)c(k)}{c(k+\sum_{j=1}^r k_j)}  g^{k+ \sum_{j=1}^r k_j}_{\b}  \\ 
  \left( \frac{1}{g^{k_1}_{\a+k_1}} - \frac{1}{g^{k_1}_{\a-1+k_1}} \right) \prod_{j=1}^{r-1} 
  \left( \frac{1}{g^{k_{i+1}}_{k_i+k_{i+1}}} - \frac{1}{g^{k_{i+1}}_{k_i-1+k_{i+1}}} \right)
  \left( \frac{1}{g^{k_{r}}_{k_r+k}} - \frac{1}{g^{k_{r}}_{k_r-1+k}} \right)\Bigr]
\end{eqnarray*}   
which is  a function of $(g^k_{\a+k},g^k_{\a+k-1},g^l_k:1
\leq k\leq \b,\, 1\leq l\leq k )$, $(c(k):1\leq k\leq\b)$ and $(c^*(\a),c^*(k):1\leq k\leq\b)$
only.
\end{remark}
\subsection{Condensation vs. attractiveness.}
 An intriguing question  is whether an MMP could be at the same time attractive
and exhibit condensation. \\
The answer is `no' in the two following cases.
\begin{proposition}\label{no-attract-cond-single-jump-and-mm-tp} 
 
(a) A (single jump) ZRP or  MP
cannot at the same time be attractive and exhibit condensation.

(b) An MM-TP cannot at the same time be attractive and exhibit condensation.
\end{proposition}
\begin{proof}
(a)  Recall Remark \ref{rk:single-attra}. We also  saw  in Section
 \ref{subsubsec:single-jump} 
 that if there exist product invariant measures of the form \eqref{muf}
 then  (see \eqref{k=1:mug} and \eqref{zrp-mu-invt}),
 $$r(n)=\frac{\mu(n)}{\mu(n+1)} =\varphi_\mu\, g^1_{n+1}\quad\mbox{ for ZRP,}\qquad 
 r(n)=\frac{\mu(n)}{\mu(n+1)} = \overline{\varphi}_\mu\,\frac{g^1_{n+1,0}}{g^1_{1,n}}\quad\mbox{ for MP}.$$ 
 We conclude that an attractive ZRP or an attractive MP 
 has $r(n)$ nondecreasing and so $Z_{\vfi_c} = +\infty$
 by usual criteria for convergence of series. Thus they cannot exhibit condensation.\\
(b) 
Assume a MM-TP has an invariant product measure with marginals $\mu(\cdot)$. If 
the rates $g^k_{*,\b}$ satisfy the conditions \eqref{atr:GTP-1}--\eqref{atr:GTP-2} for attractiveness
then $r(n)=\mu(n)/\mu(n+1)$ is a nondecreasing function 
and therefore $Z_{\vfi_c}=+\infty$.
Indeed, from \eqref{atr:GTP-1}--\eqref{atr:GTP-2} we obtain in particular that 
$g^{\a}_{*,\b} -   g^{\a}_{*,0} \leq 0 $ for  all $\a,\b\ge 1$. 
Combining this with \eqref{tp:rec1} gives the result. 
\end{proof}
\smallskip
 Let us consider the second set-up for condensation.
For the latter it is  proven in \cite{rcg} that: ``All spatially homogeneous processes with
stationary product measures that exhibit condensation 
with a finite critical density are necessarily not attractive''. 
 However,  processes  with an infinite critical density  could be also attractive. 
In \cite{rcg}  an example is numerically studied to support this assertion.
 In Section~\ref{s-examples} below, we prove on a particular example, that this coexistence is indeed possible:
 \begin{proposition}\label{prop:attra+cond}
 The MM-ZRP with rate 
 \begin{equation}\label{rate:attra+cond}
\begin{cases} \gka = \pi(1)\frac{\prod\limits_{i=\a-k}^{\a-1} \Bigl(1+\dfrac{b}{i}\Bigr)}
 {\prod\limits_{i=1}^{k-1} \Bigl(1+\dfrac{b}{i}\Bigr)}\,\1_{[k<\a]}
 \quad \text{ for }\ \a\geq 1, k\geq 1, \\
g_{\a}^{\a}=\pi(0)\quad \text{ for }\ \a\geq 1.
 \end{cases}
\end{equation} 
where $\pi$ is a positive function on $\N$ with $\pi(0) \geq (1+b)\pi(1)$
is attractive and exhibits condensation in the second set-up for  $b=3/2$.
\end{proposition}  
 %
 %
 %
\section{Examples} \label{s-examples}
 In this section, we present some examples of MMPs on $\Z^d$, 
for which we check existence conditions, whether they have product invariant probability measures,
whether they are attractive, and whether they exhibit  condensation. 
 
 For each dynamics, we thus verify the following conditions:
\begin{itemize} \itemsep2pt
\item[(i)] \textit{sufficient condition  \eqref{assum}  for  existence;}  \par 
The  following 
condition is equivalent to \eqref{assum}: there exists $C>0$ such that 
\begin{equation}\label{assum-ter} 
 \sum_{k=1}^{\a+1} k\,|\gen^k_{\a+1,\b} - \gen^k_{\a,\b}|
+   \sum_{k=1}^{\a} k\,|\gen^k_{\a,\b+1} - \gen^k_{\a,\b}| \, \leq \,  C
\quad \text{ for all } \a, \b\geq 0. 
\end{equation}
It means that there exists $C>0$ such that 
\begin{equation}\label{assum-quatre} 
\sum_{k=1}^{\a} k\,|\gen^k_{\a+1,\b} - \gen^k_{\a,\b}| \leq C,\ \ \
\sum_{k=1}^{\a} k\,|\gen^k_{\a,\b+1} - \gen^k_{\a,\b}| 
\leq  C \text{ and } \a \gen^{\a}_{\a,\b} \leq C \text{ for all }\a,\b\geq 0.
\end{equation} 
 \item[(i)'] \textit{sufficient condition \eqref{un-assum} 
 (which implies \eqref{bound-to-generator}) and condition \eqref{assum-alternative}}; 
\item[(ii)] \textit{necessary and sufficient conditions 
for attractiveness}, by Lemma \ref{lemma:attract}, \eqref{atr:mmp-1}-\eqref{atr:GTP-2}; 
\item[(iii)] \textit{necessary and sufficient conditions for}
 the existence of product, translation-invariant, \textit{invariant measures}, 
by  the appropriate result in Section \ref{s-invariant}.
\end{itemize}
If there are such product invariant measures, we compute 
their single site marginals $\mu_\varphi$, and 
$\varphi_c,\rho_c$. We also check  condensation properties.\\
\begin{remark}\label{rk:compar-cond}
For the various conditions in (i) and (i)', we have that
  \eqref{assum} as well as \eqref{assum-alternative}   implies \eqref{un-assum}, but converses are wrong. 
None of the implications between  \eqref{assum} and \eqref{assum-alternative} hold. Indeed, 
taking $g^k_{\a,\b}=0$ for $k>1$, and $g^1_{\a,\b}=\a+\b$, we have that \eqref{assum}
 and \eqref{un-assum} hold, but not \eqref{assum-alternative}.
Considering MM-ZRP with $\gka$ as in Example \ref{ex:GZRP-2} below, with $r(\a)=1/\a $  for $\a$ even, and
 $r(\a)=1/\a^2$ for $\a$ odd, we have that \eqref{assum-alternative} and \eqref{un-assum} hold, but not
 \eqref{assum}. 
\end{remark}
\smallskip 
\noindent
We now study four examples of MM-ZRP, and one of MM-TP, with the following features:
condensation is impossible  in Example \ref{ex:GZRP-1}; there are no product
translation-invariant invariant probability measures  in
Example \ref{ex:GZRP-2} (with the exception of the stick process); 
Example  \ref{trojka} possesses a one-parameter family
of product translation-invariant invariant probability measures; 
Example \ref{gen-zrp}  is a particular case of Example  \ref{trojka},
for which attractiveness coexists with condensation for one value of its parameter;
 Example 5 is an MM-TP dual to the MM-ZRP of Example 4.\\ 

\begin{example}\label{ex:GZRP-1}
MM-ZRP with $\gka = h(k)\,\1_{[k\leq\a]}$ for $\a\geq 1,k\geq 1$, $h$ nonnegative function.
\end{example}
\begin{itemize} \itemsep2pt
\item[(i)]  If there exists $C>0$ such that ${\a}h(\a)\leq C$
 for every $\a$, 
then condition \eqref{assum}   is satisfied.       \\
\item[(i)'] If there exists $C>0$ such that $\sum_{k=1}^{\a} k h(k) \leq C\a\ $ 
(resp. $\sum_{k=1}^{\a} h(k) \leq C$) for every $\a$, 
then condition \eqref{un-assum} (resp. \eqref{assum-alternative}) is satisfied.     \\
\item[(ii)] The process is attractive if and only if 
\begin{equation}\label{ex1-increasing}
 h(k) \,\text{ is nonincreasing for }\, k\geq 1.
\end{equation}
%
%
\item[(iii)] 
 Conditions \eqref{eq:g1_zr0} and \eqref{inv-rate-zrp} are satisfied, hence by Proposition
\ref{gzrp-exact}(b) 
there are product, translation-invariant, invariant probability measures given by \eqref{inv-zrp}, 
which are geometric:
\\[.1cm]
\centerline{$\mu_{\varphi}(\eta(x)=n)= \varphi^{n}(1-\varphi)\quad$  for all $x\in\Z^d,n\in\N$,}
\\[.1cm]
 for every $\vfi\in (0,1)$.                                                              
They do not depend on $h$. 
The normalizing constant is $Z_{\varphi}=(1-\varphi)^{-1}$; we have $\vfi_c=1$ and
$Z_{\vfi_c} = \infty$; so  there is 
 no possibility of condensation. 
\end{itemize} 
\medskip
\noindent\textit{Particular cases:}

\noindent$\bullet$ The process with $h(k)=1/k$ is  attractive;\\
\noindent$\bullet$ the process with $h(k)=k$ for $k<k_0$, $h(k)=1/k$ for $k\geq k_0$, 
is not attractive;\\
\noindent
and both  satisfy the existence condition \eqref{assum}.\\
\noindent$\bullet$ 
The totally asymmetric $q$-Hahn ZRP is one of the models  
studied in \cite{BC}; there $p(1)=1$, with 
$\displaystyle{h(k)= \frac{q^{k-1}(1-q)}{1-q^k}}$ for $q\in(0,1)$.
The totally asymmetric $q$-Hahn ZRP is attractive, and condition \eqref{assum} is satisfied. \\

\begin{example}\label{ex:GZRP-2}
MM-ZRP with $\gka = R(\a)\,\1_{[k\leq\a]}$ for $\a\geq 1,k\geq 1$ where $R$ is a nonnegative function.
\end{example}
\noindent 
\begin{itemize} \itemsep2pt
\item[(i)] Condition \eqref{assum} is satisfied if there exists $C>0$ such that for $\a\geq 0$,
$$ \a^2\,|R(\a+1)-R(\a)| + \a R(\a)\leq C.$$ 
\item[(i)']  Conditions
\eqref{un-assum} and \eqref{assum-alternative} are satisfied  if 
there exists $C>0$ such that for $\a\geq 0$, $\a R(\a)\leq C$. \\ 
 Note that if  $R$ is a constant function, 
both conditions \eqref{assum}, \eqref{un-assum} and \eqref{assum-alternative} fail. 
It is the case for the ``stick process'' introduced in \cite{Sepa1}
and also studied in \cite{GS}; for this example, another construction was provided
in \cite{Sepa1}, on a state space smaller than $\XX$. \\ 
\item[(ii)] The process is attractive if and only if 
\begin{equation}\label{eq:attra-ex2}
R(\a)\mbox{  is nonincreasing  }\quad\&\quad\a R(\a)\mbox{   nondecreasing}
\quad \mbox{for     }\a\geq 1.
\end{equation}
%
%
\item[(iii)]   Unless $R$ is constant (that is for the stick process),
Condition \eqref{inv-rate-zrp} is never satisfied
(take $k=\a-1\ge 1$ in \eqref{inv-rate-zrp}),  so there are 
no product translation-invariant and invariant probability measures. 
\end{itemize}
\medskip
\noindent\textit{Particular cases:}

\noindent$\bullet$ The process with $R(\a)=1/\a $  is  attractive;\\
\noindent$\bullet$ The process with $R(\a)=1/\a^2$ is  not attractive;\\
\noindent
and both  satisfy the existence condition \eqref{assum}.\\

\begin{example} \label{trojka}
Generic MM-ZRP with product invariant measures:  we consider the rates
\eqref{ex:ZRh} for which we began to investigate attractiveness 
in Section \ref{subsec-attra-cond}, that is 
\begin{equation}\label{ex3-rates} 
 \gka = \frac{\pi(\a-k)}{\pi(\a)}\,h(k)\,\1_{[k\leq\a]}
 \quad \text{ for }\ \a\geq 1, k\geq 1,
\end{equation}
where $\pi$ is a positive function on $\N$ and $h$ is a 
non-negative function on $\N\setminus\{0\}$. For technical 
reasons we define $\pi(i)=0$ for $i\in\Z^-$ and $h(0)=0$.
\end{example}

\noindent
\begin{itemize} \itemsep2pt
\item[(i)] Condition \eqref{assum} is satisfied if there exists $C>0$ such that for $\a\geq 0$,
 \begin{equation} \label{ex3-existence}  
 \sum_{k=1}^{\a} k\,h(k)\, |\frac{\pi(\a+1-k)}{\pi(\a+1)} 
 - \frac{\pi(\a-k)}{\pi(\a)}| + \a\frac{h(\a)}{\pi(\a)}\leq C
 \end{equation} 
\item[(i)']  Condition  \eqref{un-assum} writes: There exists  $C>0$ such that
 for every $\a\geq 1$, 
\begin{equation}\label{9-for-ex3}
\sum_{k=1}^{\a} k h(k) \pi(\a-k) \leq C\a \pi(\a).
\end{equation}  
Condition \eqref{assum-alternative} writes: there exists $C>0$ such that for every $\a\geq 1$, 
\begin{equation}\label{ex3-existence-bis}
\sum_{k=1}^{\a}\pi(\a-k)h(k)\leq C \pi(\a).
\end{equation}
\item[(ii)]   In Section \ref{subsec-attra-cond} we proved that
the process is attractive if and only if Condition \eqref{attractGZR} is satisfied.
In Lemma \ref{attr-al-gzrp} we   simplified this condition when the function 
$r(n)=\pi(n)/\pi(n+1)$ (see \eqref{def:r-for-ZRh}) is monotone.  \par
\smallskip
\item[(iii)] 
The rate \eqref{ex3-rates} is 
the generic one to satisfy assumption (\ref{inv-s}) hence we have 
a one-parameter family of product and translation-invariant measures  
\begin{equation}\label{ex3-invar}
\begin{array}{lc}
& \{\bar{\mu}_{\vfi}: \vfi\in {\rm Rad}(Z')\},  \quad  \text{or  }
\{\bar{\mu}_{\vfi}: \vfi\in {\rm Rad}(Z)\}, \\
\text{where} & \\ 
& \bar{\mu}_{\varphi}(\eta:\eta(x)=n)= Z_{\varphi}^{-1}\varphi^{n}\pi(n) \text{ for all } x\in\Zd,n\in\N
\end{array}
\end{equation}
and ${\rm Rad}(Z')$,  ${\rm Rad}(Z)$ are given by \eqref{Rad'} and \eqref{Rad}. 
These measures do not depend on  $h$. Properties of the ratio $r(n)$ yield 
the existence of a critical value  $\vfi_c$ such that $Z_{\vfi_c}<\infty$
or moreover $\rho_c<\infty$. If  $\vfi_c\in {\rm Rad}(Z)\setminus\! {\rm Rad}(Z') $ 
then the measure $\bar{\mu}_{\varphi_c}$ is well defined by \eqref{ex3-invar} but its 
first moment is infinite. If assumption  \eqref{eitherif} is satisfied, 
then  $\bar{\mu}_{\varphi_c}$ is invariant.
\end{itemize}
\medskip
%
\begin{example} \label{gen-zrp}   
A special case of Example \ref{trojka}: MM-ZRP with product invariant measures where
\begin{equation}\label{def:ex4}
r(n)=\frac{\pi(n)}{\pi(n+1)} = 1+\frac{b}{n} ,\quad b>1, \text{ for all } n\geq 1\,;
\qquad r(0)=\frac{\pi(0)}{\pi(1)} \geq 1+b. 
\end{equation} 
We are particularly interested  in the case $h(k)=\pi(k),\ k\geq 1$,  which corresponds to
example \eqref{ex:ZRh0} with $h_0=1$, or to \eqref{gzrp-gen} (see also Remark \ref{rk:gkk}). 
\end{example}

\noindent
We have for $n\geq 1$,
\begin{equation}\label{def:ex4-withga}
\pi(n)=\dfrac{\pi(1)}{\prod\limits_{i=1}^{n-1} \Bigl(1+\dfrac{b}{i}\Bigr) }
=\pi(1) \dfrac{(n-1)!}{\prod\limits_{i=1}^{n-1} (b+i)} = 
\pi(1) \dfrac{\Gamma(n)\Gamma(b+1)}{\Gamma(b+n)} .
\end{equation}
By Stirling approximation,   when $n\rightarrow \infty$ we have 
$\pi(n)\sim \pi(1)\Gamma(b+1) n^{-b}$. Therefore there exist positive constants $\omega_1, \omega_2 $ 
depending on $b$ such that, for all $n>0$,  
\begin{equation}\label{eq:stirl-ga}
\omega_1 n^{-b} \leq \pi(n) \leq \omega_2 n^{-b} 
\end{equation} 
\begin{itemize} \itemsep2pt
\item[(i)] 
Assumption \eqref{assum} for existence is equivalent to (cf. \eqref{ex3-existence}, \eqref{assum-quatre}):
there exists $C>0$ such that, for every $\a\geq 2$, 
\begin{equation}\label{ex4-suff1}
 \sum_{k=1}^{\a-1} k\,h(k)\, \left( \frac{\pi(\a-k)}{\pi(\a)} - \frac{\pi(\a+1-k)}{\pi(\a+1)}\right) 
 \leq C \ \ \text{ and } \ \ \a\frac{h(\a)}{\pi(\a)}\leq C.
\end{equation}
Since we have from \eqref{def:ex4-withga} 
\begin{eqnarray*}
\frac{\pi(\a-k)}{\pi(\a)} - \frac{\pi(\a+1-k)}{\pi(\a+1)} 
&=& kb \frac{\Gamma(\a-k)\Gamma(b+\a)}{\Gamma(\a+1)\Gamma(b+\a+1-k)},
\end{eqnarray*}
then, using \eqref{eq:stirl-ga}, condition \eqref{ex4-suff1} is equivalent to : there exists $C>0$
such that, for every $\a\geq 2$,
\begin{equation}\label{ex4-exist-cond2}
 \a^{b-1} \sum_{k=1}^{\a-1} k^2 h(k) \frac{1}{ (\a-k)^{b+1}} 
 \leq C\ \text{ and } \ \a h(\a)\leq C \pi(\a).
\end{equation}  
With the choice $h(k)=\pi(k)$, the  sufficient condition \eqref{ex4-suff1} 
leads to  $\a h(\a)\leq C \pi(\a)$ which is never satisfied. 
 But with the choice $h(k)=\pi(k)/k$  that we put  in \eqref{ex4-exist-cond2}
combined with \eqref{eq:stirl-ga} for $\pi(k)$,
then  condition \eqref{ex4-suff1}  writes 
\begin{equation}\label{ex4-exist-particular}
\forall\,\a\geq 2,\quad \sum_{k=1}^{\a-1} f_{k,\a}(b) \leq C,
 \qquad\mbox{ where }\qquad
 f_{k,\a}(b)= \frac{\a^{b-1}}{k^{b-1} (\a-k)^{b+1}}
\end{equation}
for some $C>0$.  It is satisfied for  $b\geq 1$ since we have that 
\begin{lemma}\label{lem:9.1}
For every $b\geq 1$ 
there exists $C>0$ such that \eqref{ex4-exist-particular} is valid.
\end{lemma}
\begin{proof}
 For $b=1$, we have that
$$ \sum\limits_{k=1}^{\a-1} f_{k,\a}(1)= \sum\limits_{k=1}^{\a-1}\frac{1}{k^2} 
$$
For $b>1$, we have that for $1\le k \le \a-1$, $f_{k,\a}(b)\le \hat{f}_{k,\a}(b-1)$
(see \eqref{ex4-main-cond-12} below), hence the result follows from the proof of
Lemma \ref{lem:9.2} below. 
\end{proof}
\item[(i)'] 
In the case $h(k)=\pi(k) $, condition \eqref{un-assum}  is satisfied if and only if for some $C>0$ 
\begin{equation} \label{ex4-12}
 \sum_{k=0}^{\lfloor\a/2\rfloor} g_\a^k =
\sum_{k=0}^{\lfloor\a/2\rfloor} \frac{\pi(k)\pi(\a-k)}{\pi(\a)} \leq C \ \ \text{for all } \a\geq 2.
\end{equation}                                                                                                                  
Indeed, \eqref{un-assum} writes \eqref{9-for-ex3}, and if we use that the term 
$\frac{\pi(k)\pi(\a-k)}{\pi(\a)}$ is symmetric we obtain \eqref{ex4-12}, which,
by \eqref{def:ex4-withga}, is equivalent to:  for some $C>0$,
$$ \frac{\Gamma(b+\a)}{\Gamma(\a)}\sum_{k=1}^{\lfloor\a/2\rfloor} \frac{\Gamma(\a-k)}{\Gamma(b+\a-k)} 
\frac{\Gamma(k)}{\Gamma(b+k)} \leq C $$
 which, using Stirling approximation \eqref{eq:stirl-ga}, is equivalent to
\begin{equation}\label{ex4-main-cond-12}
\forall\,\a\geq 2,\quad\sum_{k=1}^{\lfloor\a/2\rfloor} \hat{f}_{k,\a}(b)\leq C,
 \qquad\mbox{ where }\qquad
 \hat{f}_{k,\a}(b)= \frac{\a^{b}}{k^{b} (\a-k)^{b}}
\end{equation}
for some $C>0$. We have that 
\begin{lemma}\label{lem:9.2}
For every  $b> 0$ there exists $C>0$ such that \eqref{ex4-main-cond-12} is satisfied.
 \end{lemma}
\begin{proof}
 For $b> 0$ and $\a\geq 2$, the function $x\mapsto x^{-b} (\a-x)^{-b}$ 
is decreasing on $(0,\a/2)$. Therefore 
$$\sum_{k=1}^{\a-1} \hat{f}_{k,\a}(b)\le 2\sum_{k=1}^{\lfloor\a/2\rfloor} \hat{f}_{k,\a}(b)\le
\left( \frac{\a}{\a-1}\right)^{b}+2\int_{1/(2\a)}^{1/2}   \frac{1}{u^{b} (1-u)^{b}}   du .
$$
which implies the lemma. 
\end{proof}
Let us now check condition \eqref{assum-alternative}, which writes here
\eqref{ex3-existence-bis}.
With the choice $h(k)=\pi(k)$, it becomes 
\eqref{ex4-12}. If it is satisfied, it will also be the case for \eqref{assum-alternative}
 with the choice $h(k)=\pi(k)/k$.\\
\item[(ii)] 
The process is attractive if and only if \eqref{at-1}\,\&\,\eqref{at-2} from Lemma \ref{attr-al-gzrp}
hold, since $r(\a)$ is decreasing. 
In the case $h(k)=\pi(k) $, Lemma \ref{attr-al-gzrp}\textit{(a')} gives a necessary 
and sufficient condition for the process to be attractive, namely \eqref{at-3}. 
By \eqref{arranged-at-3} and \eqref{def:ex4-withga}, for $\a\ge 2$ we have, writing from now on  $S_b(\a)$ 
instead of $S(\a)$ to stress
the dependence in $b$, 
\begin{eqnarray} \nonumber
S_b(\a)&= &\pi(1) \sum_{j=1}^{\a-1} \frac{r(\a-1)\cdots r(\a-j)}{r(j-1)\cdots r(1)} \\ 
&=& \pi(1)\frac{\Gamma(b+\a)\ \Gamma(b+1)}{\Gamma(\a)} 
\sum_{j=1}^{\a-1} \frac{\Gamma(j)\ \Gamma(\a-j)}{\Gamma(j+b)\  \Gamma(\a-j+b)} \label{Sab}
\end{eqnarray}

\begin{proposition}\label{prop:ex4-attra}
The process is attractive for $b=1$ and $b=3/2$ and  not attractive 
for all values of $b$ larger or equal to 2. 
\end{proposition}
\begin{proof} \,

First, for all $\a\ge 2$ and all $b>0$,  using that $\Gamma(z+1)=z\Gamma(z)$
for all $z> 0$,  we compute the increment
\begin{eqnarray}\nonumber 
&& S_b(\a+1)-S_b(\a)=\\\nonumber
&&=   \pi(1)\frac{b+\a}{\a} + 
 \pi(1)\sum_{j=1}^{\a-1} \frac{\Gamma(b+1) \Gamma(j)}{\Gamma(j+b)}\Bigl(
\frac{\Gamma(b+\a+1) \Gamma(\a+1-j)}{\Gamma(\a+1) \Gamma(\a+1-j+b)}
-\frac{\Gamma(b+\a) \Gamma(\a-j)}{\Gamma(\a) \Gamma(\a-j+b)}\Bigr)\\
&&= \pi(1) \frac{b+\a}{\a} + 
\pi(1) \sum_{j=1}^{\a-1} \frac{\Gamma(j) \Gamma(\a-j)}{\Gamma(\a)}\;
\frac{\Gamma(b+1) \Gamma(b+\a)}{\Gamma(j+b) \Gamma(\a-j+b) }
\Bigl(\frac{(b+\a) (\a-j)}{\a (\a-j+b)} -1\Bigr)\nonumber \\
&&= \pi(1) \frac{b+\a}{\a} + 
\pi(1) \sum_{j=1}^{\a-1} \frac{\Gamma(j) \Gamma(\a-j)}{\Gamma(\a)}\;
\frac{\Gamma(b+1) \Gamma(b+\a)}{\Gamma(j+b) \Gamma(\a-j+b) }
\Bigl(\frac{ - b j}{\a(\a-j+b)}\Bigr)\nonumber \\
&&= \pi(1) \frac{b+\a}{\a} -
b\; \pi(1) \sum_{j=1}^{\a-1} \frac{j! (\a-j-1)!}{\a!}\;
\frac{\Gamma(b+1) \Gamma(b+\a)}{\Gamma(j+b) \Gamma(\a-j+b+1) }\nonumber 
\end{eqnarray}
For $\a=2$ we get $S_b(3)-S_b(2)= \pi(1) \ge 0$ for all $b>0$. For all $\a\ge 3$, we get
\begin{eqnarray}\nonumber
&&S_b(\a+1)-S_b(\a)\\\label{att-ss}
&&=\pi(1) \Bigl(1+\frac{b(\a-2)}{\a(\a-1)}\Bigr) -
b\; \pi(1) \sum_{j=2}^{\a-1} \frac{j! (\a-j-1)!}{\a!}\;
\prod_{k=1}^{j-1} \frac{b+\a-k}{b+j-k}
\end{eqnarray}
$\bullet$ For $b=1$, the expression \eqref{att-ss} becomes
\begin{equation}
S_1(\a+1)- S_1(\a) = \pi(1) \Bigl(1+\frac{(\a-2)}{\a(\a-1)}
 - \frac{\a-2}{\a-1} \Bigr)=\frac{ 2 \pi(1)}{\a} \ge 0
\end{equation}
\noindent 
Therefore the inequality in \eqref{at-3} is satisfied for all $\a\ge 2$,
 and the process is attractive.  

\medskip\noindent
$\bullet$ We now consider $b=3/2$. The expression  \eqref{att-ss}  reads
\begin{eqnarray}\label{att-ss32}
&&S_{3/2}(\a+1)-S_{3/2}(\a)\nonumber\\
&&\qquad=\pi(1) \Bigl(1+\frac{3 (\a-2)}{2 \a(\a-1)}\Bigr) -
9 \pi(1) \frac{(2\a+1)!}{(\a!)^2} \sum_{j=2}^{\a-1} \frac{(j!)^2 (\a-j-1)!\,(\a-j+1)!}{(2j+1)!(2\a-2j+3)!} 
\end{eqnarray}
We define for all $\a\ge 3$ and all $j\in\{0,\cdots,\a-1\}$
\begin{equation}\label{phiak} 
\varphi_\a(j)= \frac{(3\a-2 j+2)}{6(\a+1)(\a+2)}\; 
\frac{j!\,(j+1)!}{(2 j+1)!} \;\frac{(\a-j)!\,(\a-j-1)!}{(2\a-2 j+1)!}
\end{equation}
In particular
\begin{eqnarray}\label{phiak_1} 
\varphi_\a(1)&=&  \frac{\a!\,(\a-2)!}{6(\a+1)(\a+2)(2\a-1)!}\\\label{phiak_a-1} 
\varphi_\a(\a -1)&=& \frac{\a!\,(\a-1)!(\a+4)}{36(\a+1)(\a+2)(2\a-1)!}
\end{eqnarray}
Now for all $j\in\{1,\cdots,\a-1\}$, we compute
\begin{eqnarray}\label{att-ss32}
&&\varphi_\a(j)-\varphi_\a(j-1)\nonumber\\
&&\qquad  =\frac{1}{3(\a+1)(\a+2)}\;\frac{(j!)^2 (\a-j-1)!\,(\a-j+1)!}{(2j+1)!(2\a-2j+3)!}\;\nonumber\\
&&\qquad  \qquad\times\Bigl((j+1)(3 \a-2 j+2)(2\a-2 j+3)-(2 j+1)(\a-j)(3\a-2 j+4)\Bigr)\nonumber\\
&&\qquad  = \frac{(j!)^2 (\a-j-1)!\,(\a-j+1)!}{(2j+1)!(2\a-2j+3)!}\;\nonumber
\end{eqnarray}
Using this expression in \eqref{att-ss32}, and then \eqref{phiak_1}, \eqref{phiak_a-1} gives
\begin{eqnarray}\label{att-ss32}
&&S_{3/2}(\a+1)-S_{3/2}(\a)\nonumber\\
&&\qquad=\pi(1) \Bigl(1+\frac{3 (\a-2)}{2 \a(\a-1)}\Bigr) -
9 \pi(1) \frac{(2\a+1)!}{(\a!)^2} \sum_{j=2}^{\a-1} \bigl(\varphi_\a(j)-\varphi_\a(j-1)\bigr)\nonumber\\
&&\qquad=\pi(1) \Bigl(1+\frac{3 (\a-2)}{2 \a(\a-1)}\Bigr) -
9 \pi(1) \frac{(2\a+1)!}{(\a!)^2} \bigl(\varphi_\a(\a-1)-\varphi_\a(1)\bigr)\nonumber\\
&&\qquad=\pi(1) \Bigl(1+\frac{3 (\a-2)}{2 \a(\a-1)} -
 \frac{(2\a+1) (\a^2+3 \a-10)}{2(\a-1)(\a+1)(\a+2)} \Bigr)\nonumber\\
&&\qquad=\frac{3\pi(1) (3\a+2)}{ \a(\a+1)(\a+2)}\nonumber\\
&&\qquad\ge 0
\end{eqnarray}
Thus, the inequalities \eqref{at-3} are satisfied and  the process is attractive for 
 $b=3/2$.\\

\noindent
$\bullet$ For the last case, for any given $b\ge 2$, we just compute
\begin{eqnarray}\label{att-ss10}
S_b(11)-S_b(10)&=&
\pi(1)\Bigl(1 + \frac{4 b}{45} - \frac{37 b (b+9)}{360 (b+1)} - \frac{31 b (b+9)(b+8)}{2520 (b+1)(b+2)}
\nonumber\\
&&\quad\qquad- \frac{ b (b+9)(b+8)(b+7)}{280 (b+1)(b+2)(b+3)}\nonumber\\
&&\quad\qquad
 - \frac{ b (b+9)(b+8)(b+7)(b+6)}{504 (b+1)(b+2)(b+3)(b+4)}\Bigr)\nonumber\\
&=&
\pi(1)
\frac{-750+(2-b)(4155+3039b+955 b^2+151 b^3+10 b^4)}{315 (b+1) (b+2)(b+3)(b+4)} <0 
\end{eqnarray}
Therefore $S_b(11)-S_b(10)<0$  for $b\ge 2$ and the process is not attractive for $b\ge 2$. 
\end{proof}
\begin{remark}\label{rk:plus}
Equation \eqref{att-ss10} shows that the process is not attractive also 
for values of $b$ slightly smaller than $2$.
In fact, further numerical computations also indicate that the quantity 
\eqref{att-ss} for larger values of $\a$ gets negative for smaller values of $b$. For instance, 
$S_b(201)-S_b(200)$ is found to be negative for $b\ge 1.55$. 
Thus we conjecture that the process is attractive for all values of $b \in (0, 3/2]$
and not attractive for all values of  $b>3/2$.
\end{remark}
\item[(iii)]  
 By part (iii) in Example \ref{trojka}, 
we have a family \eqref{ex3-invar} of product, translation-invariant, 
invariant probability measures where 
\begin{eqnarray*} 
&& \vfi_c=1,\\
&& Z_{1} <\infty \text{ {  if and only if } } \ b>1,\\
&&  \rho_c<\infty  \text{ {  if and only if}  } \ b>2.
\end{eqnarray*}
Since $\sum_{n\ge 1} n \pi(n) =+\infty $ when $b\leq 2$, 
we have ${\rm Rad}(Z')=(0,1]$ for  $b>2$ and ${\rm Rad}(Z')=(0,1)$ 
for $1<b\leq 2$.

Let us consider the case $1<b\leq 2$. If moreover
$h(k)=\pi(k) $, we have by  \eqref{ex4-12} with Lemma~\ref{lem:9.2} that 
$ \sum_{k\leq \a} g^k_{\a}  $ 
is bounded, hence,  by \eqref{ex4-12}, 
 assumption \eqref{assum-alternative} is satisfied.
Therefore, by Corollary \ref{family}, if assumption (\ref{A1}) is satisfied,
then the probability measure $\bar{\mu}_1$ is invariant, and 
 $\{\bar{\mu}_{\vfi}: \vfi\in  (0,1]\}$ (by \eqref{ex3-invar}) 
with $\pi$ given by \eqref{def:ex4-withga} are invariant measures.

In this example, we chose $\pi$  such that the process with $b>2$ exhibits  
condensation in the first set-up whereas the process with $b>1$  
exhibits  condensation in the second set-up (that is, on a fixed finite volume), 
which follows from results in \cite{FeLaSi,GSS}  (see Section \ref{s-discussion}). 
 We proved coexistence of attractiveness and condensation in the second set-up for
$b=3/2$ (cf. Proposition \ref{prop:attra+cond}). 
\end{itemize} 
\bigskip

\begin{example}\label{ex5}
Consider the MM-TP with $\gkab = \1_{[k\leq\a]} g^k_{*,\b}$.
\end{example}
\begin{itemize} \itemsep2pt 
\item[(i)]   Condition \eqref{assum} writes: there exists $C>0$ such that,
 for all $\b,\a\geq 0$, 
$$\sum_{k=1}^{\a} k\, |g^{k}_{*,\b+1} - g^{k}_{*,\b}|\ +\ \a g^{\a}_{*,\b} \ \leq C. $$
Note that if $g^{k}_{*,\b}$ is independent of $k$, this condition fails.\\
\item[(i)'] 
Condition \eqref{un-assum} writes: for all $\b,\a\geq 0$, $\sum_{k=1}^{\a} k g^k_{*,\b} \leq C (\a +\b)$.
\par\noindent
Condition \eqref{assum-alternative} writes: for all $\b,\a\geq 0$, $\sum_{k=1}^{\a} g^k_{*,\b} \leq C$.\\
\item[(ii)] The process is attractive if and only if  conditions \eqref{atr:GTP-1}--\eqref{atr:GTP-2}
are satisfied. 
Note that $g^{\a}_{*,\b} \leq g^{\a-1}_{*,\b+1}$ for $ \a\geq 1,\,\,  \b\geq 0,\, k\geq 0,$ 
is sufficient for \eqref{atr:GTP-2}. \\
\item[(iii)]  
If conditions \eqref{inv-tt}--\eqref{tp:hrec} or conditions \eqref{eq:weak-phil} on rates are satisfied
then there exist  product invariant probability measures given by formulas \eqref{inv-tar}--\eqref{inv-tar2}. 
\end{itemize} 
The MM-TP dual of Example \ref{gen-zrp} has rates \eqref{target-g.e.}
(see Remark \ref{rk:duality-MM-ZRP-MM-TP}). Contrary to duality for single-jump
dynamics which conserves attractiveness (see Remark \ref{rk:single-attra}),
we have here a non-attractive dynamics (by Proposition \ref{no-attract-cond-single-jump-and-mm-tp})
dual of a one which can be attractive (by Proposition \ref{prop:ex4-attra}). 
\\

\section{Appendix: Construction of the process} \label{sec:construction}
 In this section we shall prove Proposition \ref{prop:generator}, 
Theorem \ref{ex-theorem}, and Proposition \ref{intL=0}. 

For our construction, Consider an approximation of the infinite set 
$\X$ by an increasing sequence of finite sets 
\begin{equation}\label{seqX_n}
 \X_1 \subset \X_2 \subset \ldots\ \subset\X, \qquad\bigcup_n \X_n =\X 
\end{equation}
and 
 let us define (as in \cite{liggett-spitzer}) a restriction of
the transition probabilities $(p(x,y),x,y\in \X)$ on $\X_n$ by
\begin{equation}\label{p_n}
p_n(x,y)= \left\{
                      \begin{array}[c]{lll}
                p(x,x) + \sum_{z\not\in \X_n}p(x,z) & \quad x=y,\ x\in\X_n  \\
                1  & \quad x=y,\ x\not\in\X_n \\
                p(x,y)  & \quad x\neq y,\ x,y\in\X_n \\
                0  & \quad {\rm otherwise.}
                      \end{array}
          \right.
\end{equation}
The transition probabilities $(p_n(x,y),x,y\in \X_n)$ satisfy \eqref{Mconst} 
with $M$ replaced by $M+1$, because
$p_n(x,y)\le p(x,y)+\1_{[x=y]}$. They induce a restriction of the dynamics 
of MMP to $\X_n$, whose generator 
$\L_n$ is given by (\ref{generator}) with $p(\cdot,\cdot)$ replaced by  $p_n(\cdot,\cdot)$: 
for $f:\N^{\X}\rightarrow\R,\, f\in\Lip,\eta\in\N^{\X}$,
\begin{eqnarray} \nonumber   
\L_n f(\eta)&=& \sum_{x, y\in \X}\sum_{k>0} \ p_n(x,y)\
                \gen^k_{\eta(x),\eta(y)}
                \left(f\left(\etb\right)-f\left(\eta\right)\right)\\
&=&\sum_{x\neq y\in \X_n}\sum_{k>0} \ p(x,y)\
                \gen^k_{\eta(x),\eta(y)}
                \left(f\left(\etb\right)-f\left(\eta\right)\right).\label{ngenerator}
\end{eqnarray}
Since particles do not move outside $\X_n$,  this 
corresponds to a Markov process on the countable state space $\N^{\X_n}$. 
Hence there is a well defined  Markov semigroup of operators $(S_n(t),t\geq 0)$ 
on functions $f\in\Lip$, associated to $\L_n$:
  \begin{equation}     \label{nsemigroup}
      S_n(t)f(\eta) = f(\eta) + \int_0^t \L_n S_n(s) f (\eta) ds    .
  \end{equation}
Similarly, the  dynamics of the coupled process restricted to $\X_n$ has generator 
$\overline{\L}_{n}$ 
 given by (\ref{couplingL}) with $p(\cdot,\cdot)$ replaced by  $p_n(\cdot,\cdot)$: 
for $(\eta,\zeta)\in\XX \times \XX$, $h\in\overline{\Lip}$,
\begin{equation}\label{ncouplingL}
\overline{\L}_{n} h(\eta,\zeta) = \sum_{x\neq y\in \X_n} p_n(x,y)\, \sum_{k\geq 0} \sum_{l\geq 0}  
                                 G^{k,l}_{\eta(x),\eta(y),\zeta(x),\zeta(y)} 
\left(h\left(\etb,\zetb\right)-h\left(\eta,\zeta\right)\right).
\end{equation} 
Our goal is to define, in Subsection \ref{ex4}, 
limits of the generators $\L_n$ and semigroups $S_n(t)$, as $n\rightarrow\infty$. 
For this we will as a first step, in Subsection \ref{ex1},  
analyze the restriction of $(\eta_t)_{t\ge 0}$ to the finite
set of sites $\X_n$, keeping by a misuse of language 
the same notations $\L_n,S_n(t),(\eta_t)_{t\ge 0},\overline{\L}_{n}$ 
for the dynamics of state space $\N^{\X_n}$. 
Subsection \ref{ex3} is devoted to preliminary results.

\subsection{Preliminaries}  \label{ex3}
In this subsection we explicit some statements done in Section \ref{s-model},
and derive a key inequality about the coupling rates.\\ \\
$\bullet$ Let us exhibit a function $a$ on $\X$ satisfying \eqref{Mconst} and \eqref{sum-a_x-finite}.
Choose any $M>1+m_p$ (cf. \eqref{mpconst}) 
and define, for all $x, y\in \X$, $$\tilde{p}(x,y)=\frac{1}{1+m_p}\bigl(p(x,y)+p(y,x)\bigr).$$
Fix a reference site  $x_0\in\X$ and set 
\begin{equation}\label{def-a_x}
a_x= \sum_{n\ge 0} \left(\frac{1+m_p}{M}\right)^n \tilde{p}^{(n)}(x,x_0)
\end{equation}
where $\tilde{p}^{(n)}(.,.)$ are the $n$-step transition probabilities corresponding to
$\tilde{p}(.,.)$:
\begin{equation}\label{def:p^n}
\tilde{p}^{(0)}(x,x_0)=\1_{[x=x_0]}\quad\hbox{ and }\quad
\tilde{p}^{(n)}(x,x_0)=\sum_{x_1\in\X}\tilde{p}(x,x_1)\tilde{p}^{(n-1)}(x_1,x_0)\hbox{ for } n\ge 1.
\end{equation}
Then \eqref{Mconst} holds, as well as \eqref{sum-a_x-finite}.
\\ \\
$\bullet$ Note that the set $\Lip$ includes all bounded cylinder functions. 
More precisely let $f$  on  $\N^{\X}$ be such a function. It 
is  bounded by a value $m_f$, and its support is a finite set $V_f\subset \X$ 
such that $f$ depends only on values $(\eta(x):x\in V_f)$ and not on the whole $\eta\in \N^{\X}$. 
We can write 
\begin{equation}\label{cyl-bdd}
\|f(\eta) -f(\zeta) \| \leq 2m_f 
\sum_{x\in V_f} |\eta(x) - \zeta(x)| \leq 2m_f \|\eta-\zeta\| / \min_{x\in V_f} a_x.
\end{equation}
$\bullet$ Next Lemma provides the analogue for the coupled process
of   \eqref{assum}   for the MMP, and is crucial for the following results. 
\begin{lemma} \label{lem:key}
Under assumption   \eqref{assum},  
the rates $G_{\alpha,\b,\gamma,\delta}^{k,l}$ satisfy, for all $\a,\b,\g,\delta,k,l\geq 0 $,
\begin{equation}  \label{key}
   \sum_{k=0}^{\a}\sum_{l=0}^{\gamma}|k-l|\,G_{\alpha,\b,\gamma,\delta}^{k,l}\ 
   \leq\ 2C \bigl(|\a - \gamma|  + |\b-\delta|       \bigr).
\end{equation} 
\end{lemma}
\begin{proof} 
We  have 
\begin{eqnarray*}
\sum_{k=0}^{\a}\sum_{l=0}^{\g} |k-l|\gg &=& 
\sum_{k=0}^{\a}\sum_{l=0}^{\g}\1_{[l<k]} (k-l) \gg - \sum_{k=0}^{\a}\sum_{l=0}^{\g}\1_{[l\geq k]} (k-l) \gg    
\\
&=& \sum_{l=0}^{\g} l \left( \sum_{k=0}^{l} \1_{[k\leq\a]} \gg 
-  \1_{[l<\a]} \!\!\sum_{k=l+1}^{\a}\! \gg\right)  \\&&-
  \sum_{k=0}^{\a} k \left( \1_{[k\leq\g]}\sum_{l=k}^{\g} \gg 
  -  \sum_{l=0}^{k-1} \1_{[l\leq\g]}\gg\right)   
\\ 
&=& \sum_{l=1}^{(\a-1)\wedge\g} l \left( \sum_{k=0}^{l} \gg 
-  \!\sum_{k=l+1}^{\a}\! \gg\right) \ + \ \1_{[\a\leq\g]} \sum_{l=\a}^\g l g^l_{\g,\delta} 
\\ && \ -\   \sum_{k=1}^{\a\wedge\g} k \left(\sum_{l=k}^{\g} \gg -  \sum_{l=0}^{k-1}\gg\right) 
   +\ \1_{[\a>\g]} \sum_{k=\g+1}^\a k g^k_{\a,\b} \ \ 
\\
&=&  \sum_{i=1}^{\a\wedge\g} i R_i \ \ + \ \ \1_{[\a<\g]} 
\sum_{i=\a+1}^{\g} \! i R_i \ \ - \ \ \sum_{i=1}^{\a\wedge\g}  i R_i^*\ \ -      
     \1_{[\a>\g]}\sum_{i=\g+1}^{\a}\! i  R_i^*
\end{eqnarray*}
where we have denoted  
$$ 
R_l   =  \begin{cases}  \ \sum\limits_{k=0}^{l}  \gg -  \sum\limits_{k=l+1}^{\a} \gg, 
& \text{for } l=0,\dots,\a-1, 
         \\[1mm]        \      g^l_{\g,\delta}                               
         & \text{for } l=\a,\dots,\g \text{ if }\a\leq\g,
                             \end{cases}
$$
$$
\ \ \ R_k^*  =   \begin{cases} \ \sum\limits_{l=k}^{\g} \gg -  \sum\limits_{l=0}^{k-1}  \gg, 
& \ \ \text{for } k=0,\dots,\g,
           \\[1mm]       \     - g^k_{\a,\b}                                 
           & \ \ \text{for } k=\g+1,\dots,\a \text{ if }\a>\g.
\end{cases}
$$
Therefore using \eqref{rateg}
\begin{equation} \label{sum_1}
\sum_{k=0}^{\a}\sum_{l=0}^{\g} |k-l|G_{\a,\b,\g,\delta}^{k,l} = 
\sum_{i=1}^{\a\wedge\g} i (R_i  - R_i^* ) \ \ 
+  \1_{[\a<\g]}\sum_{i=(\a\wedge\g)+1}^{\a\vee\g} \! i\,(\gen^i_{\g,\delta} - \gen^i_{\a,\b})    \ +    
                                                 \1_{[\a>\g]}\sum_{i=(\a\wedge\g)+1}^{\a\vee \g}\! 
                                                 i\,(\gen^i_{\a,\b} - \gen^i_{\g,\delta}).
\end{equation}
\noindent Using notation $\sigb^i_{\a,\b}=\sum\limits_{k>i}\gkab$ and a telescopic argument 
\cite[Lemma 3.6]{GS} for partial sums of $G_{\a,\b,\g,\delta}^{k,l}$ gives, 
for  $\a,\b,\gamma,\delta\geq 0$ and $1\leq k\leq \a$, $1\leq l\leq \g$, 
\begin{eqnarray*}
\nonumber\gg &=& \gen^k_{\a,\b}\wedge \Bigl( \sigb^{l-1}_{\g,\delta}   
-  \sigb^{k}_{\a,\b}  \wedge   \sigb^{l-1}_{\g,\delta}      \Bigr)    - 
        \gen^k_{\a,\b}\wedge \Bigl( \sigb^{l}_{\g,\delta}   
        -  \sigb^{k}_{\a,\b}  \wedge   \sigb^{l}_{\g,\delta}      \Bigr)  
\\
\text{or} \makebox[1cm]{} && 
\\
\nonumber\gg  &=&  \gen^l_{\g,\delta}\wedge \Bigl( \sigb^{k-1}_{\a,\b}  
 -  \sigb^{k-1}_{\a,\b}  \wedge   \sigb^{l}_{\g,\delta}      \Bigr)    - 
          \gen^l_{\g,\delta}\wedge \Bigl( \sigb^{k}_{\a,\b}  
           -  \sigb^{k}_{\a,\b}  \wedge   \sigb^{l}_{\g,\delta}      \Bigr)  
\end{eqnarray*}
Then, for $i=1,\dots,\a\wedge\g$, we obtain  
\begin{eqnarray*}
R_i &=& \gen^i_{\g,\delta} - 2 \Bigl[ \gen^i_{\g,\delta} \wedge \Bigl(\sigb^i_{\a,\b}
 - \sigb^i_{\a,\b}\wedge\sigb^{i}_{\g,\delta}\Bigr)    \Bigr]    \\
&=& 
\left\{
\begin{array}{ll}
\gen^i_{\g,\delta}   & \text{ if }  \sigb^i_{\g,\delta} \geq \sigb^i_{\a,\b}    \\[.1cm]
\gen^i_{\g,\delta} - 2\bigl(\sigb^i_{\a,\b} - \sigb^{i}_{\g,\delta}\bigr) 
                     & \text{ if }  \sigb^i_{\g,\delta} < \sigb^i_{\a,\b} < \sigb^{i-1}_{\g,\delta} \\[.1cm]
-\gen^i_{\g,\delta}  & \text{ if }  \phantom{\sigb^i_{\g,\delta} < \ } \sigb^i_{\a,\b} 
\geq \sigb^{i-1}_{\g,\delta} 
\end{array}
\right.  
\end{eqnarray*}
\begin{eqnarray*}
R_i^* &=& 2 \Bigl[ \gen^i_{\a,\b} \wedge \Bigl(\sigb^{i-1}_{\g,\delta} 
- \sigb^i_{\a,\b}\wedge\sigb^{i-1}_{\g,\delta}\Bigr)    \Bigr] -\gen^i_{\a,\b}    \\
&=& 
\left\{
\begin{array}{ll}
- \gen^i_{\a,\b}   & \text{ if } \sigb^i_{\a,\b} \geq  \sigb^{i-1}_{\g,\delta}  \\[.1cm]
2\bigl(\sigb^{i-1}_{\g,\delta} - \sigb^i_{\a,\b}\bigr) - \gen^i_{\a,\b} 
                   & \text{ if } \sigb^i_{\a,\b} <  \sigb^{i-1}_{\g,\delta} < \sigb^{i-1}_{\a,\b} \\[.1cm]
\gen^i_{\a,\b}     & \text{ if } \phantom{\sigb^i_{\g,\delta} < \ } \sigb^{i-1}_{\g,\delta} 
\geq \sigb^{i-1}_{\a,\b}
\end{array}
\right.  
\end{eqnarray*}
Hence    
\begin{eqnarray*}
R_i - R_i^*= 
\left\{
\begin{array}{ll}
\gen^i_{\a,\b} - \gen^i_{\g,\delta} &     \text{ if }  \sigb^i_{\a,\b}  \geq \sigb^{i-1}_{\g,\delta} \\[.1cm]
\gen^i_{\g,\delta} - \gen^i_{\a,\b} &     \text{ if }  \sigb^i_{\a,\b}  \leq \sigb^{i}_{\g,\delta} \ \&\  
                                                       \sigb^{i-1}_{\a,\b}  
                                                       \leq \sigb^{i-1}_{\g,\delta}  \\[.1cm]
\gen^i_{\a,\b} - \gen^i_{\g,\delta} 
&     \text{ if }  \sigb^{i}_{\g,\delta}<\sigb^i_{\a,\b} 
< \sigb^{i-1}_{\g,\delta} < \sigb^{i-1}_{\a,\b}\\[.1cm]
\gen^i_{\g,\delta} - \gen^i_{\a,\b} - 2(\sigb_{\a,\b}^i - \sigb_{\g,\delta}^i)
                   &   \text{ if }  \sigb^{i}_{\g,\delta} <\sigb^i_{\a,\b} < \sigb^{i-1}_{\g,\delta} \ 
                   \&\  \sigb^{i-1}_{\a,\b} \leq \sigb^{i-1}_{\g,\delta} \\[.1cm] 
\gen^i_{\a,\b} - \gen^i_{\g,\delta} - 2(\sigb_{\g,\delta}^i 
- \sigb_{\a,\b}^i)   &     \text{ if }  \sigb^i_{\a,\b}
\leq  \sigb^{i}_{\g,\delta}\ \&\ \sigb^i_{\a,\b}<\sigb^{i-1}_{\g,\delta}
<\sigb^{i-1}_{\a,\b}.\\[.1cm]
\end{array}
\right.
\end{eqnarray*} 
Conditions in the last but one line imply that 
$0< \sigb^i_{\a,\b} - \sigb^{i}_{\g,\delta} \leq \gen^i_{\g,\delta} - \gen^i_{\a,\b} $
and similarly conditions in the last line imply that 
$0< \sigb^{i}_{\g,\delta}- \sigb^i_{\a,\b} \leq \gen^i_{\a,\b}-\gen^i_{\g,\delta} $.
Therefore we can conclude that 
$|R_i -R_i^* |\leq 2|\gen^i_{\a,\b} - \gen^i_{\g,\delta}| $.
{}From (\ref{sum_1}) and using assumption (\ref{assum}) then we obtain the desired inequality
\begin{equation}\label{key2} 
\sum_{k=0}^{\a}\sum_{l=0}^{\g} |k-l|\,G_{\a,\b,\g,\delta}^{k,l} \,\leq\,  
2\sum_{i=1}^{\a\vee\g} i\, |\gen^i_{\a,\b} - \gen^i_{\g,\delta}| \, \leq \,2C\left(|\a-\g|+|\b-\delta|\right).
\end{equation} 
\end{proof} 
%
\subsection{Finite volume}   \label{ex1} 
Because $\X_n$ is finite,  for $\eta,\eta'\in\N^{\X_n}$, if 
$N=\sum_{x\in\X_n}\eta(x),N'=\sum_{x\in\X_n}\eta'(x)$, then  the Markov process $(\eta_t)_{t\geq 0}$ 
starting from $\eta_0=\eta$ (resp. the coupled process $(\eta_t,\eta'_t)_{t\geq 0}$ 
starting from $(\eta_0,\eta'_0)=(\eta,\eta')$)
has the finite state space $\mathrm{S}=\{\xi\in\N^{\X_n}:\sum_{x\in\X_n}\xi(x)=N\}$
(resp. $\overline{\mathrm{S}}=\{(\xi,\xi')\in\N^{\X_n}\times\N^{\X_n}:\sum_{x\in\X_n}\xi(x)=N,
\sum_{x\in\X_n}\xi'(x)=N'\}$). Hence we can write  the semi-group as an exponential
of a $Q$-matrix, that is 
\begin{eqnarray}
S_n(t)f(\eta)&=&\mathbb{E}^{\eta}(f(\eta_t))
=\!\!\sum_{\zeta\in\mathrm{S}} f(\zeta) \mathbf{p}_t(\eta,\zeta)\label{S_n-fini}
\\
\overline{S}_n(t)h(\eta,\eta')&=&\mathbb{E}^{(\eta,\eta')}(h(\eta_t,\eta'_t))
=\!\!\sum_{(\zeta,\zeta')\in\overline{\mathrm{S}}} h(\zeta,\zeta') 
\overline{\mathbf{p}}_t(\eta,\eta';\zeta,\zeta'),\quad\quad\mbox{ where }\label{cS_n-fini}
\\\label{pt-rate}
\mathbf{p}_t(\eta,\zeta)&=& \sum_{j=0}^{\infty} \frac{t^j}{j!}q^{(j)}(\eta,\zeta)\\\label{q-rate}
q(\eta,\zeta)&=&\sum_{x,y\in\X_n}\sum_{k=1}^{\eta(x)} p_n(x,y)\, 
\gen^k_{\eta(x),\eta(y)} \left( \1_{[\zeta=\etb]} -\1_{[\zeta=\eta]} \right)\\\label{pt-rate-coupl}
\overline{\mathbf{p}}_t(\eta,\eta';\zeta,\zeta')
&=& \sum_{j=0}^{\infty} \frac{t^j}{j!}\overline{q}^{(j)}(\eta,\eta';\zeta,\zeta')\\\nonumber
\overline{q}(\eta,\eta';\zeta,\zeta')&=&\sum_{x,y\in\X_n}\sum_{k=1}^{\eta(x)}
\sum_{l=1}^{\eta'(x)} p_n(x,y)\, G^{k,l}_{\eta(x),\eta(y),\eta'(x),\eta'(y)} \\
&&\qquad\qquad\qquad\times\left( \1_{[(\zeta,\zeta')=(\etb,\etb')]} 
-\1_{[(\zeta,\zeta')=(\eta,\eta')]} \right).\label{q-rate-coupl}
\end{eqnarray} 
Recall the constants $m_p,C,M$  introduced in 
\eqref{mpconst}, \eqref{Mconst}, \eqref{un-assum} and \eqref{assum}. 
\begin{lemma}\label{bound1}  
Assume that the rates 
satisfy \eqref{un-assum}. We have, for all $f\in\Lip$, $\eta\in\N^{\X_n},\,n\geq 1$,
$$  |\L_n f(\eta) | \leq\,  L_f C\, (1+m_p+M)\,\| \eta \|. $$  
\end{lemma}
\begin{proof}   
We have 
\begin{eqnarray}\nonumber
|{\L}_{n} f(\eta) | \!&\!\leq \!&\! L_f \sum_{x,y\in\X_n} 
 p(x,y) \sum_{k=1}^{\eta(x)} g^k_{\eta(x),\eta(y)} \|\etb-\eta \|  \\ \nonumber
 &=& L_f \sum_{x,y\in\X_n} p(x,y) \sum_{k=1}^{\eta(x)} kg^k_{\eta(x),\eta(y)} (a_x + a_y) \\ \label{eq:bound1}
&\leq& L_f\,C\, \sum_{x, y\in\X_n} p(x,y) \left(\eta(x)+\eta(y)\right)(a_x + a_y) \\ \nonumber
&=& L_f\,C\,   \sum_{x\in\X_n} \eta(x)  \sum_{y\in\X_n}  \Bigl( p(x,y) + p(y,x)   \Bigr)  (a_x+a_y) \\
&\leq& L_f\,C\, (1+m_p+M) \sum_{x\in\X_n}  a_x  \eta(x) \nonumber
\end{eqnarray} 
where  $f\in\Lip$ implies the first inequality,
the second one  comes from (\ref{un-assum}), and the third one from  \eqref{mpconst}, \eqref{Mconst}.
\end{proof}
\begin{lemma}\label{bound2} 
Assume that the rates 
satisfy \eqref{assum}. We have, for all $\eta,\eta'\in\N^{\X_n},\,n\geq 1$, 
$$  \Big|\overline{\L}_{n} \Bigl( \|\eta-\eta'\|\Bigr) \Big|\leq \, 2C\, (1+m_p+M)\,\| \eta -\eta'\|. $$   
\end{lemma}
\begin{proof}
We have
\begin{eqnarray*}
\Big|\overline{\L}_{n} \Bigl( \|\eta-\eta'\|\Bigr)\Big|  &= &
\Big| \sum_{x,y\in \X_n} p(x,y) \sum_{k=1}^{\eta(x)}\sum_{l=1}^{\eta'(x)} 
G^{k,l}_{\eta(x),\eta(y),\eta'(x),\eta'(y)}\Bigl( \|\etb-{\mathcal{S}}^l_{x,y}\eta'\| 
- \|\eta-\eta' \| \Bigr)\Big|   
\\
&\leq& \sum_{x, y\in\X_n}p(x,y)  \sum_{k=1}^{\eta(x)}
\sum_{l=1}^{\eta'(x)}  G^{k,l}_{\eta(x),\eta(y),\eta'(x),\eta'(y)} (a_x+a_y) |k-l| \\ 
&\leq& 2C\,\sum_{x, y\in\X_n} p(x,y) (a_x+a_y)  \left(|\eta(x)-\eta'(x)|+|\eta(y)-\eta'(y)|\right)  
\\
&\leq& 2C\sum_{x\in \X_n} |\eta(x)\!-\!\eta'(x)|  \sum_{y\in \X_n} (p(x,y) + p(y,x)) (a_y+a_x)  
\end{eqnarray*}
where the second inequality comes from Lemma \ref{lem:key}.
 This implies the result by \eqref{mpconst}, \eqref{Mconst}. 
\end{proof}
Based on  Lemmas \ref{bound1} and \ref{bound2}, the following technical lemmas 
(corresponding to \cite[Lemmas 2.1--2.4]{andjel}) can be proved for the MMP
$(\eta_t)_{t\geq 0}$ on $\X_n$ finite. We rely on \eqref{S_n-fini}--\eqref{q-rate-coupl}.
\begin{lemma}     \label{an2.1}
Fix $\eta\in\N^{\X_n}$. Assume that the rates 
 satisfy \eqref{un-assum}. We have for the Markov process $(\eta_t)_{t\geq 0}$ 
starting from $\eta_0=\eta$, and for $y\in\X_n$, 
$$
\mathbb{E}^{\eta} \Bigl( \eta_t(y)\Bigr) \leq e^{Cm_pt}\,\sum_{x\in\X_n} \eta(x) \sum_{l=0}^{\infty}\frac{(Ct)^l}{l!}p^{(l)}(x,y).
$$
\end{lemma}
\begin{proof} 
Fix $\eta\in\N^{\X_n}$, $y\in\X_n$. Since by (\ref{un-assum}),   \eqref{mpconst},
\begin{eqnarray*}
\sum_{\zeta\in\N^{\X_n}} \zeta(y) q(\eta,\zeta) &=& 
\sum_{x,z\in\X_n}\sum_{k=1}^{\eta(x)} p(x,z) \gen^k_{\eta(x),\eta(z)} ({\mathcal{S}}^k_{x,z}\eta(y) -\eta(y)) \\
&\leq&   \sum_{x\in\X_n}\sum_{k=1}^{\eta(x)} p(x,y) \gen^k_{\eta(x),\eta(y)} k\leq
C    \sum_{x\in\X_n}  \eta(x) p(x,y) \,+\,Cm_p\eta(y)
\end{eqnarray*}
and $q^{(j)}(\eta,\zeta)=\sum_{\xi\in\N^{\X_n}}q^{(j-1)}(\eta,\xi)q(\xi,\zeta)$ 
an induction proof gives that for $j\ge 1$,
\begin{equation}\label{eq:maj-rec-q}
\sum_{\zeta\in\N^{\X_n}} \zeta(y) q^{(j)}(\eta,\zeta) 
\leq C^j \sum_{l=0}^j \binom{j}{l}m_p^{j-l} \sum_{z\in\X_n} \eta(z)p^{(l)}(z,y).  
\end{equation}
We conclude  by \eqref{pt-rate}  that 
$$ \mathbb{E}^{\eta} \Bigl( \eta_t(y)\Bigr)
=\sum_{\zeta\in\N^{\X_n}}\zeta(y)\mathbf{p}_t(\eta,\zeta)\leq  
\sum_{j=0}^{\infty} \frac{t^j}{j!} C^j \sum_{l=0}^j \binom{j}{l}m_p^{j-l} 
\sum_{x\in \X_n}\eta(x)p^{(l)}(x,y)        $$
and exchanging summations gives the result. 
\end{proof}
\medskip
\begin{lemma}     \label{an2.2} 
Let $f\in\Lip$. Assume that the rates 
 satisfy \eqref{assum}. Then
\begin{eqnarray*}
&& (a) \qquad S_n(t)f\in\Lip \ \text{ and } L_{S_n(t)f}\leq L_fe^{2C(M+m_p+1)t},  \qquad \qquad
\\
&& (b) \qquad \text{if } |f(\eta)|\leq c_0\|\eta\|\ \text{ then } |S_n(t)f(\eta)|\leq c_0 e^{C(M+m_p)t}\|\eta\|.  \qquad \qquad
\end{eqnarray*}
\end{lemma}
\begin{proof}   
It is based on the following consequences of Lemmas \ref{bound1} and \ref{bound2}:  
For every $\eta,\eta' \in \N^{\X_n}$, $j\in\N$,
\begin{eqnarray}\label{prel-7.5-a}
\left|\sum\limits_{\zeta\in\mathrm{S}} q^{(j)} (\eta,\zeta) \|\zeta\| \right|
&\leq&\, \left(C\, (M+m_p+1)\right)^j \,\| \eta \|, \\ \label{prel-7.5-b}
\left|\sum\limits_{\zeta,\zeta'\in\overline{\mathrm{S}}} 
\overline{q}^{(j)}(\eta,\eta',\zeta,\zeta')  \|\zeta-\zeta'\| \right|
&\leq& \, (2C\, (M+m_p+1))^j\,
\| \eta -\eta'\|. 
\end{eqnarray}
 Let us derive for instance \eqref{prel-7.5-a}
  (it is similar for \eqref{prel-7.5-b}), by induction on $j$. For $j=1$, 
  using \eqref{ngenerator}, \eqref{q-rate} and Lemma \ref{bound1} 
  for the function $f(\eta)=\| \eta \|$, which belongs to $\Lip$ with $L_f=1$, we have
  $$\left|\sum\limits_{\zeta\in\mathrm{S}} q (\eta,\zeta) \|\zeta\| \right|
  = \left|\L_n f(\eta) \right| \leq\,   C\, (1+m_p+M)\,\| \eta \|.$$
  Then for $j\ge 1$, using the computation for $j=1$ we have
  $$\left|\sum\limits_{\zeta\in\mathrm{S}} q^{(j+1)} (\eta,\zeta) \|\zeta\| \right|
  =\left|\sum\limits_{\xi\in\mathrm{S}} q^{(j)} (\eta,\xi) 
  \sum\limits_{\zeta\in\mathrm{S}} q (\xi,\zeta)\|\zeta\|\right|
  \leq\,   C\, (1+m_p+M)\sum\limits_{\xi\in\mathrm{S}} q^{(j)} (\eta,\xi) \| \xi \|.
   $$
(a) For $f\in\Lip$,
we define $g(\eta,\eta')= f(\eta)-f(\eta')$ and $h(\eta,\eta')=\|\eta-\eta'\|$ for all $\eta,\eta'$. 
Then by \eqref{prel-7.5-b},
\begin{eqnarray*} \nonumber
&& |S_n(t)f(\eta) - S_n(t)f(\eta')|=   |\overline{S}_n(t)g(\eta,\eta')|     
  \leq L_f|\overline{S}_n(t)h(\eta,\eta')|
  \\
 && 
 = L_f \left|\sum_{j=0}^{\infty} \frac{t^j}{j!} \sum_{\zeta,\zeta'\in\overline{\mathrm{S}}} \overline{q}^{(j)}(\eta,\eta',\zeta,\zeta') h(\zeta,\zeta')\right|
 \leq L_f e^{2C\, (M+m_p+1)t} \| \eta -\eta'\| . \label{Sn-for-invt}
\end{eqnarray*}
(b) For $f\in\Lip$, by \eqref{pt-rate}, \eqref{prel-7.5-a},
 $$ |S_n(t)f(\eta)|\leq  c_0\sum\limits_{\zeta\in\mathrm{S}} \mathbf{p}_t(\eta,\zeta) \|\zeta\|
= c_0\left|\sum\limits_{j=0}^{\infty} \frac{t^j}{j!} \sum\limits_{\zeta\in\mathrm{S}} 
q^{(j)}(\eta,\zeta) \|\zeta\|\right| \leq c_0 e^{C\, (M+m_p+1)t}\|\eta\|.$$                     
\end{proof}

\smallskip

\begin{lemma}     \label{an2.4}
Let $m<n$ and $f\in\Lip$. Assume that the rates 
 satisfy \eqref{un-assum}. Then
\begin{equation}\label{eq:an2.4}
| (\L_n-\L_m) f(\eta)  |\leq 
C L_f  \sum_{\substack{x,y\in\X \\ x\neq y}}|p_n(x,y)-p_m(x,y)| (\eta(x)+\eta(y)) (a_x+a_y) .
\end{equation}
\end{lemma}
\begin{proof} We have, using the proof of Lemma \ref{bound1},
\begin{eqnarray*} 
&&| (\L_n-\L_m) f(\eta)  |\leq \sum_{\substack{x,y\in\X \\ 
x\neq y}}\sum_{k=1}^{\eta(x)} \gen^k_{\eta(x),\eta(y)}  |p_n(x,y)-p_m(x,y)|
L_f \|\etb-\eta\|\\
&& \leq L_f\sum_{\substack{x,y\in\X \\ x\neq y}}|p_n(x,y)-p_m(x,y)|
\sum_{k=1}^{\eta(x)} k\gen^k_{\eta(x),\eta(y)}(a_x+a_y)
\end{eqnarray*}
 and we conclude by \eqref{un-assum} and the proof of Lemma \ref{bound1}. 
\end{proof}
\smallskip
\begin{corollary}\label{coro:St}
Let $f\in\Lip$,  $\eta\in\XX$ and $t\geq 0$. Assume that the rates 
 satisfy \eqref{assum}. Then 
\begin{equation}\label{aux2}
    |S_n(t)f(\eta) - S_m(t)f(\eta)|     \leq CL_f e^{2Cm_pt} \int_0^t H_{n,m}(\eta)\ ds
\end{equation}
where 
\begin{eqnarray}\nonumber 
H_{n,m}(\eta) &=&  e^{2C(M+1)(t-s)}  
\sum_{\ z\in\X_n} \eta(z)\sum_{l=0}^{\infty} \frac{(Cs)^l}{l!}\\ 
&&\qquad\qquad\times 
\sum_{\substack{x,y\in\X \\ x\neq y}}\bigl(p_n^{(l)}(z,x) + p_n^{(l)}(z,y) \bigr)
(a_x+a_y)     |p_n(x,y)-p_m(x,y)|.\label{def_H}
\end{eqnarray} 
\end{corollary}
\begin{proof}
The result follows from the successive use of the  integration by parts formula 
$$  
 S_n(t)f(\eta) - S_m(t)f(\eta) = \int_0^t S_n(s) (\L_n-\L_m) S_m(t-s) f(\eta) ds
$$
and  Lemmas  \ref{an2.4}, \ref{an2.2}(a) and \ref{an2.1}.
\end{proof}
\smallskip

\subsection{Infinite volume}    \label{ex4} 
Through this paragraph  $\X$ is infinite. 
We consider its approximation introduced in \eqref{seqX_n}. 
Recall notation $p_n$, $\L_n$ and $S_n(t)$.

We obtained in Subsection \ref{ex1} some properties of finite volume processes with generator $\L_n$ 
and semigroup $S_n(t),\, t\geq 0$ that we shall now use to take the limit  $n$ to infinity. 
\begin{lemma} \label{L-n-conv}
Assume that the rates 
 satisfy \eqref{un-assum}. For all $f\in\Lip$, $\eta\in\XX$,
 $\L_n f(\eta) \longrightarrow \L f(\eta) $ when $n\rightarrow\infty$.
 \end{lemma}
\begin{proof}  We use Lemma \ref{an2.4} and the proof of Lemma \ref{bound1} to write
\begin{eqnarray*} 
 | (\L_n-\L_m) f(\eta)  |
 \leq C L_f  \sum_{x,y\in \X} p(x,y)  (\eta(x)+\eta(y)) (a_x+a_y) \leq C L_f (1+m_p+M) \|\eta\|
\end{eqnarray*}
and each term on the right-hand side of \eqref{eq:an2.4} goes to 0 when $n,m\rightarrow\infty$,
 we have a Cauchy sequence that converges using the dominated convergence theorem, and 
the fact that $p_n(x,y)\rightarrow p(x,y)$ when $n\rightarrow\infty$  yields the limit.
\end{proof}
\smallskip
\begin{proof} {\it of Proposition \ref{prop:generator}.}\quad 
It is a consequence of 
 Lemmas \ref{L-n-conv} and \ref{bound1}, which also imply  \eqref{bound-to-generator}.
\end{proof}
 \smallskip 
 \begin{lemma}     \label{an2.6}
  Assume that the rates 
 satisfy \eqref{assum}. 
 For all $f\in\Lip$, $\eta\in\XX$ and $t\geq 0$, $S_n(t)f(\eta)$ 
converges when $n\rightarrow\infty$.
\end{lemma}

\begin{proof}     Let us fix $t>0$ and $\eta\in\XX$.                  
 From Corollary \ref{coro:St} we have 
 that $H_{n,m}(\eta)$ given by \eqref{def_H} satisfies the upper bound
\begin{eqnarray} 
H_{n,m}(\eta)
 &\leq& 2 (M+2) e^{2C(M+1)t} \|\eta\|.
 \label{bound-Hnm}
\end{eqnarray} 
Indeed, on the one hand, by induction on $l$ using \eqref{Mconst}
we get, for $z\in\X_n$,
\begin{equation}\label{Mconst-l}
\sum_{x\in \X_n} a_x p_n^{(l)}(z,x)\leq a_z(M+1)^l.
\end{equation}
On the other hand, for $x,y\in X_n$ we have
\begin{equation}\label{eq:pn-pm}
|p_n(x,y)-p_m(x,y)|=p(x,y)(\1_{[x,y\in X_n\setminus X_m]}
+\1_{[x\in X_n\setminus X_m,\,y\in X_m]}
+\1_{[y\in X_n\setminus X_m,\,x\in X_m]}).
\end{equation}
 Hence using  \eqref{eq:pn-pm}, \eqref{Mconst} and \eqref{Mconst-l}, 
  we bound the second line of \eqref{def_H} by 
 \begin{eqnarray*} 
&&\sum_{x\in X_n\setminus X_m} a_x p_n^{(l)}(z,x) \sum_{\substack{y\in X_n\\ y\neq x}} p(x,y)
+ \sum_{x\in  X_m} a_x p_n^{(l)}(z,x) \sum_{y\in X_n\setminus X_m} p(x,y)\\
&& +\sum_{x\in X_n\setminus X_m} p_n^{(l)}(z,x) \sum_{\substack{y\in X_n\\ y\neq x}} a_y p(x,y)
+ \sum_{x\in X_m}  p_n^{(l)}(z,x) \sum_{y\in X_n\setminus X_m} a_y p(x,y)\\
&& +\sum_{y\in X_n\setminus X_m} p_n^{(l)}(z,y) \sum_{\substack{x\in X_n\\ x\neq y}} a_x p(x,y)
+\sum_{y\in X_m} p_n^{(l)}(z,y) \sum_{x\in X_n\setminus X_m} a_x  p(x,y)\\
&& + \sum_{y\in X_n\setminus X_m} a_y p_n^{(l)}(z,y) \sum_{\substack{x\in X_n\\ x\neq y}} p(x,y)
+\sum_{y\in X_m} a_y  p_n^{(l)}(z,y) \sum_{x\in X_n\setminus X_m} p(x,y)\\
\leq &&  \sum_{x\in X_n} a_x p_n^{(l)}(z,x) + \sum_{x\in X_n} p_n^{(l)}(z,x) (M+1) a_x \\
&& + \sum_{y\in X_n} p_n^{(l)}(z,y) (M+1) a_y +  \sum_{y\in X_n} a_y p_n^{(l)}(z,y)\\
\leq && 2 (M+2) \sum_{x\in X_n} a_x p_n^{(l)}(z,x)\leq 2 (M+2)a_z(M+1)^l.
\end{eqnarray*}
Combining \eqref{def_H} and \eqref{aux2},
by the dominated convergence theorem we obtain that $(S_n(t)f(\eta))_{n\ge 0}$ 
is a Cauchy sequence. 
\end{proof}
\smallskip

\begin{definition}\label{sem_def}
For all $f\in\Lip$, $\eta\in\XX$ and $t\geq 0$ we define 
\begin{equation}\label{Sdef}
    S(t)f(\eta) = \lim_{n\rightarrow\infty} S_n(t)f (\eta).
\end{equation}
\end{definition}
The following lemma shows that  $\left( S(t) \, :\, t\geq 0\right)$ thus defined
is a semigroup with infinitesimal generator $\L$ given by (\ref{generator}).  
\begin{lemma} \label{lemma-spitzer}  
 Let $f\in\Lip$, $t,t_1,t_2\geq 0$, $\eta\in\XX$    \\[.2cm]
(i) $S(t_1+t_2)=S(t_1)S(t_2),\, S(0)=I$;
\\
(ii) $S(t) f(\eta)=f(\eta) + \int\limits_0^{t} \L S(s) f(\eta) \d s $;  
\\
(iii) $|S(t) f(\eta)-f(\eta)|\leq \|\eta\|L_f (e^{2C(M+m_p+1)t}-1)$;
 \\[.2cm]
(iv) $S(t)f(\eta)$ is continuous in $t$;  
\\[.2cm]
(v) $\L S(t)f(\eta)$ is continuous in $t$;  
\\[.2cm]
(vi) $\displaystyle{\lim_{t\searrow 0}\frac{S(t)f(\eta) - f(\eta)}{t} = \L f(\eta)}$;
\\[.2cm]
(vii) $\L S(t)f(\eta) = S(t)\L f(\eta)$.
\end{lemma}
 \begin{proof}
The properties given in the previous Lemmas enable to use exactly the same arguments as in the paper
\cite{liggett-spitzer} to derive \cite[Lemma 2.16]{liggett-spitzer}. 
Hence we refer to it, and just note that for \textit{(iii)}, using Lemmas \ref{bound1} and \ref{an2.2}\textit{(a)} we have:
$$ |\L_n S_n(t) f(\eta) |\leq  L_{S_nf} C (1+m_p+M)  \|\eta\| \leq L_{f} e^{2C(M+m_p+1)t} C (1+m_p+M) \|\eta\|.   $$
\end{proof} 
\mbox{}
Now we are ready to prove Theorem \ref{ex-theorem} and Proposition \ref{intL=0}.
\begin{proof}\textit{of Theorem \ref{ex-theorem}.} \ \
We have that \textit{(1)} follows from Lemma  \ref{an2.2}\textit{(a)} and the definition \eqref{Sdef}
of $S(t)$, that \textit{(2)} follows from Lemma \ref{lemma-spitzer}\textit{(ii)}, and that
\textit{(3)} is again a consequence of definition \eqref{Sdef}.
\end{proof}

\begin{proof}\textit{of Proposition \ref{intL=0}.}  
\mbox{}\\ 
\noindent $\bullet$
Properties \textit{(i)--(vii)} of Lemma \ref{lemma-spitzer} 
and the same arguments as in  \cite[Lemma 2.17]{liggett-spitzer} 
(see also \cite[Lemma 2.9]{andjel}) yield that $\bar{\mu}$ is invariant if and only if 
\begin{equation} \label{bound-intL=0} 
\int \L f \ d\bar{\mu} =0 \ \text{ for every bounded } f\in\Lip.
\end{equation}
 Namely we write that, for a bounded $f\in\Lip$,
\begin{eqnarray}\label{for-intL=0} 
\int (S(t) f(\eta)-f(\eta))\ d\bar{\mu}(\eta)
=\int\Bigl(\int\limits_0^{t} \L S(s) f(\eta) \d s\Bigr)\ d\bar{\mu}(\eta)
=\int\limits_0^{t}\int\Bigl( \L S(s) f(\eta)\ d\bar{\mu}(\eta)\Bigr) \d s=0
\end{eqnarray}
where  \eqref{bound-to-generator} enables to exchange limits.\par 
\smallskip\noindent $\bullet$
Now consider an arbitrary $f\in\Lip$ and assume \eqref{bound-intL=0} holds. 
The sequence of bounded Lipschitz functions defined by, for $\eta\in\XX$,
$$
f_m(\eta) = \left\{\begin{array}{ll}
              f(\eta) & \text{if } |f(\eta)|\leq m \\
              m       & \text{if } f(\eta)  > m \\
              -m      & \text{if } f(\eta)  < -m \\
              \end{array}
             \right. 
$$
 satisfies, for every $\eta\in\XX$, $\lim_{m\rightarrow +\infty} f_m(\eta) = f(\eta) $, as well as 
 $\lim_{m\rightarrow +\infty} \L f_m(\eta) = \L f(\eta) $. Indeed, 
 the Lipschitz constant of $f_m$ is $L_f$ for every $m$, and we have
$$|\L f_m(\eta) -\L f(\eta) |\leq \sum_{x,y\in\X}p(x,y) 
\sum_{k=1}^{\eta(x)} g^k_{\eta(x),\eta(y)} |f_m(\etb)-f(\etb) + f(\eta) -f_m(\eta)| $$
which is dominated (using \eqref{un-assum})
by $2CL_f \sum_{x,y\in\X}p(x,y) (a_x + a_y) (\eta(x) + \eta(y) )$, which is finite
(cf. the proof of Lemma \ref{bound1}).
Because of \eqref{bound-to-generator}, we know that $|\L f_m(\eta) | $ 
is dominated by $L_f\|\eta\|$ (up to a multiplicative constant)
which is $\bar{\mu}$-integrable, we can use the dominated convergence theorem and conclude that
$$\int \L f(\eta) d\bar{\mu}(\eta) =\lim_{m\rightarrow \infty} \int \L f_m(\eta) d\bar{\mu}(\eta) =0. $$
We have therefore proved that $\bar{\mu}$ is invariant if and only if 
\begin{equation}\label{necc-intL=0}
\int \L f \ d\bar{\mu} =0    \ \text{ for every }  f\in \Lip.
\end{equation}
\noindent $\bullet$
Now fix a bounded $f\in\Lip$ and assume that $\int \L g \ d\bar{\mu} =0 $ 
for every bounded cylinder function~$g$. Given the sequence \eqref{seqX_n} of finite sets,
one can define a sequence of bounded cylinder functions
$$
f_n(\eta) = f(\eta\!\! \upharpoonright_n) \quad \text{ where }\  \eta\!\! \upharpoonright_n (x) = 
             \left\{\begin{array}{ll}
              \eta(x) & \text{if } x\in \X_n \\
              0       & \text{if } x\not\in \X_n 
              \end{array}
             \right. 
$$
that satisfy $\lim_{n\rightarrow \infty} f_n(\eta) = f(\eta) $ for every $\eta$. 
Using the same arguments as before we obtain that
$$\int \L f(\eta) d\bar{\mu}(\eta) = \int \lim_{n\rightarrow \infty}\L f_n(\eta) d\bar{\mu}(\eta) 
=  \lim_{n\rightarrow \infty} \int \L f_n(\eta) d\bar{\mu} (\eta)=0. $$
Hence the proposition is proved.
\end{proof}

\section*{Acknowledgements}   We thank the anonymous referee for his/her extremely careful
reading of our manuscript, which helped us to improve it. We are grateful to Stefan Grosskinsky
for useful discussions.  This work was supported by the 
French Ministry of Education through the grants ANR-07-BLAN-0230 and ANR-2010-BLAN-0108, 
by the Grant Agency of the Czech Republic under Grant No.~201/12/2613.  
We thank MAP5 at Universit\'e Paris Descartes and \'UTIA in Prague for hospitality and financial support. 

\end{document}